\documentclass[12pt]{amsart}
\usepackage{amsfonts}
\usepackage{mathrsfs}
\usepackage{bbm}%
\usepackage{bm}
\usepackage{bbding}
\usepackage[mathcal]{euscript}
\usepackage{graphicx}
\usepackage{amsthm,amscd}
\usepackage{calc}
\usepackage{mathtools}
\usepackage{mathrsfs,dsfont}
\usepackage{CJK,fancyhdr,amscd}
\usepackage[x11names]{xcolor}
\usepackage{anysize}
\usepackage{amsmath,amssymb,amsfonts}
\usepackage{epsfig,enumerate}
\usepackage{latexsym}
\usepackage{indentfirst, latexsym}
\usepackage{graphics}
\usepackage{tikz-cd}
\usepackage{lineno}
\usepackage[all,poly,knot]{xy}
\usepackage{xpatch}
\usepackage{hyperref}

\usepackage[T1]{fontenc}
\usepackage{textcomp}
\usepackage[utf8]{inputenc}

\hypersetup{
	colorlinks=true,
	linkcolor=Magenta4,
	citecolor=Blue2,
	urlcolor=Chartreuse3
}

\xpatchcmd{\step}{%
	\normalfont\scshape\centering}{%
	\normalfont\scshape}{\typeout{Success}}{\typeout{Failure}}%

\allowdisplaybreaks[2]

\providecommand{\U}[1]{\protect\rule{.1in}{.1in}}

\numberwithin{equation}{section}

\newtheorem{theorem} {Theorem} [section]
\newtheorem{proposition}[theorem]{Proposition}
\newtheorem{corollary}  [theorem]     {Corollary}
\newtheorem{lemma}  [theorem]     {Lemma}
\theoremstyle{definition}
\newtheorem{example}  [theorem]     {Example}
\newtheorem{question}  [theorem]     {Question}

\newtheorem{step}{Step}
\theoremstyle{definition}

\newtheorem{definition}  [theorem]     {Definition}

\newtheorem{Construction}  [theorem]     {Construction}

\theoremstyle{definition}
\newtheorem{remark}  [theorem]     {Remark}

\newcommand{\dz}{d \bar{z}}

\newcommand{\w}{\wedge}
\newcommand{\im}{\mathrm{im}}
\renewcommand{\1}{\mathsf{1}}
\newcommand{\G}{\mathbb{G}}
\newcommand{\db}{\overline{\partial}}
\newcommand{\dbs}{\overline{\partial}^*}

\newcommand{\lc}{\lrcorner}

\newcommand{\lk}{\left(}
\newcommand{\rk}{\right)}

\newcommand{\Om}{\Omega}
\newcommand{\btheorem}{\begin{theorem}}
	\newcommand{\etheorem}{\end{theorem}}
\newcommand{\bproposition}{\begin{proposition}}
	\newcommand{\eproposition}{\end{proposition}}
\newcommand{\bdefinition}{\begin{definition}}
	\newcommand{\edefinition}{\end{definition}}
\newcommand{\bcorollary}{\begin{corollary}}
	\newcommand{\ecorollary}{\end{corollary}}
\newcommand{\bproof}{\begin{proof}}
	\newcommand{\eproof}{\end{proof}}
\newcommand{\beq}{\begin{equation}}
	\newcommand{\eeq}{\end{equation}}
\newcommand{\ee}{\end{eqnarray*}}
\newcommand{\be}{\begin{eqnarray*}}

\newcommand{\elemma}{\end{lemma}}
\newcommand{\blemma}{\begin{lemma}}

\renewcommand{\>}{\rightarrow}
\newcommand{\sL}{{\mathcal L}}
\newcommand{\p}{\partial}
\newcommand{\bp}{\overline\partial}
\renewcommand{\H}{\mathbb H}

\newcommand{\bd}{\begin{enumerate} }
	\newcommand{\ed}{\end{enumerate}}

\def\p{\partial}

\def\b{\bar}
\def\mb{\mathbb}
\def\mc{\mathcal}

\def\v{\varphi}
\def\w{\wedge}

\def\l{\lrcorner}
\newtheorem*{Conventionno}{Convention}
\theoremstyle{plain} 
\newtheorem{bigthm}{Theorem}

\theoremstyle{plain} 
\newtheorem{bigque}{Question}

\theoremstyle{plain} 
\newtheorem{bigcor}{Corollary}

\begin{document}
	\title{On extension of closed complex (basic) differential forms: (basic) Hodge numbers and (transversely) $p$-K\"ahler structures}
	\let\oldmaketitle\maketitle
	\renewcommand\maketitle{{\bfseries\boldmath\oldmaketitle}}	
	
	\begin{abstract}
		Inspired by a recent work of D. Wei--S. Zhu on the extension of  closed complex differential forms and Voisin's usage of the $\p\bp$-lemma, we obtain several  new  theorems of deformation invariance of Hodge numbers and reprove the local stabilities of $p$-K\"ahler structures with the $\p\bp$-property. Our approach is more concerned with the $d$-closed extension by means of the  exponential operator $e^{\iota_\v}$.
		
		 Furthermore, we prove the local stabilities of transversely $p$-K\"ahler structures with mild $\p\bp$-property by adapting the power series method to the foliated case, which strengthens the works of A. El Kacimi Alaoui--B. Gmira and P. Ra\'zny on that of the transversely K\"ahler foliations with homologically orientability. We observe that a transversely K\"ahler foliation, even without homologically orientability, also satisfies the $\p\bp$-property. So even when $p=1$ (transversely K\"ahler), our results are new as we can drop the  assumption in question on the initial foliation. Several theorems  on the deformation invariance of basic Hodge/Bott--Chern numbers with mild $\p\bp$-properties are also presented.
	\end{abstract}
	
	\author{Sheng Rao}
	\address{School of Mathematics and statistics, Wuhan  University,
		Wuhan 430072, China}
	\email{likeanyone@whu.edu.cn}
	\thanks{Both authors are partially supported by NSFC (Grant No. 12271412, W2441003) and Hubei Provincial Innovation Research Group Project (Grant No. 2025AFA044).}
	\dedicatory{In Memory of Jean-Pierre Demailly (1957-2022)}
	
	\author{Runze Zhang}
	\address{School of Mathematics and statistics, Wuhan  University,
		Wuhan 430072, China; Universit\'e C\^ote d'Azur, CNRS, Laboratoire J.-A. Dieudonn\'e, Parc Valrose, F-06108 Nice Cedex 2, France}
	\email{runze.zhang@unice.fr}

	\date{\today}
	
	
	\maketitle
	
	\setcounter{tocdepth}{1}
	\tableofcontents
	
	\section{Introduction}\label{Introduction}
	The classical deformation theory of compact complex manifolds, developed by Kodaira--Nirenberg--Spencer and Kuranishi \cite{KNS58,KS60,Ku64,Ko86}, intensively studies the complex structures `close to' a given one. Subsequently, the theory had been  extended to the case of complex pseudogroup structures such as \cite{Ko60,KS61}. In particular, many excellent works concerning the deformations  of transversely holomorphic foliations and holomorphic foliations  appear, see \cite{DK79,DK80,GM81,GHS83,EKA88,EKAN89,GN89,Gb92,EKAG97,BHTF21,Ra21}, etc.  Most of these works essentially dealt with the extension of closed complex (basic) differential forms.
	
	As an important and direct application, one considers the deformation invariance of Hodge numbers. It becomes a useful tool in the study of the deformation limit problems. For instance, let $\pi:\mathcal{X}\rightarrow \Delta$ be a holomorphic family over an open disk in $\mb C$. 
    D. Popovici \cite{Po13} (resp. D. Barlet \cite{Ba15}) proved that if all fibers $X_{t}$ are projective (resp. Moishezon) for $0\neq t\in\Delta$ and $X_{0}$ satisfies the deformation invariance for Hodge number of type $(0,1)$ or admits a strongly Gauduchon metric, then $X_{0}$ is Moishezon. The first author--I-H. Tsai \cite{RT21} prove that if there exist uncountably many Moishezon fibers in the family, then any fiber with either of the above two conditions is still Moishezon.
	
It is well known that each Hodge number
	takes a constant value along the small differentiable deformation $X_t$ of $X_0$ when the central fiber $X_0$ satisfies the (standard) $\p\bp$-property
	or more generally, the Fr\"{o}licher spectral sequence of $X_0$ degenerates at the $E_1$-level, cf. \cite[Section
	5.1]{Gs} or \cite[Proposition 9.20]{Vo02}. Recall that the \textit{(standard)  {$\p\b{\p}$-property}} refers
	to: for every pure-type $d$-closed form on a compact complex manifold, the properties of
	$d$-exactness, $\p$-exactness, $\b{\p}$-exactness and
	$\p\b{\p}$-exactness are equivalent. So it is natural to study this topic with more general conditions on $X_0$, such as some `weak' $\p\bp$-properties  see Subsection \ref{section Variations ofddbar}.  Such results were first given in \cite{ZR15,RZ18}. Significantly,  the weak version of $\p\bp$-property was initially introduced in \cite{FY11} while investigating deformations of balanced manifolds. Recently, W. Xia studied this further in terms of canonical deformations \cite[Theorem 1.3]{Xi21b}, see also a much recent work \cite{WX23} of Wan--Xia.  
    
    Drawing on  Xia's work and a  recent work of Wei--Zhu \cite{WZ20}, we obtain several new  theorems on the deformation invariance of Hodge numbers:
	
	\begin{bigthm}\label{pq def}
		Let $\pi:\mathcal{X} \rightarrow
		B$  be	a differentiable family of compact complex $n$-dimensional manifolds
		over a sufficiently small domain in $\mathbb{R}^d$  as in Definition \ref{Differential family} with the
		central  fiber $X_0:= \pi^{-1}(0)$ and the general fibers $X_t:=\pi^{-1}(t)$.
		Consider the  function
		\begin{equation}\label{hodge number equation}
			B \ni t\longmapsto h^{p,q}_{\db_t}(X_t):=\dim_{\mathbb{C}}H^{p,q}_{\db_t}(X_t),\quad \text{for any non-negative integers $p,q\leq n$}.
		\end{equation}
		If the injectivity of the mapping $\iota_{\mathrm {BC},\p}^{p,q+1}$, the surjectivity of the mapping $\iota_{\mathrm {BC},\bp}^{p,q}$ on the central fiber
		$X_0$ and the deformation invariance of the $(p,q-1)$-Hodge number $h^{p,q-1}_{\db_t}(X_t)$ hold,
		then $h^{p,q}_{\db_t}(X_t)$ are independent of $t$. See the  notations $\iota_{\mathrm {BC},\p}^{p,q+1},$ $\iota_{\mathrm {BC},\bp}^{p,q}$ in Subsection \ref{section Variations ofddbar}.
\end{bigthm}
	
	Examples \ref{example 1}, \ref{example 2}, and \ref{example 3} show  that the deformation invariance of the $(p,q)$-Hodge number fails when one of the three conditions in Theorem \ref{pq def} is not satisfied, while the other two hold, thanks to the  Kuranishi family of the Iwasawa manifold (see \cite[Appendix]{Ag13} and \cite[Section 1.c]{Sc07}).
	
	Following \cite[Notation 3.5]{RZ18}, we say that a compact complex manifold $X$ satisfies $\mathbb{S}^{p,q}$
	(resp. $\mathbb{B}^{p,q}$), if for any $\db$-closed
	$\p g\in A^{p,q}(X)$, the equation
	\begin{equation}\label{555555}
		\db x = \p g
	\end{equation}
	has a solution
	(resp. a $\p$-exact solution) of pure-type complex differential form. Similarly,
	a compact complex manifold $X$ is said to satisfy
	$\mathcal{S}^{p,q}$
	(resp. $\mathcal{B}^{p,q}$), if
	for any $\db$-closed $g\in A^{p-1,q}(X)$, the equation (\ref{555555}) has a solution
	(resp. a $\p$-exact solution) of pure-type complex differential form.
	
	The special cases of the $(p,0)$  and $(0,q)$  types may allow for a weakening of the conditions in Theorem~\ref{pq def}:
	\begin{bigthm}[=Theorems \ref{inv-p0}+\ref{inv-0q}]\label{inv-p0-0q-intro}With the setting of Theorem \ref{pq def},
		\begin{enumerate}[$(1)$]
			\item \label{inv-p0-intro}
			if $X_0$ satisfies $\mathbb{B}^{p,1}$ (i.e., the mapping $\iota_{\mathrm {BC},\p}^{p,1}$ is injective) and $\mathcal{S}^{p+1,0}$, then $h^{p,0}_{\db_t}(X_t)$ are deformation invariant;
			\item \label{inv-0q-intro} if $X_0$ satisfies $\mathcal {B}^{1,q}$ (i.e., the mapping $\iota^{0,q}_{\mathrm {BC},\db}$ is surjective)
			and $h^{0,q-1}_{\db_t}(X_t)$  satisfies the deformation invariance, then
			$h^{0,q}_{\db_t}(X_t)$ are independent of $t$.
		\end{enumerate}
	\end{bigthm}
	
	Note  that Theorem \ref{inv-p0-0q-intro} (\ref{inv-0q-intro}) first appeared as \cite[Theorem 3.7]{RZ18}. So we can also obtain \cite[Remark 3.8, Corollary 3.9]{RZ18}  as a consequence. Furthermore, an example applicable to  Theorem  \ref{inv-p0-0q-intro} (\ref{inv-p0-intro}) in Remark \ref{researchgate} shows that the function of Theorem  \ref{inv-p0-0q-intro} (\ref{inv-p0-intro}) for $(p,0)$-Hodge numbers goes beyond Kodaira--Spencer's squeeze \cite[Theorem 13]{KS60} sometimes.
	
	Before setting out the strategy to prove Theorems \ref{pq def} and \ref{inv-p0-0q-intro}, we first briefly state some knowledge of analytic deformation theory of complex structures to be introduced in detail in Subsection \ref{subsection Deformation theory and extension equ }. Let $\pi:\mathcal{X} \rightarrow B$ be a differentiable family as aforementioned inducing a canonical differentiable family of \textit{integrable Beltrami differentials} on $X_0$, denoted by $\varphi(z,t),$ $\varphi{(t)},$ and $\varphi$ interchangeably.

Now  consider the \textit{exponential operator} of contraction operators
	$$e^{\iota_\v}:=\sum_{k=0}^{\infty}\frac{{\iota^k_{\v(t)}}}{k!},$$
	and one can show that
	$$
	e^{\iota_{\v}}: A^{p, 0}(X_0) \longrightarrow A^{p, 0}(X_{t}).
	$$However, $e^{\iota_{\v}}$ can't preserve $(p,q)$-forms  to $X_t$ when $1\leq q\leq n$.  A useful fact due to  \cite[Theorem 5.1]{FM09} or \cite[Section 3.1]{WZ20} shows that it actually maps from  a filtration on $X_0$ to one on $X_t$, that is
	$$e^{\iota_{\varphi}}: F^{p} A^{k}(X_0)=\mathop\bigoplus_{p \leq  s \leq k} A^{ s, k- s}(X_0) \longrightarrow F^{p} A^{k}(X_{t}).$$
	According to \cite{WZ20,Xi21b}, one can define the \textit{projection operator}
	$$	\mathcal P_{\v}:A^{0,1}(X_0)\longrightarrow A^{0,1}(X_t).$$
	
We then use  the exponential operator
	$e^{\iota_\v}$ and the projection operator $\mathcal P_\v$ to define the \textit{extension operator}
	\begin{equation}\label{filtration}\rho_\varphi: A^{p,q}(X_0)\longrightarrow A^{p,q}(X_t)
	\end{equation}
	as in \eqref{ccccc}.
	To prove the deformation invariance of Hodge numbers,  the first author--Q. Zhao \cite{ZR15,RZ18} introduced an extension map
	\begin{equation}\label{RAozhao}
		e^{\iota_{\varphi}|\iota_{\overline{\varphi}}}:
		A^{p,q}(X_0)\longrightarrow A^{p,q}(X_t),
	\end{equation} which can preserve all $(p,q)$-forms and played an important role in their subsequent papers \cite{RWZ19,RWZ21}. By comparing the explicit local expressions,  one can deduce the relationship between these  two extension maps:
	$$\rho_{\varphi}=e^{\iota_{\varphi}|\iota_{\overline{\varphi}}}\circ{\left(\1-\overline\varphi\varphi\right)}^{-1}\Finv,$$
	where the notation $\Finv$  denotes the
	simultaneous contraction on each component of a complex
	differential form.

	An application of Kuranishi's completeness theorem \cite{Ku64} can reduce our Theorems \ref{pq def} and \ref{inv-p0-0q-intro} to the Kuranishi family, which contains all sufficiently small differentiable deformations of $X_0$ in some sense (see Subsection \ref{subsection-Kuranishi family}). So to prove them, we just need to consider a Kuranishi family of deformations of $X_0$ over $\Delta_{\epsilon}:=\left\{t\in\mb C\,|\,|t|\leq\epsilon\right\}$ with small $\epsilon$, and there exists a family of integrable Beltrami differentials $\left\{\varphi(t)\right\}_{|t|\leq\epsilon} $ depending on $t$ holomorphically and describing the variations of complex structures on $X_0.$
	
	Recently, Wei--Zhu \cite{WZ20} applied the $\p\bp$-Hodge theory to extend a $d$-closed $(p,q)$-form on $X_0$ to a $d$-closed filtrated $(p+q)$-form on $X_t$, whose  $(p,q)$-part on $X_t$ is $\bp_{t}$-closed via the extension operator $\rho_\varphi$ in (\ref{filtration}).  This is surely an interesting and important result. We will reprove their theorem in Subsection \ref{important}, from the perspective of Bott--Chern theory.
	
	\begin{bigthm}[{\cite[Theorem 1.1]{WZ20}}]\label{main thm}
		Let  $X_0$ be a compact complex manifold that  satisfies   $\mathbb B^{p,q+1}$. Given a $d$-closed $\mu_{0} \in A^{p, q}(X_0)$, one can construct $\mu(z,t)\in A^{p,q}(X_0)$ satisfying $d(e^{\iota_\varphi}(\mu(z,t)))=0$ on $X_0$ (or $X_t$) and $\mu(z,0)=\mu_0(z)$, which is  holomorphic in small $t$. Furthermore, the extension $\rho_\v(\mu(z,t)) \in A^{p, q}\left(X_{t}\right)$
		is $\bar{\partial}_{t}$-closed.
	\end{bigthm}
	
	Now we are ready to describe our strategy to consider the deformation invariance of Hodge numbers briefly. The Kodaira--Spencer's upper semi-continuity theorem (\cite[Theorem $4$]{KS60}) tells us that the function \eqref{hodge number equation} is always upper semi-continuous for $t\in \Delta_{\varepsilon}$ and thus, to approach the deformation invariance of $h^{p,q}_{\db_t}(X_t)$, we only need to obtain the lower semi-continuity. 
    
    And, our main strategy is to  look for an injective extension map from $H^{p,q}_{\db}(X_0)$ to $H^{p,q}_{\db_t}(X_t)$ with the help of Theorem \ref{main  thm}. More precisely,  we first need to find a nice uniquely-chosen $d$-closed representative $\mu_0$ of the given  initial Dolbeault cohomology class in $H^{p,q}_{\db}(X_0)$ (see Lemma \ref{lemma1}), and then apply power series method and Bott--Chern theory to construct  $\mu(z,t)\in A^{p,q}(X_0)$, such that it is smooth in $(z,t)$ and holomorphic in small $t$ and $\rho_\v(\mu(z,t))$ is $\bar{\partial}_{t}$-closed in $A^{p, q}\left(X_{t}\right)$ by Theorem \ref{main  thm}. Finally, we try to verify that the extension map  $$H^{p,q}_{\db}(X_0) \longrightarrow H^{p,q}_{\db_t}(X_t):[\mu_0]_{\db} \longmapsto [\rho_\v(\mu(z,t))]_{\db_t}$$ is injective.
	
	Our Theorem \ref{pq def} and the first assertion of Theorem \ref{inv-p0-0q-intro} on the deformation invariance of Hodge numbers are different from the results in \cite{RZ18}, cf.  Remarks \ref{example remark}, \ref{Remark 1}, and \ref{Remark 2}. Our approach focuses more on the $d$-closed extension, while they concentrated on the specific $\bp$-extension from $A^{p,q}(X_0)$ to $A^{p,q}(X_t)$ (see \cite[Proposition 1.2]{RZ18}) by use of their extension map \eqref{RAozhao}.
	
	As another application of the extension of closed complex differential forms in Theorem \ref{main thm}, we study the local stabilities of several special complex structures. Inspired by the proof of \cite[Theorem 9.23]{Vo02}, we can take advantage of Theorem \ref{main thm}  and also the deformation openness of the $\p\bp$-property to prove the local stabilities of \textit{$p$-K\"{a}hler structures} (for this concept, one can refer to Appendix \ref{p} for more details) with the  $\p\bp$-property.

	\begin{bigthm}[{\cite[Theorem 4.9]{RWZ21}}]\label{pkahler introduction}
		For any positive integer $p\leq n-1$, any
		small differentiable deformation $X_t$ of a
		$p$-K\"ahler manifold $X_0$ satisfying the $\p\db$-property is still
		$p$-K\"ahlerian.
	\end{bigthm}

	Notice that \cite{RWZ21} presented a power series proof for Kodaira--Spencer's local stabilities theorem of K\"{a}hler structures via their extension map (\ref{RAozhao}), which is a problem at latest dated back to \cite[Remark 1 on p. 180]{MK71}: 	`A
	good problem would be to find an elementary proof (for example,
	using power series methods). Our proof uses nontrivial results from
	partial differential equations'. In our proof, our primary focus is on the $d$-closed extension, by means of the exponential operator $e^{\iota_\v}$, which is more natural and succinct in some sense. However, we still need to use \cite[Theorem 7]{KS60} to guarantee the positivity of the constructed explicit $p$-K\"ahler form, see Subsection \ref{proof pkahler} for more details. A challenge problem proposed by \cite{WZ20} is how to prove
	a direct corollary of Theorem \ref{pkahler introduction}, that is Corollary \ref{cor-mt} \eqref{stab-Kahler},  without using \cite[Theorem 7] {KS60}.
	
	It is worth mentioning that in \cite{RWZ19}, the first author--X. Wan--Q. Zhao looked deeper into the local stabilities of $p$-K\"{a}hler structure when the central fiber $X_0$ satisfies some `weak' $\p\bp$-properties. More concretely, for any positive integer $p\leq n-1$, any
	small differentiable deformation $X_t$ of  an $n$-dimensional
	$p$-K\"ahler manifold $X_0$ satisfying  the
	$(p,p+1)$-th mild $\p\db$-property is still
	$p$-K\"ahlerian (\cite[Theorem 1.1]{RWZ19}). However, in our approach, it seems that  the standard $\p\bp$-property condition on $X_0$ in Theorem \ref{pkahler introduction} can't be weakened, see Remark \ref{weaken remark}.
	
	As is well known,  the Bott--Chern numbers are  always upper semi-continuous with respect to the ordinary topology in a small differentiable family. Since the ordinary topology is much finer than  the analytic Zariski topology, it's natural to ask:
	\begin{bigque}[{=Question \ref{question}}] \label{question!!!}
		For a holomorphic family $\{X_t\}$,
		is the function
		$$B\ni t\longmapsto h^{p,q}_{\mathrm {BC}}(X_t):=\dim_{\mathbb{C}}H^{p,q}_{\mathrm {BC}}(X_t)$$
		upper semi-continuous with respect to the analytic Zariski topology?
	\end{bigque}
	
	To solve this question,  we wish to use Grauert's upper semi-continuity theorem \ref{Upper semi-continuity}. In other words, one needs to find some holomorphic vector bundle $V$ on $\mathcal X$, such that  $H^{p,q}_{\mathrm {BC}}(X_t)\simeq H^q(X_t,V|_{X_t}).$  This seems hard to achieve due to the results in \cite{Bg69,AN71}, see Subsection  \ref{further} for more details.  Recently,  Xia \cite[Theorem 1.1 and Remark 3.6]{Xi21a} confirmed Question \ref{question!!!} when the type is $(p,0)$ or $(0,q)$. One motivation to affirm  Question \ref{question!!!} is to obtain Theorem \ref{uncontable}, as long as one notices that the $\p\bp$-property in Theorem \ref{pkahler introduction} actually can be replaced by the deformation invariance of $(p,p)$-Bott--Chern numbers as shown in \cite[Remark 4.13]{RWZ21}.

	As is widely recognized, the transversely K\"ahler structures hold a central position within the field of foliation theory and   are closely linked to a wealth of geometric structures.  For instance, Vaisman manifolds, {LVM} manifolds (a generalized version of Calabi--Eckmann manifolds), and Sasakian manifolds all possess transversely K\"ahler structures despite not being K\"ahler themselves, see \cite{Me00,BG08,MSY08,No08,FOW09,OV22}, etc. 
    
    Recently, the transversely balanced structures (or, more generally, the transversely Gauduchon structures, see Remark \ref{nvwa}), are also actively studied, as evidenced by  \cite{FZ19,BH22}, etc. In light of this, it naturally becomes a logical progression to extend the concept of compact $p$-K\"ahler manifolds, as originally introduced by L. Alessandrini--M. Andreatta \cite[Definition 1.11]{AA87}, to the transverse context. This extension leads us to the introduction of \textit{transversely $p$-K\"ahler foliations} (Definition \ref{TRAN P-Kahler}). Remarkably, this overarching notion unifies the two aforementioned structures, specifically when $p$ takes values of $1$ and $r-1$ (in the context of a transversely holomorphic foliation of codimension $r$),
	respectively (Remark \ref{nvwa}). 
    
When delving into the intermediate cases ($1<p<{r-1}$), a thought-provoking question naturally arises: Does there exist a non-trivial (specifically, with $1<p<{r-1}$, and non-transversely K\"ahler) transversely $p$-K\"ahler foliation? The answer to this question is affirmative, as demonstrated in Example \ref{nilmanifold}.
	So, it seems  that the transversely $p$-K\"ahler structures represent an interesting subject in foliation theory, especially within the realm of non-transversely K\"ahler geometry.

	In the second part of this paper, we are mainly concerned with  the local stabilities of transversely $p$-K\"ahler structures.
	
	El Kacimi Alaoui proved that small deformations of a compact K\"ahler orbifold as an orbifold are still K\"ahler \cite{EKA88}, which is a generalization of Kodaira--Spencer's result of stability  for compact K\"ahler manifolds. Later, El Kacimi Alaoui--Gmira \cite{EKAG97}  proved a much more general result as follows: let $\mathcal{F}$ be a \textit{homologically orientable} Hermitian foliation on a compact manifold and $\mathcal{F}_{t}$ a deformation of $\mathcal{F}$ by a transversely holomorphic foliation with fixed differentiable type, parametrized by an open neighborhood of 0 in $\mathbb{R}^{d}$. Suppose that a transversely Hermitian metric $\sigma$ of $\mathcal{F}=\mathcal{F}_{0}$ is K\"ahlerian. Then $\mathcal{F}_{t}$ has a transversely K\"ahler metric $\sigma_{t}$ for every $t$ sufficiently close to 0, depending differentiably on $t$ with $\sigma_{0}=\sigma$. See also \cite[Theorem 5.2]{Ra21}.  One can refer to Subsections \ref{subsection transversely }--\ref{subsection foliation defor} for relevant definitions. 
    
    Notice that if the deformation $\mc F_t$ is arbitrary, then the above assertion doesn't hold, due to the example built in \cite{EKAN90},  namely an analytic family $\left\{\mc F_t\right\}_{t\in\mb C}$ of holomorphic foliations on a compact complex nilmanifold such that $\mc F_0$ is transversely K\"ahler, but for any $t\not=0$, $\mathcal F_t$ has no transversely K\"ahler structure. See also the discussion  in \cite[Section 6]{Ra21} based on the examples  from \cite{No08, GNT16}.

	Inspired by \cite{RWZ19},  we  extend the `weak' $\partial\bp$-properties to the  foliated version  (Definition \ref{MILD}), and then prove the following local stabilities theorem by adapting the power series method  to our setting. The main ingredients in this proof are the transversely elliptic operator theory initiated from \cite{EKA90} and the Kuranishi family constructed by \cite{EKAN89}, see
	Subsection \ref{Foliation proofproof} for more details. This is our  \textit{main result} in this paper.
	\begin{bigthm}[{Main Result}]\label{FOLiation Pkahler}
		Let $\left\{\mc F_t\right\}_{t\in U}$ be a smooth family of transversely Hermitian structures on a compact foliated manifold $(M,\mc F)$   with fixed differentiable type ($\mc F$ is of complex codimension $r$), parametrized by an open neighborhood $U$ of $0$ in $\mathbb{R}^{d}$. If   $\mathcal{F}=\mathcal{F}_{0}$ is transversely $p$-K\"ahler and satisfies the $(p,p+1)$-th mild $\p\b{\p}$-property for $1\leq p\leq r-1$, then $\mathcal{F}_{t}$ is also  transversely  $p$-K\"ahler for every $t$ sufficiently close to $0$.
	\end{bigthm}

This result strengthens the works of  \cite{EKAG97,Ra21}  on that of the transversely K\"ahler foliations with homologically orientability, and  can be viewed as a generalization of \cite[Theorem 1.1]{RWZ19} to the foliated case.  In a certain sense, the approach we have adopted has resulted in greater efficiency compared to the methods used in \cite{EKAG97,Ra21}, cf. Remark \ref{Non} (\ref{AAA}).

Notice that even when $p=1$ (transversely K\"ahler) Theorem \ref{FOLiation Pkahler}  is new.

\begin{bigcor}\label{1FOLiation 1kahler}
	With the setting of Theorem \ref{FOLiation Pkahler}, if   $\mathcal{F}=\mathcal{F}_{0}$ is transversely K\"ahler, then $\mathcal{F}_{t}$ is also  transversely  K\"ahler for every $t$ sufficiently close to $0$.
\end{bigcor}

Compared with the previous works \cite{EKAG97,Ra21}, Corollary \ref{1FOLiation 1kahler} can drop the \textit{homologically orientability} assumption. To this end, We observe that a transversely K\"ahler foliation, even without homologically orientability, also satisfies the $\p\bp$-property and, therefore,   satisfies $(1,2)$-th mild $\p\b{\p}$-property, cf. Subsection \ref{nonHO} and Remark \ref{Non} (\ref{BBB}).
	
	Finally, several theorems concerning the deformation invariance of basic Hodge/Bott--Chern numbers with `weak' $\p\bp$-properties are displayed, see Theorem \ref{basic hodge}. In these theorems, the homologically orientability assumption is necessary, as indicated in Subsection \ref{defor}. We also require the foliation to remain unchanged in this context. For further insights, one can refer to Question \ref{Razny question} posed by Ra\'zny (and also the preceding discussions) regarding the invariance of basic Hodge numbers under arbitrary deformations of transversely K\"ahler foliations.
	
	\begin{Conventionno}
		All compact complex (or smooth) manifolds in this paper are assumed to be connected unless mentioned otherwise and $\pi: \mathcal{X} \rightarrow B$ will always denote a differentiable or holomorphic family of $n$-dimensional compact complex manifolds, whose central fiber is $(X_0,z)$ with local holomorphic coordinates $z:=(z^i)_{i=1,\ldots,n}$ and general fiber is $X_t:=\pi^{-1}(t)$.
	\end{Conventionno}

	\noindent
\textbf{Acknowledegments}
Both authors would like to thank Professors A. El Kacimi Alaoui, M. Nicolau and P. Ra\'zny, for explaining us the details on their articles \cite{EKAN89} and \cite{Ra21}, respectively. We also appreciate Professors Wei Xia, Xiangdong Yang, Tao Zheng and Shengmao Zhu for giving many useful suggestions on this paper.  
Additionally, we thank Professors D. Angella and J.-P. Demailly for their two discussions via emails on the upper semi-continuity of Bott--Chern numbers. Last but not least, we would like to express our gratitude to the anonymous reviewers for their numerous helpful and constructive suggestions, particularly on Example \ref{nilmanifold} and Subsection \ref{nonHO}, which have significantly improved this paper.

\section{Variations of $\partial\bar\partial$-property and $\partial\bar\partial$-equation}
In this section, we collect some basics  to be used later. Throughout this section, we will always denote by $X$ a compact complex manifold of complex dimension $n$.
\subsection{Cohomology groups and variations of $\p\bp$-property}\label{section Variations ofddbar}
	We will often use the commutative diagram:
	\begin{equation}\label{diag}
		\xymatrix{      & H^{p,q}_{\p}(X) \ar[dr]^{\iota^{p,q}_{\p,\mathrm{A}}} &            \\
			H^{p,q}_{\mathrm {BC}}(X) \ar[ur]^{\iota^{p,q}_{\mathrm {BC},\p}} \ar[dr]_{\iota^{p,q}_{\mathrm {BC},\db}} \ar[r]^{\iota^{p,q}_{\mathrm {BC},dR}} &  H^{p+q}_{dR}(X)\ar[r]^{\iota^{p,q}_{dR,\mathrm{A}}} &  H^{p,q}_{\mathrm{A}}(X)  \\
			& H^{p,q}_{\db}(X) \ar[ur]_{\iota^{p,q}_{\db,\mathrm{A}}} &         . }
	\end{equation}
	Recall that \textit{Dolbeault cohomology groups} $H^{\bullet,\bullet}_{\db}(X)$ of $X$ are defined by:
	$$H^{\bullet,\bullet}_{\db}(X):=\frac{\ker\db}{\im\ \db},$$
	with $H^{\bullet,\bullet}_{\p}(X)$ likewise defined, while \textit{Bott--Chern} and \textit{Aeppli cohomology groups} are defined as \begin{equation}\label{Bottchern coho}
		H^{\bullet,\bullet}_{\mathrm {BC}}(X):=\frac{\ker \p\cap \ker\db}{\im\ \p\db}\quad
		\text{and}\quad H^{\bullet,\bullet}_{\mathrm{A}}(X):=\frac{\ker \p\db}{\im\ \p+\im\ \db},
	\end{equation}
	respectively. 
    
    The dimensions of $H^{p+q}_{dR}(X)$, $H^{p,q}_{\db}(X)$, $H^{p,q}_{\mathrm {BC}}(X)$, $H^{p,q}_{\mathrm{A}}(X)$ and $H^{p,q}_{\p}(X)$ over $\mathbb{C}$
	are denoted by $b_{p+q}(X)$, $h^{p,q}_{\db}(X)$, $h^{p,q}_{\mathrm {BC}}(X)$, $h^{p,q}_{\mathrm{A}}(X)$ and $h^{p,q}_{\p}(X)$, respectively, and the first four of them
	are usually called the \textit{$(p+q)$-th Betti numbers}, \textit{$(p,q)$-th Hodge numbers}, \textit{Bott--Chern numbers} and \textit{Aeppli numbers} of $X$, respectively. From the very definition of these cohomology groups,
	the following equalities clearly hold
	\begin{equation*} \label{dual Aeppli}
		h^{p,q}_{\mathrm {BC}} = h^{q,p}_{\mathrm {BC}}=h^{n-q,n-p}_{\mathrm{A}}=h^{n-p,n-q}_{\mathrm{A}},  h^{n-p,n-q}_{\db}=h^{p,q}_{\db}=h^{q,p}_{\p}=h^{n-q,n-p}_{\p}.
	\end{equation*}
	
	So the \textit{(standard) $\p\db$-property},  which means for every pure-type $d$-closed form on a compact complex manifold, the properties of
	$d$-exactness, $\p$-exactness, $\b{\p}$-exactness and
	$\p\b{\p}$-exactness are equivalent, is equivalent to the following mappings
	$$\iota^{p,q}_{\mathrm {BC},dR}: H^{p,q}_{\mathrm {BC}}(X)\longrightarrow H^{p+q}_{dR}(X)$$
	are injective for all $p,q$, or to the isomorphisms of all the maps in Diagram \eqref{diag} by \cite[Remark 5.16]{DGMS75}.
	
	Recall that  a compact complex manifold $X$ satisfies $\mathbb{S}^{p,q}$ (resp. $\mathbb{B}^{p,q}$)
	if for any $\db$-closed
	$\p g\in A^{p,q}(X)$, the equation (\ref{555555})
	has a solution
	(resp. a $\p$-exact solution) of pure-type complex differential form. Similarly,
	a compact complex manifold $X$ is said to satisfy
	$\mathcal{S}^{p,q}$
	(resp. $\mathcal{B}^{p,q}$), if
	for any $\db$-closed $g\in A^{p-1,q}(X)$, the equation (\ref{555555}) has a solution
	(resp. a $\p$-exact solution)  of pure-type complex differential form. It is easy to verify the following implications:
	\[ \begin{array}{ccc}
		\mathbb{B}^{p,q} & \Rightarrow & \mathbb{S}^{p,q} \\
		\Downarrow &   & \Downarrow \\
		\mathcal{B}^{p,q} & \Rightarrow & \mathcal{S}^{p,q}. \\
	\end{array} \]
	And it is apparent that a compact complex manifold $X$, where the $\p\db$-property holds, satisfies
	$\mathbb{B}^{p,q}$ for any $(p,q)$.
	
	It is easy to check that the following statements are equivalent:
	\begin{enumerate}\label{equivalent }
		\item\label{equ 1}
		the \textit{injectivity} of $\iota^{p,q}_{\mathrm {BC},\p}$ holds on $X$ $\Leftrightarrow$
		{$X$ satisfies $\mathbb{B}^{p,q}$};
		\item\label{equ2}
		the \textit{injectivity} of $\iota^{p,q}_{\db,\mathrm{A}}$ holds on $X$ $\Leftrightarrow$
		$X$ satisfies $\mathbb{S}^{p,q}$;
		\item\label{equ3}
		{the \textit{surjectivity} of $\iota^{p-1,q}_{\mathrm {BC},\db}$ holds on $X$} $\Leftrightarrow$
		{$X$ satisfies $\mathcal{B}^{p,q}$}.
	\end{enumerate}
	\subsection{Hodge theory on compact complex manifolds: Bott--Chern}
	Let $X$ be a compact complex manifold. The Bott--Chern cohomology group has been introduced in (\ref{Bottchern coho}).
	Also, the \textit{Bott--Chern Laplacian} is given by
	\beq\label{bc-Lap}
	\square_{{\mathrm {BC}}}:=\p\db\db^*\p^*+\db^*\p^*\p\db+\db^*\p\p^*\db+\p^*\db\db^*\p+\db^*\db+\p^*\p,
	\eeq and $\G_{\mathrm {BC}}$ is the associated \textit{Green's operator} of
	this  fourth order Kodaira--Spencer
	operator. 
    
    Then we have the \textit{Hodge decomposition} of $\square_{\mathrm {BC}}$ on $X$:
	$$\label{bc-hod}
	A^{p,q}(X)=\ker\square_{\mathrm {BC}}\oplus \textrm{im}~(\p\db)
	\oplus(\textrm{im}~\p^*+\textrm{im}~\db^*),
	$$
	whose three parts are orthogonal to each other with respect to the
	$L^2$-scalar product defined by a Hermitian metric on $X$, combined with the equality
	$$\label{identity}
	\1=\mathbb{H}_{\mathrm {BC}}+\square_{\mathrm {BC}}\G_{\mathrm {BC}}=\mathbb{H}_{\mathrm {BC}}+\G_{\mathrm {BC}}\square_{\mathrm {BC}},
	$$
	where $\mathbb{H}_{\mathrm {BC}}$ is the \textit{harmonic projection operator}. And it
	should be noted that
	\begin{equation*}\label{bc-ker}
		\ker\square_{\mathrm {BC}}=\ker\p\cap\ker\db\cap\ker\ (\p\db)^*.
	\end{equation*}
	We get  the following two observations: \begin{enumerate}[(1)]\label{commuate}
		\item \label{case-1}
		$\square_{\mathrm {BC}}\p\db(\p\db)^*=\p\db(\p\db)^*\square_{\mathrm {BC}};$
		\item \label{case-2}
		$\G_{\mathrm {BC}}\p\db(\p\db)^*=\p\db(\p\db)^*\G_{\mathrm {BC}}.$
	\end{enumerate}
	
	For the resolution of $\p\db$-equations, we need a crucial lemma.
	\begin{lemma}[{\cite[Theorem 4.1]{Po15}}]\label{ddbar-eq}
	Let $(X,\omega)$ be a compact Hermitian manifold and $\alpha$ a pure-type complex differential form.
	If the $\p\db$-equation	$\p \db x = \alpha$
	admits a pure-type solution, then
	\[
	(\p\db)^*\G_{\mathrm{BC}}\alpha
	\]
	is a solution that uniquely minimizes the $L^2$-norm among all solutions with respect to~$\omega$.   Besides, the equalities hold
		\[\G_{\mathrm {BC}}(\p\db) = (\p\db) \G_{\mathrm{A}}\quad \text{and} \quad (\p\db)^*\G_{\mathrm {BC}} = \G_{\mathrm{A}}(\p\db)^*,\]
		where $\G_{\mathrm {BC}}$ and $\G_{\mathrm{A}}$ are the associated Green's operators of
		$\square_{\mathrm {BC}}$ and $\square_{\mathrm{A}}$, respectively. Here  $\square_{\mathrm {BC}}$ is defined in \eqref{bc-Lap} and
		$\square_{{\mathrm{A}}}$ is the second
		Kodaira--Spencer operator (often also called Aeppli
		Laplacian)
		$$\square_{{\mathrm{A}}}=\p^*\db^*\db\p+\db\p\p^*\db^*+\db\p^*\p\db^*+\p\db^*\db\p^*+\db\db^*+\p\p^*.$$
	\end{lemma}

	\section{Extension of closed forms: Bott--Chern approach}\label{deformation p,q Bottchern}
	Wei--Zhu \cite{WZ20} applied the $\p\bp$-Hodge theory to extend a $d$-closed $(p,q)$-form on $X_0$ to a $d$-closed filtrated $(p+q)$-form on $X_t$, whose  $(p,q)$-part on $X_t$ is $\bp_{t}$-closed by type projection.  We will reprove their theorem in this section by applying Bott--Chern theory.
	\subsection{Beltrami differentials and extension map}\label{subsection Deformation theory and extension equ }
	By a \textit{holomorphic family} $\pi:\mc{X}\rightarrow B$ of compact complex manifolds from a complex manifold to a connected complex manifold, we mean that $\pi$ is a proper and surjective holomorphic submersion, as in \cite[Definition 2.8]{Ko86}; while for differentiable one, we adopt:
	\begin{definition}[{\cite[Definition 4.1]{Ko86}}] \label{Differential family} {Let $\mc{X}$ be a differentiable manifold, $B$ a domain of $\mathbb{R}^d$ and $\pi$ a smooth map of $\mc{X}$ onto $B$.
			By a  \textit{differentiable family of $n$-dimensional compact complex manifolds} we mean the triple $\pi:\mc{X}\rightarrow B$ satisfying the following conditions:
			\begin{enumerate}[$(a)$]
				\item
				The rank of the Jacobian matrix of $\pi$ is equal to $d$ at every point of $\mc{X}$;
				\item
				For each point $t\in B$, $\pi^{-1}(t)$ is a compact connected subset of $\mc{X}$;
				\item
				$\pi^{-1}(t)$ is the underlying differentiable manifold of the $n$-dimensional compact complex manifold $X_t$ associated to each $t\in B$;
				\item
				There is a locally finite open covering $\{\mathcal{U}_j\ |\ j=1,2,\cdots\}$ of $\mc{X}$ and complex-valued smooth functions $\zeta_j^1(p),\cdots,\zeta_j^n(p)$, defined on $\mathcal{U}_j$ such that for each $t$, $$\{p\longmapsto (\zeta_j^1(p),\cdots,\zeta_j^n(p))\ |\ \mathcal{U}_j\cap \pi^{-1}(t)\neq \emptyset\}$$ form a system of local holomorphic coordinates of $X_t$.
		\end{enumerate}}
	\end{definition}
	
	Beltrami differential plays an important role in deformation theory. For a compact complex manifold $X$, we call  an element in $A^{0,1}(X,
	T^{1,0}_X)$ a \textit{Beltrami differential}, where $T^{1,0}_X$ denotes the holomorphic tangent
	bundle of $X$. Then $\iota_\varphi$ or $\varphi\lrcorner$ denotes the
	contraction operator with $\varphi\in A^{0,1}(X,T^{1,0}_X)$
	alternatively if there is no confusion. One similarly follows the
	notation
	\begin{equation*}\label{0e-convention}
		e^{\heartsuit}=\sum_{k=0}^\infty \frac{1}{k!} \heartsuit^{k},
	\end{equation*}
	where $\heartsuit^{k}$ denotes $k$-time action of the operator
	$\heartsuit$.  The summation in
	the above formulation is often finite since the dimension of $X$ is finite.
	
	We will always consider a differentiable family $\pi:\mathcal{X} \rightarrow
	B$ of compact complex $n$-dimensional manifolds
	over a sufficiently small domain in $\mathbb{R}^d$ with the
	reference fiber $X_0:= \pi^{-1}(0)$ and the general fibers $X_t:=
	\pi^{-1}(t).$ For simplicity we set $d=1$. Denote by
	$\zeta:=(\zeta^\alpha_j(z,t))$ the holomorphic coordinates of $X_t$
	induced by the family with the holomorphic coordinates $z:=(z^i)$ of
	$X_0$, under a coordinate covering $\{\mathcal{U}_j\}$ of
	$\mathcal{X}$, when $t$ is assumed to be fixed, as the standard
	notions in deformation theory described at the beginning of
	\cite[Chapter 4]{MK71}. This family induces a canonical differentiable family of
	integrable Beltrami differentials on $X_0$, denoted by $\varphi(z,t)$,
	$\varphi(t)$ and $\varphi$ interchangeably.
	
	In the sequel we will state some explicit  computations as presented in \cite[Chapter 4.1]{MK71} or \cite[Section 2.1]{RZ18}. A Beltrami differential can be written as
	\begin{equation}\label{varphi local expression}
		\varphi(t)=\left(\frac{\partial}{\partial z}\right)^{\rm{T}}\left(\frac{\partial \zeta}{\partial z}\right)^{-1} \bar{\partial} \zeta,
	\end{equation}
	where $\frac{\partial}{\partial z}=\left(\begin{array}{c}\frac{\partial}{\partial z^{1}} \\ \vdots \\ \frac{\partial}{\partial z^{n}}\end{array}\right), \,\overline{\partial} \zeta=\left(\begin{array}{c}\overline\partial \zeta^{1} \\ \vdots \\ \overline{\partial} \zeta^{n}\end{array}\right), \,\frac{\partial \zeta}{\partial z}$ stands for the matrix $\left(\frac{\partial \zeta^{\alpha}}{\partial z^{j}}\right)_{1 \leq \alpha \leq n \atop 1 \leq j \leq n}$ and $\alpha, j$ are the row and column indices. Here $\left(\frac{\partial}{\partial z}\right)^{\rm T}$ is the transpose of $\frac{\partial}{\partial z}$ and $\overline{\partial}$ denotes the Cauchy--Riemann operator with respect to the holomorphic structure on $X_{0}.$
	Locally, $\varphi(t)$ is expressed as ${\varphi_{\overline j}^i}\dz^j\otimes\frac{\p}{\p z^i}\in A^{0,1}(X_0,T_{X_0}^{1,0}),$ so it can be considered as a matrix ${\left(\v^i _{\overline j}\right)}_{1 \leq i \leq n \atop 1 \leq j \leq n}$. By (\ref{varphi local expression}), this matrix   can be explicitly written as
	\begin{equation}\label{6666}
		\varphi= {\left(\v^i _{\overline{ j}}\right)}_{1 \leq i \leq n \atop 1 \leq j \leq n}=\varphi(t)\big( \frac{\p\ }{\p \bar{z}^j}, dz^i \big)
		= \lk \lk \frac{\p \zeta}{\p z} \rk^{\!\!\!-1} \lk \frac{\p \zeta}{\p \bar{z}} \rk
		\rk^i_{\overline{ j}}.
	\end{equation}
    
	A fundamental fact is that the Beltrami differential $\varphi(t)$
	defined as above satisfies the integrability:
	\begin{equation*}\label{int-conditin}
		\bp\varphi(t)=\frac{1}{2}[\varphi(t),\varphi(t)],
	\end{equation*}
	and we then call $\varphi(t)$ an \textit{integrable Beltrami differential}.
	Now  consider the exponential operator of contraction operators
	$$e^{\iota_\v}:=\sum_{k=0}^{\infty}\frac{{\iota^k_{\v(t)}}}{k!}.$$
	A calculation shows  that $$e^{\iota_{\varphi}}\left(d z^{i_{1}} \wedge \cdots \wedge d z^{i_{p}}\right)=\left(d z^{i_{1}}+{\varphi(t)}\l d z^{i_{1}}\right) \wedge \cdots \wedge\left(d z^{i_{p}}+{\varphi(t)}\l d z^{i_{p}}\right).$$
	Then one obtains $$
	e^{\iota_{\v}}: A^{p, 0}(X_0) \longrightarrow A^{p, 0}(X_{t})
	$$
	by
	$$d \zeta^{\beta}=\frac{\partial \zeta^{\beta}}{\partial z^{i}} d z^{i}+\frac{\partial \zeta^{\beta}}{\partial \bar{z}^{i}} d \bar{z}^{i}=\frac{\partial \zeta^{\beta}}{\partial z^{i}}e^{\iota_\varphi} (d z^{i})$$
	according to (\ref{6666}). However, $e^{\iota_{\v}}$ can't preserve $(p,q)$-forms for $1\leq q\leq n.$  
    
    An interesting observation in \cite[Theorem 5.1]{FM09} or \cite[Section 3.1]{WZ20} tells that it in fact maps from a filtration on $X_0$ to one on $X_t$, i.e.,
	$$
	e^{\iota_{\varphi}}: F^{p} A^{k}(X_0)=\mathop\bigoplus_{p \leq s
		\leq k} A^{ s, k- s}(X_0) \longrightarrow F^{p} A^{k}(X_{t}).
	$$
	Actually, for $\sigma \in A^{ s, k- s}(X_0)$  ($p \leq  s \leq k$) with  the local expression
	$$\sigma=\sigma_{i_1\cdots i_s {j}_1\cdots {j}_{k-s}}dz^{i_1}\wedge \cdots \wedge d z^{i_{ s}} \wedge d \bar{z}^{j_{1}} \wedge \cdots \wedge d \bar{z}^{j_{k- s}},$$ one has
	$$\label{opearte}
	\begin{aligned}
		e^{\iota_{\varphi}}(\sigma)
		&=\sigma_{i_1\cdots i_s {j}_1\cdots {j}_{k-s}}e^{\iota_\varphi}\left(d z^{i_{1}} \wedge \cdots \wedge d z^{i_{ s}}\right) \wedge d \bar{z}^{j_{1}} \wedge \cdots \wedge d \bar{z}^{j_{k- s}} \\
		&=\sigma_{i_1\cdots i_s {j}_1\cdots {j}_{k-s}}\left(d z^{i_{1}}+{\varphi(t)}\l  d z^{i_{1}}\right) \wedge \cdots \wedge\left(d z^{i_{ s}}+{\varphi(t)}\l  d z^{i_{ s}}\right)\\&\quad\wedge\left(\frac{\partial \bar{z}^{j_{1}}}{\partial \zeta^ {\beta}} d \zeta^ {\beta}+\frac{\partial \bar{z}^{j_{1}}}{\partial \bar{\zeta}^ {\beta}} d \bar{\zeta}^ {\beta}\right) \wedge \cdots \wedge\left(\frac{\partial \bar{z}^{j_{k- s}}}{\partial \zeta^ {\beta}} d \zeta^ {\beta}+\frac{\partial \bar{z}^{j_{k- s}}}{\partial \bar{\zeta}^ {\beta}} d \bar{\zeta}^ {\beta}\right),
	\end{aligned}
	$$
	which clearly belongs to $F^{p} A^{k}\left(X_{t}\right)$. 
    
    Next, we will state the definitions of the projection operator and the extension operator, originating from \cite{WZ20,Xi21b}. For $\sigma\in A^{0,1}(X_0)$ with the local expression $\sigma=\sigma_j\dz^j$, one defines the projection operator $$\mathcal P_{\v}:A^{0,1}(X_0)\longrightarrow A^{0,1}(X_t)$$by
	\begin{equation}\label{prooooo}
		\begin{aligned}
			\mathcal P_{\v}\left(\sigma\right)&=\mathcal P_{\v}\left(\sigma_j\left(\frac{\partial \bar{z}^{j}}{\partial \zeta^ {\beta}}d \zeta^ {\beta}+\frac{\partial \bar{z}^{j}}{\partial \bar{\zeta}^ {\beta}} d \bar{\zeta}^ {\beta}\right)\right)
			:=\sigma_j \left(\frac{\partial \bar{z}^{j}}{\partial \bar{\zeta}^ {\beta}} d \bar{\zeta}^ {\beta}\right)\\
			&=\sigma_j\left({\left(\1-\overline\varphi\varphi\right)}^{-1}\right)^{\overline j}_{\bar k}\left(d\overline{z}^{k}+\overline{\varphi(t)}\lc
			d\overline{z}^{k}\right),\\
		\end{aligned}
	\end{equation}
	where $\varphi \overline{\varphi} =
	\overline{\varphi} \lc \varphi, \overline{\varphi} \varphi = \varphi
	\lc \overline{\varphi}$ and $\1$ is the identity matrix.
	Since $\varphi(t)$ is a well-defined, global $(1,0)$ vector valued $(0,1)$-form on $X_0$, $\mathcal P_\v$ is globally well-defined thanks to the last equality in (\ref{prooooo}) from \cite[(2.18)]{RZ18}.  
    
    We then use  the exponential operator
	$e^{\iota_\v}$ and the projection operator $\mathcal P_\v$ to define the extension operator $$\rho_\varphi: A^{p,q}(X_0)\longrightarrow A^{p,q}(X_t).$$ Concretely,  for $\sigma\in A^{p,q}(X_0)$ with the local expression $$\sigma=\sigma_{i_1\cdots i_p {j}_1\cdots {j}_{q}}d z^{i_{1}} \wedge \cdots \wedge d z^{i_{p}} \wedge d \bar{z}^{j_{1}} \wedge \cdots \wedge d \bar{z}^{j_{q}},$$ one has
	\begin{equation}\label{ccccc}
		\begin{aligned}
			\rho_\varphi\left(\sigma\right)
			&:=\sigma_{i_1\cdots i_p {j}_1\cdots {j}_{q}}e^{\iota_\varphi}\left(d z^{i_{1}} \wedge \cdots \wedge d z^{i_{p}}\right) \wedge \mathcal P_\v \left(\dz^{j_1}\right)\wedge\cdots\wedge \mathcal P_\v \left(\dz^{j_q}\right)\\
			&\;= \sigma_{i_1\cdots i_p {j}_1\cdots {j}_{q}}\left(d z^{i_{1}}+{\varphi(t)}\l  d z^{i_{1}}\right) \wedge \cdots \wedge\left(d z^{i_{p}}+{\varphi(t)}\l  d z^{i_{p}}\right)\\&\quad\, \wedge 	
			\lk \lk \1- \overline{\varphi} \varphi \rk^{\!-1}
			\rk^{ {j}_{1}}_{\bar k} \left(d\overline{z}^{k}+\overline{\varphi(t)}\lc
			d\overline{z}^{k}\right)\wedge\cdots\wedge\lk \lk \1- \overline{\varphi} \varphi \rk^{\!-1}
			\rk^{ {j}_{q} }_{\bar k} \left(d\overline{z}^{k}+\overline{\varphi(t)}\lc
			d\overline{z}^{k}\right).
		\end{aligned}
	\end{equation}
	Likewise, $\rho_\varphi$ is also globally well-defined.
	
	In \cite{ZR15,RZ18}, Zhao--the first author  introduced an extension map
	$$\label{RZ}
	e^{\iota_{\varphi}|\iota_{\overline{\varphi}}}:
	A^{p,q}(X_0)\longrightarrow A^{p,q}(X_t),
	$$ which can preserve all $(p,q)$-forms and played an important role in their subsequent papers \cite{RWZ19,RWZ21}.  More precisely, with the above notations, for $\sigma\in A^{p,q}(X_0)$, one defines \cite[Definition 2.8]{RZ18}:
	\begin{equation}\label{explicit}
		\begin{aligned}
			e^{\iota_{\varphi}|\iota_{\overline{\varphi}}}(\sigma)
			&=\sigma_{i_1\cdots i_p {j}_1\cdots {j}_{q}} e^{\iota_\varphi}\left(d z^{i_{1}} \wedge \cdots \wedge dz^{i_{p}}\right) \wedge e^{\iota_{\overline{\varphi}}}\left(\dz^{j_{1}} \wedge \cdots \wedge \dz^{j_q}\right)\\
			&=\sigma_{i_1\cdots i_p {j}_1\cdots {j}_{q}}\left(d z^{i_{1}}+{\varphi(t)}\l  d z^{i_{1}}\right) \wedge \cdots \wedge\left(d z^{i_{ p}}+{\varphi(t)}\l d z^{i_{p}}\right)\\&\quad\, \wedge
			\left(\dz^{j_{1}}+{\overline{\varphi(t)}}\l \dz^{j_{1}}\right) \wedge \cdots \wedge\left(\dz^{j_{q}}+{\overline{\varphi(t)}}\l \dz^{j_{q}}\right).\\
		\end{aligned}
	\end{equation}

	In \cite[Section 2.1]{RWZ21}, the first author--X. Wan--Q. Zhao introduced a new notation $\Finv$ to denote the
	simultaneous contraction on each component of a complex
	differential form. \footnote{For example,  for $\sigma\in A^{p,q}(X_0)$ with the above notation,
		$(\1-\b{\varphi}\varphi+\b{\varphi})\Finv\sigma$ means:
		$$
		\begin{aligned}
			\begin{split}
				&(\1-\b{\varphi}\varphi+\b{\varphi})\Finv(\sigma_{i_1\cdots
					i_p {j_1}\cdots {j_q}}dz^{i_1}\wedge\cdots \wedge dz^{p_p}\wedge
				d\b{z}^{j_1}\wedge\cdots \wedge d\b{z}^{j_q})
				\\
				=&\sigma_{i_1\cdots
					i_p {j_1}\cdots {j_q}}(\1-\b{\varphi}\varphi+\b{\varphi})\l
				dz^{i_1} \wedge\cdots\wedge(\1-\b{\varphi}\varphi +\b{\varphi})\l
				dz^{i_p}	\\&\qquad\quad\quad\wedge(\1-\b{\varphi}\varphi+\b{\varphi})\l
				d\b{z}^{j_1}\wedge\cdots\wedge(\1-\b{\varphi}\varphi+\b{\varphi})\l
				d\b{z}^{j_q}.
			\end{split}
		\end{aligned}
		$$}
	By comparing (\ref{ccccc}) and (\ref{explicit}), we get the relationship between the  two extension maps:
	\begin{equation}\label{sssds}
		\rho_{\varphi}=e^{\iota_{\varphi}|\iota_{\overline{\varphi}}}\circ{\left(\1-\overline\varphi\varphi\right)}^{-1}\Finv.
	\end{equation}
	the first author--Zhao obtained the $\bp$-extension obstruction for $(p,q)$-forms of the smooth family via
	the extension map $e^{\iota_{\varphi}|\iota_{\overline{\varphi}}}$ (cf. \cite[Proposition 2.13]{RZ18}), namely
	\begin{equation}\label{aaaaaa}
		\b{\p}_t(e^{\iota_{\varphi}|\iota_{\b{\varphi}}}(\sigma))
		=e^{\iota_{\varphi}|\iota_{\b{\varphi}}}((\1-\b{\varphi}\varphi)^{-1}\Finv([\p,\iota_{\varphi}]+\b{\p})(\1-\b{\varphi}\varphi)\Finv\sigma),
	\end{equation}
	where $\sigma\in A^{p,q}(X_0)$.
	One concludes that (\ref{aaaaaa}) is equivalent to the following assertion due to (\ref{sssds}):
	$${\rho_{\varphi}}^{-1} \b{\p}_t{\rho_{\varphi}}=\bp-\mathcal{L}^{1,0}_{\varphi(t)} \qquad\text{on}\,\, A^{p,q}(X_0),$$
	which is exactly \cite[Theorem 2.9, the case of $E=\Omega^{p}$]{Xi21b}.
	\begin{remark}
		In \cite{Tu21}, J. Tu used the pair deformation $\left\{(X_t,E_t)\right\}$ to give a correspondence between $E_0$-valued $(p,q)$-forms on $X_0$ and $E_t$-valued $(p,q)$-forms on $X_t$ by
		$$P_t: A^{p,q}(X_0, E_0)\longrightarrow A^{p,q}(X_t,E_t).$$ Recall that the  \textit{pair deformation} $\left\{\left(X_{t}, E_{t}\right)\right\}$ is a holomorphic family of pairs $\left\{\left(X_{t}, E_{t}\right)\right\}$ where each $E_{t}$ is a holomorphic vector bundle over a compact complex manifold $X_{t} .$ The holomorphic structures of $E_{t}$ and complex structures of $X_{t}$ vary simultaneously. In the case when each $E_t$ is the trivial line bundle, $P_t$ coincides with  (\ref{sssds}) and, therefore \cite[Theorem 2]{Tu21} is  equivalent to  \cite[Proposition 2.13]{RZ18}.
	\end{remark}
	
	The following proposition plays a key role in this paper:
	
	\begin{proposition}[{\cite[Theorem 3.4]{LRY15}, \cite[Proposition 2.2]{RZ18}}]\label{main1}
	Let $\phi\in A^{0,1}(X,T^{1,0}_X)$ on a complex manifold $X$. Then on $A^{*,*}(X)$,
	\begin{equation}\label{ext-old}
		e^{-\iota_{\phi}} \circ d \circ e^{\iota_{\phi}}
		= d - \mathcal{L}_\phi^{1,0} + \iota_{\bar{\partial}\phi - \frac{1}{2}[\phi,\phi]},
	\end{equation}
	where $\mathcal{L}_\phi^{1,0} := \iota_\phi \partial - \partial \iota_\phi$ is the Lie derivative.
	
	In particular, if $\varphi$ is an integrable Beltrami differential (i.e., $\bar{\partial}\varphi = \frac{1}{2}[\varphi,\varphi]$), then the last term in \eqref{ext-old} vanishes and we get
	\begin{equation}\label{ssss}
		e^{-\iota_{\varphi}} \circ d \circ e^{\iota_{\varphi}} = d - \mathcal{L}_\varphi^{1,0} = d + \partial \iota_\varphi - \iota_\varphi \partial.
	\end{equation}
	\end{proposition}
	
	From the proof of Proposition~\ref{main1}, we see that \eqref{ext-old} naturally generalizes the Tian--Todorov Lemma \cite{Ti87,To89},
	whose variants appear in \cite{Fr91,BK98,LSY09,Cle05} and also \cite{LR12,LRY15} for vector bundle-valued forms.
\begin{lemma}\label{aaaa}
	Let $\phi, \psi\in A^{0,1}(X,T^{1,0}_X)$ and $\alpha\in A^{*,*}(X)$ on an $n$-dimensional complex
	manifold $X$. Then
	\begin{equation}\label{f1}
		[\phi,\psi]\lrcorner\alpha
		= -\p(\psi\lrcorner(\phi\lrcorner\alpha))
		-\psi\lrcorner(\phi \lrcorner\p\alpha)
		+\phi\lrcorner\p(\psi\lrcorner\alpha)
		+\psi\lrcorner\p(\phi\lrcorner\alpha),
	\end{equation}
	where
	\[
	[\phi,\psi]
	:=\sum_{i,j=1}^n\big(\phi^i\wedge\partial_i\psi^j
	+\psi^i\wedge\partial_i{\phi}^j\big)\otimes\partial_j
	\]
	for $\phi=\sum_{i}\phi^i\otimes\partial_i$
	and $\psi=\sum_{i}\psi^i\otimes\partial_i$. Here $\partial_i:=\frac{\partial\ }{\partial z^i}$ and similar for others. 
\end{lemma}

\begin{remark} As shown in (the proof of) \cite[Lemma 3.2]{LRY15},
Equation~\eqref{f1} means that the Lie derivative $\mathcal{L}^{1,0}_\phi$ and the contraction $\iota_\psi$ do not commute; their difference is precisely the contraction with the Lie bracket $[\phi,\psi]$, i.e., $[\mathcal{L}^{1,0}_\phi,\iota_\psi]=\iota_{[\phi,\psi]}.$
\end{remark}
	
	\subsection{Kuranishi family}\label{subsection-Kuranishi family}
	We introduce some basics on Kuranishi family of complex structures in this subsection originally from \cite{Ku64}.
	
	By (the proof of) Kuranishi's completeness theorem \cite{Ku64}, for any compact complex manifold $X_0$, there exists a complete holomorphic family
	$\varpi:\mc{K}\rightarrow T$ of complex manifolds at the reference point $0\in T$ in the sense that for any differentiable family $\pi:\mc{X}\rightarrow B$ with $\pi^{-1}(s_0)=\varpi^{-1}(0)=X_0$, there exist a sufficiently small neighborhood $E\subseteq B$ of $s_0$, and smooth maps $\Phi: \mathcal {X}_E\rightarrow \mathcal {K}$,  $\tau: E\rightarrow T$ with $\tau(s_0)=0$ such that the diagram commutes
	$$\xymatrix{\mathcal {X}_E \ar[r]^{\Phi}\ar[d]_\pi& \mathcal {K}\ar[d]^\varpi\\
		(E,s_0)\ar[r]^{\tau}  & (T,0),}$$
	$\Phi$ maps $\pi^{-1}(s)$ biholomorphically onto $\varpi^{-1}(\tau(s))$ for each $s\in E$, and $$\Phi: \pi^{-1}(s_0)=X_0\longrightarrow \varpi^{-1}(0)=X_0$$ is the identity map.
This family is called the \textit{Kuranishi family} and constructed as follows. 
    
    Let $\{\eta_\nu\}_{\nu=1}^m$ be a  base for the harmonic space $\mathbb{H}^{0,1}(X_0,T^{1,0}_{X_0})$, where some suitable Hermitian metric is fixed on $X_0$ and $m\geq 1$; Otherwise the complex manifold $X_0$ would be  \textit{rigid}, i.e., for any differentiable family $\phi:\mc{M}\to P$ with $s_0\in P$ and $\phi^{-1}(s_0)=X_0$, there is a neighborhood $V \subseteq P$ of $s_0$ such that $\phi:\phi^{-1}(V)\to V$ is trivial. Then one can construct a holomorphic family
	$$\label{phi-ps-pp}\varphi(t) = \sum_{|I|=1}^{\infty}\varphi_{I}t^I:=\sum_{j=1}^{\infty}\varphi_j(t),\ I=(i_1,\cdots,i_m),\ t=(t_1,\cdots,t_m)\in \mathbb{C}^m,$$ {for small $t$,} of Beltrami differentials as follows:
	$$\label{phi-ps-0}
	\varphi_1(t)=\sum_{\nu=1}^{m}t_\nu\eta_\nu
	$$
	and for $|I|\geq 2$,
	$$\label{phi-ps}
	\varphi_I=\frac{1}{2}\db^*\G\sum_{J+L=I}[\varphi_J,\varphi_{L}].
	$$
	
Clearly, $\varphi(t)$ satisfies the equation
	$$\varphi(t)=\varphi_1+\frac{1}{2}\db^*\G[\varphi(t),\varphi(t)].$$
	Let
	$$T=\{t\ |\ \mathbb{H}[\varphi(t),\varphi(t)]=0 \}$$
	where $\mathbb H$ is the harmonic projection.
	So for each $t\in T$, $\varphi(t)$ satisfies
	\begin{equation}\label{int}
		\b{\p}\varphi(t)=\frac{1}{2}[\varphi(t),\varphi(t)],
	\end{equation}
	and determines a complex structure $X_t$ on the underlying differentiable manifold of $X_0$. More importantly, $\varphi(t)$ represents the complete holomorphic family $\varpi:\mc{K}\rightarrow T$ of complex manifolds. Roughly speaking, the Kuranishi family $\varpi:\mc{K}\rightarrow T$ contains all small differentiable deformations of $X_0$.
	
	\subsection{$d$-closed extension of $(p,q)$-form: Bott--Chern  approach}\label{important}
	
	Assume that a compact complex manifold $X_0$ satisfies $\mathbb B^{p,q+1}.$ Now, considering a Kuranishi family of deformations of $X_0$ over $\Delta_\epsilon$ for small $\epsilon$, one can find a family of integrable Beltrami differentials $\left\{\varphi(t)\right\}_{|t|\leq\epsilon} $  depending on $t$ holomorphically and describing the variations of complex structure on $X_0$.
	
	In this subsection, we will reprove Theorem \ref{main thm} due to \cite{WZ20}, from the perspective of  Bott--Chern Hodge theory.
	\begin{proof}	
		The proof will be divided into four steps.
		\begin{step}\label{step 1}
			{\rm\textbf{Transfer to the certain differential equations.}}\end{step}
By  equation \eqref{ssss} in Proposition \ref{main1}, for any $\mu \in A^{p,q}(X_0)$, the form $e^{\iota_{\varphi}}(\mu)$ is $d$-closed on $X_0$ (or $X_t$) if and only if
\[
\bigl(d + \partial \iota_{\varphi} - \iota_{\varphi} \partial\bigr) \mu = 0.
\]
		By comparing the types, one knows that the above equation amounts to
		\beq\label{!!}
		\begin{cases}
			\p\mu=0,\\
			\bp\mu=-\p(\varphi\l\mu).\\
		\end{cases}
		\eeq
		\begin{step}\label{step 2}
			{\rm\textbf{Study the integral equation.}}\end{step}
		Given any $d$-closed $(p,q)$-form $\mu_0\in A^{p,q}(X_0)$, one studies the integral equation as follows:
		\beq\label{integral equation}
		\mu+\p{(\p\bp)}^*\G_{\mathrm {BC}}\p({\varphi(t)}\l\mu)=\mu_{0}.
		\eeq
		In the sequel we will prove that the equation (\ref{integral equation}) possesses a unique solution $\mu(z,t)\in A^{p,q}(X_0)$ which is smooth  in $(z,t)$ and holomorphic in  $t\in \Delta_\epsilon$.
		
		Denote the completion of the norm space $\left(A^{p, q}(X_0),\|\cdot\|_{k, \alpha}\right)$ by $\mathscr{E}$, and consider the linear operator $$\mathcal{Q}_{\varphi(t)}:=-\p{(\p\bp)}^*\G_{\mathrm {BC}}\p \iota_{\varphi(t)}$$
		on $\mathscr{E}$. Then (\ref{integral equation}) is equivalent to the following equation
		\begin{equation}\label{new equation}
			(\mathrm{I}-\mathcal Q_{\v (t)})\mu=\mu_0.
		\end{equation}
		One can easily see that $\mathcal Q_{\v (t)}$  satisfies $\|\mathcal Q_{\v (t)}\|<1$ on $\mathscr{E}$ due to  the standard estimates for Green's operator $\G_{\mathrm {BC}}$ as $t\in \Delta_\epsilon$. We then apply \cite[Chapter \uppercase\expandafter{\romannumeral2}.1, Theorem 2]{Yo80} to $T= \mathrm{I}-\mathcal Q_{\v (t)}$ to obtain the unique solution
		\begin{equation}\label{zzz}
			\mu(z,t)={(\mathrm{I}-\mathcal Q_{\v (t)})}^{-1}\mu_{0}=\mu_{0}+\sum^{\infty}_{k=1}\mathcal Q_{\v (t)}^k(\mu_{0})
		\end{equation}
		of (\ref{new equation}) for any $|t|\leq\epsilon.$
		
		The integrable Beltrami differential $\left\{\varphi(t)\right\}_{|t|\leq\epsilon} $  can be written as a convergent power series in $t$ since it depends on $t$ holomorphically.  Thus, $\mu(z,t)$ is also a power series in $t$ by (\ref{zzz}). Moreover, it is easy to verify by the standard elliptic estimates and \eqref{new equation} that the $\|\cdot\|_{k, \alpha}$-norm of $\mu(z,t)$ is finite which implies that $\mu(z,t)$ is convergent for $|t|\leq\epsilon$. So $\mu(z,t)$ is holomorphic in $t$ for $|t|\leq\epsilon.$
		Since we have got the $C^k$-continuity of $\mu(z,t)$ from above, one can use the standard regularity theory for elliptic differential operator such as in \cite[Proposition 2.6 of Chapter 4]{MK71} to obtain that the unique solution $\mu(z,t)$ satisfying (\ref{integral equation}) is a smooth $(p,q)$-form in $(z,t)$ for small $t$.
		\begin{step}\label{step 3}
			{\rm\textbf{Show the solution $\mu(z,t)$ obtained in Step \ref{step 2} satisfies (\ref{!!})}}.\end{step}
		The unique solution $\mu(z,t)$ obtained in Step \ref{step 2} clearly satisfies the first equality of (\ref{!!}) and,  therefore we just need to check the second equality.
		
		We can locally express $\varphi(t)$ and  $\mu(z,t)$ as $\sum_{i\geq 1}\varphi_i t^i$ and $\sum_{i\geq 0}\mu_i t^i$, respectively, since they are both holomorphic in small $t$. Then the equation (\ref{int}) is equivalent to
		\begin{equation}
			\begin{cases}\label{closedness varphi1}
				\bp \v_1=0,\\
				\bar{\partial} \varphi_{k}=\mathop\sum\limits_{i+j=k \atop i, j \geq 1} \frac{1}{2}\left[\varphi_{i}, \varphi_{j}\right], \,\,\, k\geq 2.
			\end{cases}
		\end{equation}
		
		Rewrite equation (\ref{integral equation}) as
\[
\mu_k = -\partial (\partial \overline{\partial})^* \G_{\mathrm{BC}} \partial \Bigl( \sum_{\substack{i+j=k \\ i \geq 1, j \geq 0}} \varphi_i \lrcorner \mu_j \Bigr), \quad k \geq 1.
\]
		We aim to show  that
	\begin{equation}\label{solve}
		\bp \mu_k = -\p \Bigl( \sum_{\substack{i+j=k \\ i \geq 1, j \geq 0}} \varphi_i \lrcorner \mu_j \Bigr), \quad k \geq 1.
	\end{equation}
		When $k=1$, one has by \eqref{closedness varphi1}
		$$\label{aaa}
		\begin{aligned}
			\bp\p(\v_1\l\mu_0)
			&=-\p(\bp\v_1\l\mu_0+\v_1\l\bp\mu_0)=0.\\
		\end{aligned}
		$$
		Since $X$ satisfies $\mb B^{p,q+1}$, there exists some $\beta_1$ such that $\p(\v_1\l\mu_0)=\p\bp \beta_1$, implying $\H_{\mathrm {BC}}\p(\v_1\l\mu_0)=0$ because of (\ref{bc-ker}). We then obtain
		\begin{equation}\label{dbarmu1=0}
			\begin{aligned}
				\bp\mu_1&={(\p\bp)\left(\p\bp\right)}^*\G_{\mathrm {BC}}\p(\v_1\l\mu_{0})\\
				&=\G_{\mathrm {BC}}{(\p\bp)\left(\p\bp\right)}^*\p(\v_1\l\mu_{0})=\G_{\mathrm {BC}}\square_{{\mathrm {BC}}}\p(\v_1\l\mu_{0})\\&=\p(\v_1\l\mu_{0})-\H_{\mathrm {BC}}\p(\v_1\l\mu_0)=\p(\v_1\l\mu_0).\\
			\end{aligned}
		\end{equation}
		Here in (\ref{dbarmu1=0}), the second, third and fourth  equalities follow from the observation (\ref{case-2}),  (\ref{bc-Lap}) together with (\ref{closedness varphi1})  and  (\ref{identity}), respectively.
        
		Now, by induction hypothesis, assume that (\ref{solve}) holds for $k\leq {l-1}$, and then similarly for  $k=l$, one has
	\begin{align*}
		\bp \mu_k
		&= \partial \bp (\partial \overline{\partial})^{*} \G_{\mathrm{BC}} \partial \Bigl( \sum_{\substack{i+j=k \\ i \geq 1, j \geq 0}} \varphi_i \lrcorner \mu_j \Bigr) \\
		&= \partial \Bigl( \sum_{\substack{i+j=k \\ i \geq 1, j \geq 0}} \varphi_i \lrcorner \mu_j \Bigr)
		- \H_{\mathrm{BC}} \partial \Bigl( \sum_{\substack{i+j=k \\ i \geq 1, j \geq 0}} \varphi_i \lrcorner \mu_j \Bigr).
	\end{align*}
		Hence, it suffices to show
		\begin{equation}\label{harmonic-partial}
		\H_{\mathrm{BC}} \partial \Bigl( \sum_{\substack{i+j=k \\ i \geq 1, j \geq 0}} \varphi_i \lrcorner \mu_j \Bigr)=0.
		\end{equation} By explicit computation, we have
	\begin{align*}
		\bp \partial \Bigl( \sum_{\substack{i+j=k \\ i \geq 1, j \geq 0}} \varphi_i \lrcorner \mu_j \Bigr)
		&= -\partial \bp \Bigl( \sum_{\substack{i+j=k \\ i \geq 1, j \geq 0}} \varphi_i \lrcorner \mu_j \Bigr)= -\partial \sum_{\substack{i+j=l \\ i \geq 1, j \geq 0}} \bigl( \bp \varphi_i \lrcorner \mu_j + \varphi_i \lrcorner \bp \mu_j \bigr) \\
		&= -\partial \sum_{\substack{i+j=l \\ i \geq 1, j \geq 0}} \Bigl( \frac{1}{2} \sum_{a+b=i} [\varphi_a, \varphi_b] \lrcorner \mu_j
		- \sum_{a+b=j} \varphi_i \lrcorner \partial (\varphi_a \lrcorner \mu_b) \Bigr) \\
		&= -\partial \sum_{a+b+c=l} \left( - \partial (\varphi_a \lrcorner \varphi_b \lrcorner \mu_c) \right) = 0,
	\end{align*}
		where the last but one equality comes from Lemma \ref{aaaa} plus the $\p$-exactness of $\mu_j$.
		By applying the same method as before, one obtains \eqref{harmonic-partial}.  Thus, the solution $\mu(z,t)$ obtained in Step \ref{step 2} satisfies (\ref{!!}), i.e., it satisfies $d(e^{\iota_\varphi}(\mu(z,t)))=0$ on $X_0$ (or $X_t$).
		\begin{step}\label{step 4}
			{\rm\textbf{Extend $\mu(z,t)$ by $\rho_\v$}}.\end{step}
		Now, as $e^{\iota_\varphi}(\mu(z,t))$ is $d$-closed in $F^{p}A^{p+q}(X_t)$, one writes
		\begin{equation*}\label{decomposition}
			\begin{aligned}
				e^{\iota_\varphi}(\mu(z,t)) &=\alpha_t^{p, q}+\alpha_t^{p+1, q-1}+\cdots+\alpha_t^{p+q, 0} \\
				& \in A^{p, q}\left(X_{t}\right) \oplus A^{p+1, q-1}\left(X_{t}\right) \oplus \cdots \oplus A^{p+q, 0}\left(X_{t}\right),
			\end{aligned}
		\end{equation*}
		and gets $\alpha_t^{p, q}=\rho_\varphi(\mu(z,t)).$ Since $d\left(e^{\iota_\varphi}(\mu(z,t))\right)=0$ and $d=\partial_{t}+\bp_{t}$, we have $\bar{\partial}_{t}\left(\alpha_t^{p, q}\right)=0$ by comparing the types.
		
		The proof of  Theorem   \ref{main thm} is completed.
	\end{proof}

	\begin{remark}
		Notice that in \cite{WZ20},  Wei--Zhu  first obtained Theorem   \ref{main thm}  by use of $\p\bp$-Laplacian $\square_{\p\bp}$. The integral equation what they study  therein is
		\begin{equation}\label{ggg}
			\mu=\mu_0-\p{(\p\bp)}^*\G_{\p\bp}\p(\iota_{\varphi(t)}\mu),
		\end{equation}
		where  $\mu_0\in A^{p,q}(X_0)$ is the given $d$-closed $(p,q)$-form and $\G_{\p\bp}$ is the associated Green's operator of the $4^\textit{th}$ order real elliptic differential operator $\square_{\p\bp}$ defined by
		$$\square_{\p\bp}=\p\bp(\p\bp)^*+(\p\bp)^*(\p\bp).$$ In addition, they used Banach fixed point theorem to get the unique solution of (\ref{ggg}) motivated by \cite{LZ18,LZ20}. Our bounded linear operator $\mathcal{Q}_{\varphi(t)}$ in Step \ref{step 2} coincides with their contraction mapping.
	\end{remark}

	\section{Deformation invariance of Hodge numbers}
	\label{section 4}
	On deformation invariance of Hodge numbers, one only needs to consider a Kuranishi family of deformations
	$\pi:\mathcal{X}\rightarrow \Delta_{\epsilon}$  of $n$-dimensional compact complex manifolds over a small complex disk
	with the general fibers $X_t:=\pi^{-1}(t)$ throughout this section, and always fixes a Hermitian metric $g$ on the central fiber $X_0$.

	\subsection{Some examples}\label{basic strategy}
	
In this subsection,  we will state the following three examples that the deformation invariance of the $(p,q)$-Hodge number fails when one of the three conditions in Theorem \ref{pq def} is not satisfied, while the other two hold, with the help of Hodge,
	Bott--Chern and Aeppli numbers of manifolds in the Kuranishi family of the Iwasawa manifold (see \cite[Appendix]{Ag13} and \cite[Section 1.c]{Sc07}). It implies that the three
	conditions  in Theorem  \ref{pq def}  cannot be dropped for the sake of the validity of a theorem  concerning the deformation invariance
	of all the $(p,q)$-Hodge numbers.

    \vspace{0.4em}
	Let $\mathbb{I}_3$ be the Iwasawa manifold of complex dimension $3$ with $\varphi^1,\varphi^2,\varphi^3$
	denoted by the basis of the holomorphic one form $H^0(\mathbb{I}_3,\Omega^1)$ of $\mathbb{I}_3$, satisfying
	the relation \[ d\varphi^1 =0,\ d \varphi^2=0,\ d\varphi^3 = - \varphi^1 \w \varphi^2. \]
	And the convention $\varphi^{12\bar{1}\bar{2}\bar{3}}:= \varphi^1 \w \varphi^2 \w \overline{\varphi}^1\w \overline{\varphi}^2\w \overline{\varphi}^3$
	will be used for simplicity. In order to reduce the number of cases under consideration, notice that, on a compact complex Hermitian manifold $X$ of complex dimension $n$, for any integer $0\leq{p,q}\leq n$, the Hodge-$*$-operator induces an isomorphism between
	$$
	H_{\bar{\partial}}^{p, q}(X) \stackrel{\simeq}{\longrightarrow} H_{\partial}^{n-q, n-p}(X) \simeq \overline{H_{\bar{\partial}}^{n-p, n-q}(X)}.
	$$
	\begin{example}[$(p,q)=(1,0)$]\label{example 1}
		The injectivity of $\iota^{1,1}_{\mathrm {BC},\p}$ holds on $\mathbb{I}_3$
		with the deformation invariance of $h^{1,-1}_{\db_t}(X_t)$ trivially established but
		$\iota^{1,0}_{\mathrm {BC},\bp}$ is not surjective. In this case, $h^{1,0}_{\db_t}(X_t)$ are deformation variant.
	\end{example}
	\begin{proof}
		\cite[Appendix]{Ag13} presents that $h_{\mathrm {BC}}^{1,1}=4, h_{\p}^{1,1}=6$, and $h_{\mathrm {BC}}^{1,0}=2, h_{\bp}^{1,0}=3$. More precisely, one has by \cite[p.   6]{Sc07}:
		\[ \begin{aligned}
			H^{1,1}_{\mathrm {BC}}(X) &= \langle[\varphi^{1\bar{1}}]_{\mathrm {BC}}, [\varphi^{1\bar{2}}]_{\mathrm {BC}},[\varphi^{2\bar{1}}]_{\mathrm {BC}},[\varphi^{2\bar{2}}]_{\mathrm {BC}} \rangle, \\
			H^{1,1}_{\p}(X)\simeq{\overline{H_{\bp}^{1,1}(X)}} &= \langle[\varphi^{1\bar{1}}]_{\bp}, [\varphi^{1\bar{2}}]_{\bp},[\varphi^{2\bar{1}}]_{\bp},[\varphi^{2\bar{2}}]_{\bp}, [\varphi^{1\bar{3}}]_{\bp},[\varphi^{2\bar{3}}]_{\bp}\rangle, \\
		\end{aligned}\]
		which gives the injectivity of $\iota^{1,1}_{\mathrm {BC},\p}$. However, according to \cite[Last para. in Section 1.c]{Sc07}, $\iota^{1,0}_{\mathrm {BC},\bp}$ is not surjective. The deformation variance of $h^{1,0}_{\db_t}(X_t)$ follows from \cite[Appendix]{Ag13}.
	\end{proof}
	\begin{example}[$(p,q)=(2,0)$]\label{example 2}
		The surjectivity of $\iota^{2,0}_{\mathrm {BC},\bp}$ holds on $\mathbb{I}_3$
		with the deformation invariance of $h^{2,-1}_{\db_t}(X_t)$ trivially established but
		$\iota^{2,1}_{\mathrm {BC},\p}$ is not injective. In this case, $h^{2,0}_{\db_t}(X_t)$ are deformation variant.
	\end{example}
	\begin{proof}
		By \cite[p. 6]{Sc07}, we get
		\[ \begin{aligned}
			H^{2,0}_{\mathrm {BC}}(X) &= \langle[\varphi^{12}]_{\mathrm {BC}}, [\varphi^{13}]_{\mathrm {BC}},[\varphi^{23}]_{\mathrm {BC}} \rangle, \quad
			H^{2,0}_{\bp}(X)= \langle[\varphi^{12}]_{\bp}, [\varphi^{13}]_{\bp},[\varphi^{23}]_{\bp}\rangle, \\
		\end{aligned}\]
		which implies the surjectivity of $\iota^{2,0}_{\mathrm {BC},\bp}$. However, the following fact shows the non-injectivity of 	$\iota^{2,1}_{\mathrm {BC},\p}$:
		\[ \begin{aligned}
			H^{2,1}_{\mathrm {BC}}(X) &= \langle[\varphi^{12\bar{1}}]_{\mathrm {BC}}, [\varphi^{12\bar{2}}]_{\mathrm {BC}},[\varphi^{13\bar{1}}]_{\mathrm {BC}},[\varphi^{13\bar{2}}]_{\mathrm {BC}},[\varphi^{23\bar{1}}]_{\mathrm {BC}},[\varphi^{23\bar{2}}]_{\mathrm {BC}}\rangle, \\
			H^{2,1}_{\p}(X)\simeq{\overline{H_{\bp}^{1,2}(X)}} &= \langle[\varphi^{13\bar{1}}]_{\bp}, [\varphi^{13\bar{2}}]_{\bp},[\varphi^{13\bar{3}}]_{\bp},[\varphi^{23\bar{1}}]_{\bp}, [\varphi^{23\bar{2}}]_{\bp},[\varphi^{23\bar{3}}]_{\bp}\rangle. \\
		\end{aligned}\]
		Moreover, \cite[Appendix]{Ag13} gives the fact of the deformation variance of $h^{2,0}_{\db_t}(X_t).$
	\end{proof}
	\begin{example}[$(p,q)=(2,3)$]\label{example 3}
		The mapping $\iota^{2,3}_{\mathrm {BC},\bp}$ on $\mathbb{I}_3$ is surjective and  the injectivity of $\iota^{2,4}_{\mathrm {BC},\p}$  trivially holds,  but
		$h^{2,2}_{\db_t}(X_t)$ are deformation variant. In this case, $h^{2,3}_{\db_t}(X_t)$ are deformation variant.
	\end{example}
	\begin{proof}
		It is clear that $\iota^{2,3}_{\mathrm {BC},\bp}$ is surjective since by \cite[p. 6]{Sc07},
		\[ \begin{aligned}
			H^{2,3}_{\mathrm {BC}}(X) &= \langle[\varphi^{12\bar{1}\bar{2}\bar{3}}]_{\mathrm {BC}}, [\varphi^{13\bar{1}\bar{2}\bar{3}}]_{\mathrm {BC}},[\varphi^{23\bar{1}\bar{2}\bar{3}}]_{\mathrm {BC}} \rangle, \\
			H^{2,3}_{\bp}(X) &= \langle[\varphi^{12\bar{1}\bar{2}\bar{3}}]_{\bp}, [\varphi^{13\bar{1}\bar{2}\bar{3}}]_{\bp},[\varphi^{23\bar{1}\bar{2}\bar{3}}]_{\bp}\rangle. \\
		\end{aligned}\]
		The deformation variance of $h^{2,2}_{\db_t}(X_t)$ and $h^{2,3}_{\db_t}(X_t)$ can also be got from \cite[Appendix]{Ag13}.
	\end{proof}
	
	\begin{remark}\label{example remark}
		The three examples above partially coincide with those listed in \cite[Examples 3.2, 3.3, and 3.4]{RZ18}, where  the first author--Zhao  verified the indispensability of the three conditions therein to state a theorem  for the deformation invariance
		of all the $(p,q)$-Hodge numbers. More precisely, their three conditions are the injectivity of the mappings $\iota^{p+1,q}_{\mathrm {BC},\p}$, $\iota^{p,q+1}_{\db,A}$ and
		the deformation invariance of the $(p,q-1)$-Hodge number (see also our Remarks \ref{Remark 1}, \ref{Remark 2}).
	\end{remark}
	Certain special types, such as $(p,0)$ and $(0,q)$, may allow for a weakening of the conditions in Theorem~\ref{pq def}. 
	Hence, another two theorems follow, whose proofs will be given in Subsection \ref{p00q-section}. Note  that Theorem \ref{inv-0q} first appeared as \cite[Theorem 3.7]{RZ18}. So we can as well obtain \cite[Remark 3.8, Corollary 3.9]{RZ18}  as a consequence.
	\begin{theorem}\label{inv-p0}
		If $X_0$ satisfies $\mathbb{B}^{p,1}$ (i.e., the mapping $\iota_{\mathrm {BC},\p}^{p,1}$ is injective) and $\mathcal{S}^{p+1,0}$, then $h^{p,0}_{\db_t}(X_t)$ are deformation invariant.
	\end{theorem}
	
	\begin{theorem}\label{inv-0q}
		If $X_0$ satisfies $\mathcal {B}^{1,q}$ (i.e., the mapping $\iota^{0,q}_{\mathrm {BC},\db}$ is surjective)
		and $h^{0,q-1}_{\db_t}(X_t)$  satisfies the deformation invariance, then
		$h^{0,q}_{\db_t}(X_t)$ are independent of $t$.
	\end{theorem}
	
	\begin{remark}\label{researchgate}
		Drawing on the work of S. Console--A. Fino--Y.-S. Poon \cite{CFP16}, the first author--Zhao constructed an example, namely a holomorphic family of nilmanifolds of complex dimension $5$,
		whose central fiber is endowed with an abelian complex structure.
		This family admits the deformation invariance of the $(p,0)$-Hodge numbers for $1 \leq p \leq 5$,
		but not the $(1,1)$-Hodge number or $(1,1)$-Bott--Chern number. By verification, this example is also applicable to our Theorem \ref{inv-p0}, which shows the function of Theorem \ref{inv-p0} possibly beyond Kodaira--Spencer's squeeze \cite[Theorem 13]{KS60} in this case. One can see  Example 3.11  in the update version on the first author's researchgate \cite{RZ} for more details.
	\end{remark}

	\subsection{Proof of Theorem \ref{pq def}}\label{sbs-pq}
	We aim to prove Theorem \ref{pq def} in this subsection, which can be restated as the following theorem thanks to the equivalent statements \eqref{equ 1} and \eqref{equ3} displayed  in Subsection \ref{section Variations ofddbar}.
	
	\begin{theorem}\label{inv pq'''''''}
		If the central fiber $X_0$ satisfies both
		$\mathbb{B}^{p,q+1}$ and $\mathcal{B}^{p+1,q}$ with
		the deformation invariance of $h^{p,q-1}_{\db_t}(X_t)$ established,
		then $h^{p,q}_{\db_t}(X_t)$ are independent of $t$.
	\end{theorem}
	
	We have described our strategy in the Introduction \ref{Introduction}, so the next lemma, motivated by \cite[Lemma 3.1]{Po19}  is crucial, which can help us to achieve the first step. For the convenience of the readers, we recall the proof.
	\begin{lemma}[{\cite[Lemma 3.13]{RZ18}}]\label{lemma1}
		Assume that the compact complex manifold $X$ satisfies $\mathcal{B}^{p+1,q}$. Then each Dolbeault class $[\sigma]_{\db}$ of the type $(p,q)$ can be canonically
		represented by a uniquely-chosen $d$-closed $(p,q)$-form $\gamma_{\sigma}$.
	\end{lemma}
	\begin{proof}
		We first choose the unique harmonic representative $\sigma$ of $[\sigma]_{\bar{\partial}}$. The $d$-closed representative $\gamma_{\sigma} \in A^{p,q}(X)$ then satisfies
		\[
		\gamma_{\sigma} = \sigma + \bar{\partial} \beta_{\sigma}
		\]
		for some $\beta_{\sigma} \in A^{p,q-1}(X)$ solving
		\[
		\partial \bar{\partial} \beta_{\sigma} = - \partial \sigma.
		\]
		Existence follows from the assumption on $X$, and uniqueness with minimal $L^2$-norm is given by \cite[Lemma 3.12]{RZ18}, allowing us to take
		\[
		\beta_{\sigma} = - (\partial \bar{\partial})^{*} \mathbb{G}_{\mathrm{BC}} \partial \sigma.
		\]
	\end{proof}
	By Theorem   \ref{main thm}, one has:
	\begin{proposition}\label{thmpq}
		Assume that the compact complex manifold $X_0$ satisfies $\mathbb{B}^{p,q+1}$ and $\mathcal{B}^{p+1,q}$. Then for each Dolbeault class in $H^{p,q}_{\db}(X_0)$ with the unique canonical $d$-closed representative $\mu_0$
		given as Lemma \ref{lemma1}, we can get $\mu(z,t)\in A^{p,q}(X_0)$ with $\mu(z,0)=\mu_0(z)$, which is smooth in $(z,t)$ and holomorphic in $t$  for small $t$,  such that  $\rho_\v(\mu(z,t))$
		is $\bar{\partial}_{t}$-closed in $A^{p, q}\left(X_{t}\right).$
	\end{proposition}
	
	According to the strategy as we represented after  the statement of  Theorem \ref{main thm}, Theorem \ref{inv pq'''''''} follows immediately from the following result:
	\begin{proposition}\label{pq-pq-1}
		If the $\db$-extension of $H^{p,q}_{\db}(X_0)$
		as in Proposition \ref{thmpq} holds for a complex manifold $X_0$,
		then the deformation invariance of $h^{p,q-1}_{\db_t}(X_t)$
		assures that the extension map
		$$H^{p,q}_{\db}(X_0) \longrightarrow H^{p,q}_{\db_t}(X_t):[\mu_0]_{\db} \longmapsto [\rho_\v(\mu(z,t))]_{\db_t}$$
		is injective.
	\end{proposition}
	\begin{proof}
		Fix a smooth family of Hermitian metrics
		$\{h_t\}_{t \in \Delta_{\epsilon}}$ for the infinitesimal
		deformation $\pi: \mathcal{X} \rightarrow \Delta_{\epsilon}$ of
		$X_0$. So, if the Hodge numbers $h^{p,q-1}_{\db_t}(X_t)$ are deformation
		invariant, the Green's operator $\G_t$, acting on the
		$A^{p,q-1}(X_t)$, depends differentiably with respect to $t$ from
		\cite[Theorem 7]{KS60}. Applying this, one guarantees
		that this extension map can not send a non-zero class in
		$H^{p,q}_{\db}(X_0)$ to a zero class in $H^{p,q}_{\db_t}(X_t)$.
		
		If one supposes that
		$$\rho_\v(\mu(z,t))=\db_t \eta_t$$
		for some $\eta_t\in A^{p,q-1}(X_t)$ when $0\neq t \in {\Delta_{\epsilon}}$, the Hodge decomposition of $\db_t$ and the
		commutativity of $\G_t$ with $\dbs_t$ and $\db_t$ imply that
		\begin{align*}
			\rho_\v(\mu(z,t))
			&=\db_t \eta_t= \db_t\big(\mathbb{H}_t(\eta_t) + \square_t \mathbb{G}_t \eta_t\big)\\
			&=\db_t\big(\dbs_t\db_t\G_t \eta_t\big)=\db_t\G_t\big(\dbs_t\db_t \eta_t\big)\\
			&=\db_t\G_t\big(\dbs_t\rho_\v(\mu(z,t))\big),
		\end{align*}
		where $\mathbb{H}_t$ and $\square_t$ are the harmonic projectors and
		the Laplace operators with respect to $(X_t,\omega_t)$, respectively.
		Let $t$ converge to $0$ on both sides of the equality
		\[ \rho_\v(\mu(z,t))
		=
		\db_t\G_t\big(\dbs_t\rho_\v(\mu(z,t))\big),\]
		which implies that $\mu_0$ is $\db$-exact on the central fiber
		$X_0$. Here the Green's operator $\mathbb{G}_t$ depends
		differentiably with respect to $t$.
	\end{proof}
	\begin{remark}\label{Remark 1}
	In \cite[Theorem~1.3]{RZ18}, the first author--Zhao also proved an invariance theorem for the Hodge numbers $h^{p,q}_{\bp_t}(X_t)$. In both results, the deformation invariance of $h^{p,q-1}_{\db_t}(X_t)$ is required. Our assumption that $X_0$ satisfies $\mathbb{B}^{p,q+1}$ (resp.\ $\mathcal{B}^{p+1,q}$) is stronger (resp.\ weaker) than their assumption that $X_0$ satisfies $\mathbb{S}^{p,q+1}$ (resp.\ $\mathbb{B}^{p+1,q}$).
	\end{remark}

	\subsection{Proof of Theorem \ref{inv-p0-0q-intro}}\label{p00q-section}
	\begin{proof}[Proof of Theorem \ref{inv-p0}]
		Since any holomorphic $(p,0)$-form on $X_0$ satisfying $\mc S^{p+1,0}$ is actually $d$-closed and Proposition \ref{pq-pq-1} holds automatically in the $(p,0)$-case, one can complete the proof only by Proposition \ref{thmpq}.
	\end{proof}
	\begin{proof}[Proof of Theorem \ref{inv-0q}]
		Compared with Theorem \ref{inv pq'''''''}, this one drops the condition $\mathbb{B}^{0,q+1}$ since it automatically holds.
	\end{proof}
	
	\begin{remark}\label{Remark 2}
		In \cite[Theorem 3.18]{RZ18}, the first author--Zhao considered the invariance of
		Hodge numbers $h^{p,0}_{\bp_{t}}(X_t)$, where their condition `$\mb S^{p+1,0}$ (resp. $\mb S^{p,1}$) on $X_0$' is stronger (resp. weaker) than ours.
	\end{remark}

	\section{Local stabilities of $p$-K\"{a}hler structures}
	Local stabilities of complex structures are significant  in deformation theory of complex structures.  Alessandrini--Bassanelli \cite{AB90} proved that the $(n-1)$-K\"ahlerian property is not preserved under the small deformations for balanced manifolds, nor  for more general \textit{$p$-K\"ahler manifolds} (the basics of this concept are introduced in Appendix \ref{p}), where $p > 1$, while, motivated by the proof of \cite[Theorem 9.23]{Vo02}, we plan to reprove Theorem \ref{pkahler introduction} due to \cite{RWZ19} in this section. This will be based on the extension of closed complex differential forms discussed in Subsection \ref{important}, as well as  the deformation openness of the $\p\bp$-property.

	\subsection{Several useful results and classical proofs}\label{sus-introductionpkahler}

The $\partial\bar{\partial}$-property is a powerful tool in the study of K\"ahler manifolds, with applications including the solvability of the $\bar{\partial}$-equation $\bar{\partial}x = \partial\alpha$ for $\partial\alpha$ closed, and the degeneration of the Fr\"olicher spectral sequence.  As we know, there are many interesting alternative characterizations of the $\partial \bar{\partial}$-property on compact complex manifolds, among which are P. Deligne--P. Griffiths--J. Morgan--D. Sullivan's  one  \cite{DGMS75} by the decomposition type of Dolbeault complex  and D. Angella--A. Tomassini's one \cite{AT13} by Fr\"olicher-type equality in terms of Betti, Bott--Chern and Aeppli numbers.

Recall that  a complex manifold is \textit{$p$-K\"{a}hlerian} if it admits a  \textit{$p$-K\"{a}hler form}, i.e., a $d$-closed transverse $(p,p)$-form as in Definition \ref{pkf}. One notices that each $n$-dimensional complex manifold is $n$-K\"ahlerian and there are two basic properties of $p$-K\"ahlerian structures:
	\begin{lemma}[{\cite[Proposition 1.15]{AA87} or  \cite[Corollary 4.6]{RWZ21}}]\label{pkb}
		A complex manifold $M$ is $1$-K\"ahler if and only if $M$ is
		K\"ahler;  an $n$-dimensional complex manifold $M$ is $(n-1)$-K\"ahler if and only if $M$ is  {balanced}, i.e., it admits a real positive $(1,1)$-form $\omega$, such that $w^{n-1}$ is $d$-closed.
	\end{lemma}
	
	Thus, as a direct corollary of Theorem \ref{pkahler introduction}, one has:
	\begin{corollary}\label{cor-mt}
		Let $\pi: \mathcal{X} \rightarrow B$ be a differentiable family of compact complex manifolds.
		\begin{enumerate}[$(i)$]
			\item {{ {\textbf{$($\cite[Theorem 15]{KS60}$)$}}}\label{stab-Kahler}
				If the fiber $X_0:= \pi^{-1}(t_0)$ admits a K\"ahler metric, then for a sufficiently small neighborhood $U$ of $t_0$ on $B$, the fiber $X_t:=\pi^{-1}(t)$ over any point $t\in U$ still admits a K\"ahler metric, which coincides for $t=t_0$ with the given K\"ahler metric on $X_0$.}
			\item{$($\cite[Theorem 5.13]{Wu06}$)$} \label{balanced-Kahler}
			Let the fiber $X_0$ be a balanced manifold and satisfy the  $\p\db$-property. Then $X_t$ also admits a balanced metric for small $t$.
		\end{enumerate}
	\end{corollary}
	
	The first assertion of Corollary \ref{cor-mt} is the fundamental Kodaira--Spencer's local stability theorem of K\"ahler structures, and motivates the second one of Corollary \ref{cor-mt} and many other related works on local stabilities of
	complex structures in \cite{FY11,Vo02,Wu06,AU16,AU17,RWZ19}.

	Let us sketch Kodaira--Spencer's proof of local stability theorem
	\cite{KS60}. Let $\mathbb H_{\mathrm {BC},t}$ be the orthogonal projection to the kernel
	$\mathds{F}_{\mathrm {BC},t}$ of the Bott--Chern Laplacian
	\begin{equation*}
		\square_{{\mathrm {BC}},t}=\p_t\db_t\db_t^*\p_t^*+\db_t^*\p_t^*\p_t\db_t+\db_t^*\p_t\p_t^*\db_t+\p_t^*\db_t\db_t^*\p_t+\db_t^*\db_t+\p_t^*\p_t
	\end{equation*}
	and $\G_{\mathrm {BC},t}$ the corresponding Green's operator with respect to $\alpha_t$ on $X_t$.
	Here
	$$\alpha_t=\sqrt{-1}g_{i\overline{ j}}(\zeta,t)d \zeta^i \w d \overline{\zeta}^j$$
	is a Hermitian metric on $X_t$ depending differentiably on $t$ with $\alpha_0$ being a K\"ahler metric on $X_0$,
	and $\db_t^*$ (resp.  $\p_t^*$) is the dual of $\db_t$ (resp. $\p_t$) with respect to $\alpha_t$. Using a cohomological argument together with  the upper semi-continuity theorem, they prove that
	$\mathbb H_{\mathrm {BC},t}$ and $\G_{\mathrm {BC},t}$ depend differentiably on $t$. Then they can
	construct the desired K\"ahler metric on $X_t$ as
	$$\widetilde{\alpha_t}=\frac{1}{2}(\mb H_{\mathrm {BC},t}\alpha_t+\overline{\mb H_{\mathrm {BC},t}\alpha_t}).$$

	We next roughly state  Voisin's proof of local stability theorem \cite[Theorem 9.23]{Vo02}. Let $\pi: \mathcal{X} \rightarrow B$ ($0 \in B$) be a family of complex manifolds with K\"ahlerian $X=X_0$. One  puts a Hermitian metric on $\mathcal{X}$ to get  an induced Hermitian metric on each $X_{b}:=\pi^{-1}(b)$. We can identify $H^{1}\left(X_{b}, \Omega_{X_{b}}\right)$  with the forms of type $(0,1)$ with values in $\Omega_{X_{b}}$ which are harmonic for the Laplacian associated to the operator $\bp$.  By virtue of the deformation invariance of Hodge numbers \cite[Proposition 9.20]{Vo02} together with the deformation theory from \cite{MK71}, she proved the existence of a $\mathcal{C}^{\infty}$ section $\left(\omega_{b}\right)_{b \in B}, \omega_{0}=\omega$ ($d$-closed K\"ahler form) of the bundle $\Omega_{\mathcal{X} / B} \otimes \Omega_{\mathcal{X} / B}^{0,1}$ such that $\omega_{b}$ is $\bar{\partial}_b$-closed for every $b$ sufficiently near 0, where
	$$ \Omega_{\mathcal{X} / B}:=\Omega_{\mathcal{X}} / \pi^{*} \Omega_{B},$$
	and
	$$\Omega_{\mc X/B} ^{0,q}:=\bigwedge ^q\Omega_{\mathcal{X} / B}^{0,1},\quad \Omega_{\mathcal{X} / B}^{0,1}:=\Omega_\mc X^{0,1}/\pi^*\Omega_B^{0,1}.$$
	Finally, she got the desired K\"ahler form $\Re \omega_{b}^{\prime}$ from a $d$-closed $(1,1)$-form $\omega_{b}^{\prime}$ constructed on $X_b$ due to her (separate) usage of deformation openness of $\p\bp$-property \cite[Propositions 9.20, 9.21]{Vo02} and some uniform convergence arguments. 
    
    One should notice that the former  property was first showed by Voisin, while one realizes that the $\p\bp$-property on $X$ is equivalent to the Fr\"{o}licher (or Hodge to de Rham) spectral sequence degenerates at $E_1$ plus the property that the canonical filtration on $H^l (X,\mb C)$ gives a Hodge structure of weight $l$, for every $l\geq 0$. See \cite[Theorem 5.12]{Wu06} and \cite[Corollary 3.7]{AT13} for alternative proofs.

	It's worth noting that when using Voisin's approach to deal with the local stabilities of $p$-K\"{a}hler structures ($p>1$), one might encounter some obstacles.  Unlike the case of $1$-K\"ahler forms, a $p$-K\"ahler form $\omega_{0}$ on $X_0$ is not necessarily $\bp$-harmonic.  Hence,  given an initial $d$-closed  $\omega_{0}\in A^{p,p}(X_0)$, one first needs to find a $\bp$-harmonic representative which is much possibly different form $\omega_0$, and then imitates the subsequent actions. However, it appears that the transversality is likely to be lost in this process.
	
	Motivated by Voisin's proof,  we now briefly describe our method to prove local stabilities of $p$-K\"{a}hler structures, which is quite different from Kodaira--Spencer's. By means of the Kuranishi's completeness theorem  introduced in Subsection \ref{subsection-Kuranishi family}, we can reduce Theorem \ref{pkahler introduction}  to the  Kuranishi family  $\varpi:\mc{K}\to T$ by shrinking $E$ if necessary, that is, we will explicitly construct a  $p$-K\"{a}hler extension
	$\widetilde{\omega_t}$ of the $p$-K\"{a}hler form $\omega_0$ on $X_0$, such that $\widetilde{\omega_t}$ is a
	$p$-K\"{a}hler form on the general fiber $\varpi^{-1}(t)=X_t$. 
	
	In this method, the following observation plays a prominent role.
	\begin{proposition}[{\cite[Proposition 4.12]{RWZ21}}]\label{ext-trans}
		Let $\pi: \mathcal{X} \> B$ be a  differentiable family of compact complex $n$-dimensional manifolds and
		$\Omega_t$  a family of real $(p,p)$-forms with $p< n$, depending smoothly on $t$. Assume that
		$\Omega_0$ is a transverse $(p,p)$-form on $X_0$. Then $\Omega_t$ is also transverse on $X_t$ for small $t$.
	\end{proposition}
	
	This proposition demonstrates that any smooth real extension of a transverse $(p,p)$-form remains transverse. Therefore, the only obstruction to extending a $d$-closed transverse $(p,p)$-form on a compact complex manifold lies in its $d$-closed property. We will provide the proof of Theorem \ref{pkahler introduction}  in the next subsection.

	\subsection{Proof of Theorem \ref{pkahler introduction}}\label{proof pkahler}
	Consider a Kuranishi family of deformations
	$\pi:\mathcal{X}\rightarrow \Delta_{\epsilon}$  (with small $\epsilon$) of $n$-dimensional complex manifolds over a small complex disk
	with the general fibers $X_t:=\pi^{-1}(t)$, and	fix a smooth family of  Hermitian metrics
	$\{h_t\}_{t \in \Delta_{\epsilon}}.$ 
    
Now given the $p$-K\"{a}hler form $\omega_0$ on $X_0$, namely a $d$-closed transverse $(p,p)$-form, based on Theorem   \ref{main thm}, one can find a unique solution $\omega(t)\in A^{p,p}(X_0)$ satisfying the equation (\ref{integral equation}) which is smooth in $(z,t)$ and holomorphic in small $t$. Thus, we apply (\ref{decomposition}) to our situation to see
	$$
	\begin{aligned}
		e^{\iota_{\varphi}}(\omega(t)) &=\omega_t^{p, p}+\omega_t^{p+1, p-1}+\cdots+\omega_t^{2p, 0} \\
		& \in A^{p, p}\left(X_{t}\right) \oplus A^{p+1, p-1}\left(X_{t}\right) \oplus \cdots \oplus A^{2p, 0}\left(X_{t}\right)
	\end{aligned}
	$$
	is $d$-closed in $F^{p}A^{2p}(X_t).$ Then one has the observations due to  type reason:
	\begin{equation}\label{ddbar openness}
		\begin{cases}
			\bp_{t}{\omega_t^{p, p}}=0,\\
			\p_t{\omega_t^{p, p}}+	\bp_{t}{\omega_t^{p+1, p-1}}=0.
		\end{cases}
	\end{equation}
	Since $X_0$ satisfies the $\p\db$-lemma, so does the general fiber
	$X_t$. Thus, there exists $\beta_t\in A^{p,p-1}(X_t)$ satisfying  $$\p_t{\omega_t^{p, p}}=\p_t\bp_{t}\beta_t$$ by virtue of the second equality of (\ref{ddbar openness}). Set
	$\widehat{\omega_t}:=\omega_t^{p, p}-\bp_{t}\beta_t$. Then,  $$\widetilde{\omega_t}:=\frac{1}{2}\left(\widehat{\omega_t}+\overline{\widehat{\omega_t}}\right)$$  is real and $d$-closed. Thanks to Proposition \ref{ext-trans}, to prove that $\widetilde{\omega}_t$ is the $p$-K\"{a}hler extension of the $p$-K\"{a}hler form $\omega_0$ on $X_0$, it suffices to show that
	$\bp_{t}\beta_t$ converges uniformly to zero as $t$ tends to $0$.
	
	Following the notations introduced in Section \ref{sus-introductionpkahler}, one knows that $\mb H_{\mathrm {BC},t}$ and
	$\G_{\mathrm {BC},t}$ are $C^{\infty}$ differentiable in $t$ since  $\dim
	\mathds{F}_{\mathrm {BC},t}$ is deformation invariant, because of \cite[Theorem 7]{KS60} and \cite[Theorem
	5.12]{Wu06}.
	Choose the explicit solution
	\begin{equation*}
		\beta_t=(\p_t\bp_{t})^*\G_{\mathrm {BC},t}\p_t\omega_t^{p,p},
	\end{equation*}
	as shown in Lemma \ref{ddbar-eq}. One thus obtains that $\bp_{t}\beta_t$ converges uniformly to zero with $t$, since
	\begin{equation*}
		{\beta_t}|_{t=0}=(\p_0\bp_0)^*\G_{\mathrm {BC},0}\p_0\omega_0=0.
	\end{equation*}
	
	Theorem \ref{pkahler introduction}  is, thus, proved.
	
	\begin{remark}\label{weaken remark}
		It is natural to ask if we can weaken the standard $\p\bp$-property  that appeared in Theorem \ref{pkahler introduction}  in our approach as in \cite[Theorem 1.1]{RWZ19}.
		
		\begin{enumerate}[(a)]
			\item By \cite[Definition 3.1]{RWZ19},  we say that a complex manifold $X$ satisfies the  \textit{$(p,q)$-th mild $\p\b{\p}$-property} if for any complex differential $(p-1,q)$-form $\xi$ with $\p\b{\p}\xi=0$ on $X$, there exists a $(p-1,q-1)$-form $\theta$ such that $\p\b{\p}\theta=\p\xi$. Then obviously,  this condition is equivalent to the condition $\mathbb B^{p,q}$ as introduced in \cite[Notation 3.5]{RZ18}.
			
			In our proof of Theorem \ref{pkahler introduction} , the existence
			of $\beta_t$ relies on the deformation openness of $\p\bar{\partial}$-property. However, the mild one doesn't possess this property (see \cite[Example 3.7]{UV15}).
			\item 	There is another weak version of the $\p\bp$-property, namely  the \textit {$(p,q)$-th strong $\p\db$-property} on $X$, first proposed by
			Angella--Ugarte \cite{AU17}  in the case $(p,q)=(n-1,n)$, stated that the induced mapping
			$\iota^{p,q}_{\mathrm {BC},\mathrm A}: H^{p,q}_{\mathrm {BC}}(X) \rightarrow
			H^{p,q}_{\mathrm A}(X)$ by the identity map is injective (see Diagram \eqref{diag}), which
			is equivalent to that for any $d$-closed $(p,q)$-form $\Gamma$ of
			the type $\Gamma=\p\xi+\db \psi$, there exists a $(p-1,q-1)$-form
			$\theta$ such that
		$\p \db \theta = \Gamma.$
			It is a well-known fact that the $(p,q)$-th strong $\p\db$-property can imply the $(p,q)$-th mild $\p\b{\p}$-property.
			
			The most noteworthy is that
			Angella--Ugarte \cite[Proposition 4.8]{AU17} showed the deformation
			openness of the $(n-1,n)$-th strong $\p\db$-property. So, it's spontaneous to ask if one can weaken our second part of Corollary \ref{cor-mt} to $(n-1,n)$-th strong $\p\db$-property version by using our method.
			One can see that in the proof of Theorem \ref{pkahler introduction} , it is crucial to ensure the $C^{\infty}$  differentiability of $\G_{\mathrm {BC},t}$  in $t$ when $t$ converges  to 0, which is benefited from the  deformation invariance of Bott--Chern number.  However, \cite[Example 4.10]{AU17} shows that a small
			deformation of a completely-solvable Nakamura threefold, which is
			balanced and satisfies the $(2,3)$-th strong $\p\db$-property, is also
			balanced. But the $(2,2)$-th Bott--Chern number varies along this
			deformation.
		\end{enumerate}
		
		Hence, among all the (weak) versions of the $\p\bp$-properties as far as we know, it seems that the standard $\p\bp$-property on $X_0$ is the most general condition to prove the local stabilities of $p$-K\"{a}hler structures in this approach.
	\end{remark}
	
	\subsection{Upper semi-continuity of Bott--Chern numbers}\label{further}
	In several complex variables, one has the remarkable Grauert's semi-continuity theorem:
	\begin{theorem}[{\cite[Theorem  4.12.(\romannumeral 1) of Chapter \uppercase\expandafter{\romannumeral3}]{BS76} or \cite{Gr60}}] \label{uppersemi}
		Let $f: X \rightarrow Y$ be a proper morphism of complex spaces and $\mathcal{S}$ a coherent analytic sheaf on $X$ which is flat with respect to $Y$ (or $f$), which means that the $\mathcal{O}_{f(x)}$-modules $\mathcal{S}_{x}$ are flat for all $x \in X .$ Set $\mathcal{S}(y)$ as the analytic inverse image with respect to the embedding $X_{y}$ in $X$. Then for any integers $i, d \geq 0$, the set
		$$
		\left\{y \in Y \,\, |\,\, \dim_{\mathbb{C}} H^{i}\left(X_{y}, \mathcal{S}(y)\right) \geq d\right\}
		$$
		is an analytic subset of $Y$.
	\end{theorem}
	The topology on $Y$ defined by closed sets as analytic sets is referred to as the \textit{analytic Zariski topology}. The statement of Theorem \ref{uppersemi} implies the upper semi-continuity of $h^{i}\left(X_{y}, \mathcal{S}(y)\right)$ concerning this analytic Zariski topology.
	
	The following result, a straightforward application of Theorem \ref{uppersemi}, plays a crucial role in deformation theory.
	
	\begin{theorem}[{\cite[$\S 10.5.4$]{GR84}}]\label{Upper semi-continuity}
		Let $f:X\rightarrow Y$ be a holomorphic family of compact complex manifolds with connected complex manifolds $X,Y$ and $V$ a holomorphic vector bundle on $X$. Then for any integers $i,d\geq 0$, the set $$\{y\in Y \,\, |\,\,   \dim_{\mathbb{C}} H^i(X_y,V|_{X_y})\geq d\}$$ is an analytic subset of $Y$.
	\end{theorem}
	
	It is noteworthy that Grauert's semi-continuity theorem is considerably stronger than the corresponding result of Kodaira--Spencer when dealing with a holomorphic family, since the ordinary topology is obviously much finer.  For example, consider a Kuranishi family of deformations $\pi:\mathcal{X}\rightarrow \Delta_{\epsilon}$ of $n$-dimensional compact complex manifolds with small $\epsilon$,  Kodaira--Spencer's result tells us that the
	function $ \Delta_{\varepsilon}\ni t\mapsto h^{p,q}_{\db_t}(X_t)$
	is always upper semi-continuous (with respect to the original topology), but Grauert's theorem gives the additional information that  this function  is actually  constant outside  a complex analytic set of $\Delta_\epsilon$. In other words, in this Kuranishi family if jumping of Hodge numbers occurs it must do so only on a complex analytic set. This additional information could not be got just from the original Kodaira--Spencer's theorem. Here we apply Theorem \ref{Upper semi-continuity} to the holomorphic vector bundle of relative differential forms $ \Omega_{\mathcal X/\Delta_\epsilon}^p:=\bigwedge^{p} \Omega_{\mathcal{X} /\Delta_\epsilon },$  noting that $\Omega_{\mathcal{X} / \Delta_\epsilon}^{p}|_{ X_{t}} \cong \Omega_{X_{t}}^{p}$. One can also refer to \cite[Proposition 5.8]{Xi21b} for a new proof concerning some cases of Theorem \ref{Upper semi-continuity}.
	
	Now, we turn to study the Bott--Chern numbers under small deformation. With the  setting as stated in the last paragraph, the function is always upper semi-continuous with respect to the classical topology for $t\in \Delta_{\varepsilon},$ thanks to \cite[Theorem $4$]{KS60}. So it's natural to ask if the function  $t\mapsto h^{p,q}_{\mathrm {BC}}(X_t)$ is also upper semi-continuous with respect to the analytic Zariski topology. One possible approach is to use Grauert's upper semi-continuity Theorem \ref{Upper semi-continuity}.  In other words, one needs to find some holomorphic vector bundle $V$ on $\mathcal X$, such that  $H^{p,q}_{\mathrm {BC}}(X_t)\simeq H^q(X_t,V|_{X_t} ).$  
    
    Note  that  Aeppli cohomology is the dual of Bott--Chern cohomology (see (\ref{dual Aeppli})). Hence, in order to associate it with the cohomology group of some sheaf, we next will list a few results concerning the resolutions of the sheaf $\mathscr{H}$ of germs of pluriharmonic functions  originating from \cite{Bg69,AN71}.
	
	Denote by  $\mathscr A^{p,q}$ the sheaf of germs of $C^\infty$ complex differential forms of type $(p,q)$ on an $n$-dimensional complex manifold $X$ and set $A^{p,q}(X)=\Gamma(X,\mathscr A^{p,q})$. One also calls $\Omega^p$ (resp. $\overline{\Omega^{p}}$) the sheaf of germs of holomorphic (resp. antiholomorphic) $p$-forms on $X$.
	
	First of all, we have the exact sequence
	$$
	0 \longrightarrow \mathbb{C} \stackrel{\alpha}{\longrightarrow} \mathcal O \oplus \overline{\mathcal O} \stackrel{\beta}{\longrightarrow} \mathscr{H} \longrightarrow 0
	$$
	where $\alpha(c)=c\oplus(-c)$  and $\beta(f\oplus g)=f+g.$
	
	Secondly we have the following resolutions of types $(1,1)$ and $(2,3)$:
	
	
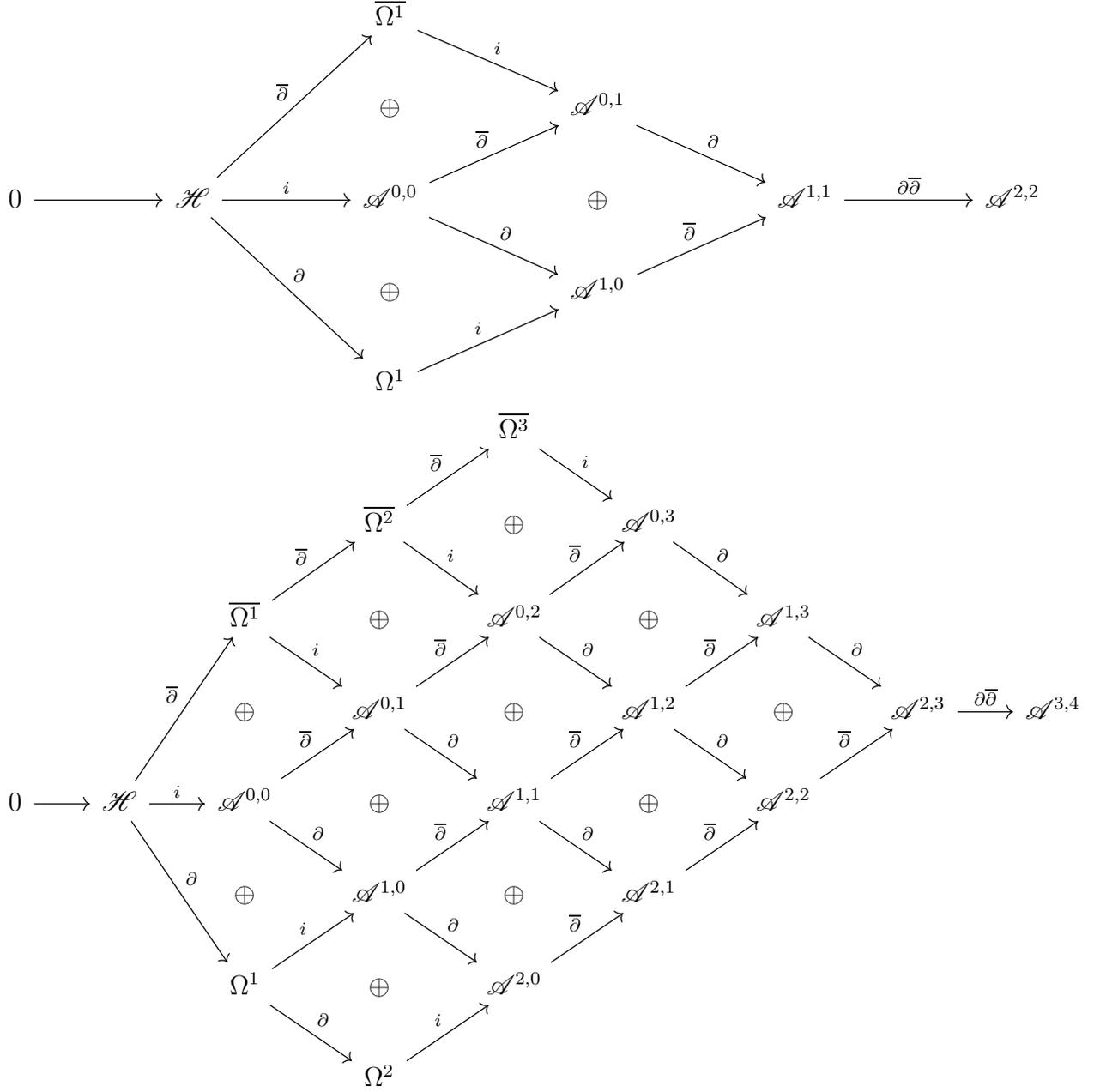
\begin{tikzcd}[column sep=huge]
		& &  \overline{\Omega^{1}}\arrow[dr, "i"] & & \\
		& & \oplus  &\mathscr A^{0,1} \ar[ dr, "\p"]& \\
		0\arrow[r]&\mathscr{H} \arrow[uur, "\bp"]\arrow[r, "i"]\arrow[ddr, "\p"] &\mathscr A^{0,0}  \ar[ur, "\bp"]\ar[ dr, "\p"] & \oplus  &\mathscr A^{1,1} \ar[r, "\p\bp"] & \mathscr A^{2,2}\\
		& & \oplus  &\mathscr A^{1,0} \ar[ ur, "\bp"]& \\
		& &  {\Omega^{1}}\arrow[ur, "i"] & &
	\end{tikzcd}
	
	\begin{tikzcd}[column sep=2.1em]
		& & & & \overline{\Omega^{3}}\arrow[dr, "i"] & & & &  \\
		& & &  \overline{\Omega^{2}} \arrow[ur, "\bp"] \arrow[dr, "i"] & \oplus &\mathscr A^{0,3} \ar[ dr, "\p"]& & & \\
		& & \overline{\Omega^{1} \arrow[ur, "\bp"] }\arrow[dr, "i"] & \oplus &\mathscr A^{0,2} \ar[ur, "\bp"]\ar[ dr, "\p"] &\oplus &\mathscr A^{1,3} \ar[ dr, "\p"]& &\\
		& & \oplus  &\mathscr A^{0,1} \ar[ur, "\bp"]\ar[ dr, "\p"] & \oplus  &\mathscr A^{1,2} \ar[ur, "\bp"]\ar[ dr, "\p"] & \oplus
		&\mathscr A^{2,3}\ar[r, "\p\bp"] & \mathscr A^{3,4}\\
		0\arrow[r]&\mathscr{H} \arrow[uur, "\bp"]\arrow[r, "i"]\arrow[ddr, "\p"] &\mathscr A^{0,0}  \ar[ur, "\bp"]\ar[ dr, "\p"] & \oplus  &\mathscr A^{1,1} \ar[ur, "\bp"]\ar[ dr, "\p"] & \oplus & \mathscr A^{2,2}\ar[ur, "\bp"]\\
		& & \oplus  &\mathscr A^{1,0} \ar[ur, "\bp"]\ar[ dr, "\p"] & \oplus  &\mathscr A^{2,1} \ar[ur, "\bp"] \\
		& & \Omega^{1} \arrow[ur, "i"]\arrow[dr, "\p"] &\oplus &\mathscr A^{2,0} \ar[ur, "\bp"] \\
		& & & \Omega^{2} \arrow[ur, "i"]\\
	\end{tikzcd}
	and in general a resolution of type  $(p,q)$:
	\footnote{Here by (\ref{dual Aeppli}) we may assume  $1\leq p\leq q\leq n$. One should aware that the resolutions are invalid for types $(p,0)$ and $(0,q)$ when trying to relate the Aeppli cohomology groups to some sheaves as shown later.}
	~\\
	
	\begin{tikzcd}[column sep=2.6em,row sep=0.8em]
		0\arrow[r]&\mathscr{H}\ar[r] &\mathscr A^0\arrow[r,"h_0"] &\mathscr A^1\arrow[r, "h_1"]&\cdots\arrow[r]&\mathscr A^{q-1}\arrow[r,"h_{q-1}"]&\,\\
		&\mathscr{B}^q\arrow[r, "h_q"]&\mathscr{B}^{q+1}\arrow[r,"h_{q+1}"]&\cdots\ar[r]&\mathscr{B}^{p+q-1}\arrow[r]&\mathscr{A}^{p,q}\arrow[r,"\p\bp"]&\mathscr{A}^{p+1,q+1},
	\end{tikzcd}
	~\\
	where
	~\\
	$$\mathscr{B}^i=\left\{\begin{array}{ll}
		\mathop\oplus\limits_{s=0}^i \mathscr{A}^{s,i-s}, & \text{for }0\leq i\leq p-1;\\ [0.8em]
		\mathop\oplus\limits_{s=0}^p \mathscr{A}^{s,i-s}, & \text{for }p\leq i\leq q-1;\\ [0.8em]
		\mathop\oplus\limits_{s=i-q}^p \mathscr{A}^{s,i-s}, & \text{for }q\leq i\leq q+p-1,\\
	\end{array} \right.$$
	and
	$$\mathscr{A}^i=\left\{\begin{array}{ll}
		\overline{\Omega^{i+1}}\oplus\mathscr{B}^i\oplus\Omega^{i+1},& \text{for }0\leq i\leq p-1;\\[0.6em]
		\overline{\Omega^{i+1}}\oplus\mathscr{B}^i,& \text{for }p\leq i\leq q-1,\\
	\end{array} \right.$$
	and where the maps  are induced by $\p$, $\bp$ and injection due to the  bidegree.
	
	To use the above resolution, we set
	$$ \mathscr{L}^j=\ker h_j,$$
	and then get the exact sequence of sheaves:
	
	\begin{tikzcd}[column sep=1.8em]\label{ssssssssssssss}
		0\arrow[r]&\mathscr L^q\arrow[r]&\mathscr{B}^q\arrow[r, "h_q"]&\mathscr{B}^{q+1}\arrow[r]&\cdots\arrow[r]&\mathscr B^{p+q-1}\arrow[r]&\mathscr{A}^{p,q}\arrow[r,"\p\bp"]&\mathscr{A}^{p+1,q+1},
	\end{tikzcd}
	
	\noindent and

	\begin{tikzcd}[row sep=tiny]
		&&&	0 \arrow[r]&\mathscr{L}^{q-1}\arrow[r]&\mathscr{A}^{q-1}\arrow[r]&\mathscr{L}^q\arrow[r]&0,	\\
		&&&	0 \arrow[r]&\mathscr{L}^{q-2}\arrow[r]&\mathscr{A}^{q-2}\arrow[r]&\mathscr{L}^{q-1}\arrow[r]&0,	\\[-0.9em]
		&&&\vdots &&\vdots&&\vdots\\[-0.6em]
		&&&	0 \arrow[r]&\mathscr{H}\arrow[r]&\mathscr{A}^{0}\arrow[r]&\mathscr{L}^{1}\arrow[r]&0.	\\
	\end{tikzcd}
	
	Since the sheaves $\mathscr{B}^s$ ($q\leq s\leq p+q-1 $) are fine  sheaves, one obtains
	\begin{equation}\label{roundabout}
		H^{n-p,n-q}_{\mathrm {BC}}(X)\simeq H^{n-q,n-p}_{\mathrm {BC}}(X)\simeq H^{p,q}_\mathrm{A}(X)\simeq H^{q,p}_\mathrm{A}(X)\simeq H^p(X,\mathscr{L}^q).
	\end{equation}
	
	However, it is easy to show that the sheaf  $\mathscr{L}^q$ is not coherent analytic, not to mention $\pi$-flat. Thus, it seems difficult to  apply Theorem \ref{uppersemi} directly to get the desired consequence. The readers probably observe that here we get (\ref{roundabout}) in a roundabout way.  In fact, one can also gain the similar result straightforwardly, i.e., find some sheaves $\mathscr{V}^q,$ such that $H^{p,q}_{\mathrm {BC}}(X)\simeq H^p (X,\mathscr{V}^q)$ by virtue of the corresponding resolutions. These two ways are actually equivalent.
	
	Based on the above discussions, here appears a basic but possibly difficult question:
	
	\begin{question}\label{question}
		Let $\pi:(\mathcal{X}, X) \rightarrow(B, 0)$ be a small holomorphic deformation  of a compact complex manifold $X$. 
		Is the function  $B\ni t\mapsto h^{p,q}_{\mathrm {BC}}(X_t)$  upper semi-continuous with respect to the analytic Zariski topology?
	\end{question}

	\begin{remark}\label{key}
		One can give an affirmative answer to Question \ref{question}
		in a special case, that is any compact complex manifold $X$ satisfying $b_2=0$ and  $(p,q)=(1,1),$ where $b_k:=b_k(X)=\dim_{\mathbb{C}}H^{k}(X,\mb C)$. Indeed, under these assumptions, we have (cf. \cite[Theorem 6.2]{Mc19} for example)
		$$h^{1,1}_{\mathrm {BC}}=h^{n-1,n-1}_\mathrm A=2h^{0,1}_{\bp}-b_1.$$
		
		Recently,  Xia   \cite[Theorem 1.1, Remark 3.6]{Xi21a} confirmed the above question when the type is $(p,0)$ or $(0,q)$. One motivation to consider  Question \ref{question} is :
		
	\end{remark}

	\begin{theorem}\label{uncontable}
		Suppose that the answer to Question \ref{question}  with $\dim B=1$ is positive, and the fibers $X_t$ in an uncountable subset of $B$ are $p$-K\"{a}hler. Then any very general fiber is still $p$-K\"{a}hler. Here a very general fiber $X_b$ is a fiber over $b\in B$ that belongs to  $V$ which is a complement of countably many Zariski closed proper subsets $Z_i$ of $B$.
	\end{theorem}
	
	\begin{proof}
		The theorem follows as long as one notices that the $\p\bp$-property in Theorem \ref{pkahler introduction}  actually can be replaced by the deformation invariance of $(p,p)$-Bott--Chern numbers as shown in \cite[Remark 4.13]{RWZ21}.
	\end{proof}
	
	Noteworthy is the following question about the analytic openness property of the $\p\bp$-property.
	
	\begin{question}[{\cite[p. 2]{Xi21a}}]\label{xia's question}
		With the setting of Question \ref{question}, is the set
		$$T:=\left\{t\in B \,\, |\,\, X_t \text{ satisfies the } \p_t{\bp}_t\text{-property}\right\}$$ an analytic open set (i.e., complement of an analytic subset) of $B$?
	\end{question}
	
	If one can confirm  Question  \ref{xia's question} with $T\neq\varnothing$,  then  $T=B$ or $T=B\setminus\left\{0\right\}$ possibly after shrinking the open base $B\subset\mb C$.
	\begin{proposition}
		If the answer to Question \ref{question} is positive, then so is the answer to Question \ref{xia's question}.
	\end{proposition}
	\begin{proof}
		\cite[Theorem]{AT13} tells us that for every $k\in\mb N$ and $t\in B$,
		\beq \label{angella}h_{\mathrm {BC}}^k(X_t)+h_\mathrm{A}^k(X_t)\ge2b_k(X_t)=2b_k,\eeq
		where $h_{\mathrm {BC}}^k(X_t):=\sum_{p+q=k}h^{p,q}_{\mathrm {BC}}(X_t)$ and $h_{\mathrm{A}}^k(X_t):=\sum_{p+q=k}h^{p,q}_{\mathrm{A}}(X_t).$
		Moreover, the equality in (\ref{angella}) holds if and only if $X_t$ satisfies the $\p_t{\bp}_t\text{-property}.$
		
		As the sum of finitely many upper semi-continuous functions (with respect to any topology) is still upper semi-continuous, one can combine it with (\ref{dual Aeppli}) to see that
		$$T=B\setminus\bigcup_{k\in \mb N }\left\{t\in B \,\, |\,\, h_{\mathrm {BC}}^k(X_t)+h_\mathrm{A}^k(X_t)\ge2b_k+1\right\}$$
		is an analytic open subset of $B$.
	\end{proof}

		\section{Local stabilities of transversely \texorpdfstring{$p$}{p}-K\"ahler foliations}\label{LOCAL}
	We first introduce the concept of transversely $p$-K\"ahler foliation,
	then extend the `weak' $\partial\bp$-properties to the foliated case.   We strengthen  the local stabilities theorems in \cite{EKAG97, Ra21} to the transversely $p$-K\"ahler cases, which can also be viewed as the foliated version of the main results in \cite{RWZ19}.  In a sense, our use of the power series approach has led to increased efficiency when contrasted with the techniques employed in \cite{EKAG97,Ra21}, as illustrated in Remark \ref{Non}.

	\subsection{Transversely $p$-K\"ahler structures}\label{subsection transversely }
	In this subsection, we first provide a rapid review of  transverse structures on foliations (see \cite{Mo88,Ni11} for example), and then give the definition of transversely $p$-K\"ahler foliation.  Furthermore, a class of examples of homologically orientable transversely Hermitian foliations on compact nilmanifolds that are transversely $p$-K\"ahler (with $p>1),$ but not transversely K\"ahler  are displayed.
	\begin{definition}
		A \textit{foliation} $\mathcal F$  of  codimension $s$ on an $m$-dimensional smooth manifold $M$ is a \textit{cocycle} $\mathscr U=(\{U_i\},\{\pi_i\},\{\gamma_{ij}\})$ modelled on an $s$-dimensional smooth manifold $N_0$ consisting of
		\begin{enumerate}[$(i)$]
			\item an open covering $\{U_i\}$ of $M$,
			\item submersions $\pi_i:U_i\to N_0$ with connected fibers defining $\mathcal{F}$,
			\item diffeomorphism transition functions $\gamma_{ij}:\pi_j(U_i\cap U_j)\to \pi_i(U_i\cap U_j)$ such that $\pi_i=\gamma_{ij}\circ \pi_j$.
		\end{enumerate}
		
	\end{definition}
	We call the disjoint union $N=\mathop\coprod_{U_{i}\in \mathscr U}
	\pi_{i}\left(U_{i}\right)$  the \textit{transverse manifold} of $\mathcal{F}$ associated to the cocycle $\mathscr U.$ The local diffeomorphisms $\gamma_{i j}$ generate a pseudogroup $\mathcal H$ of transformations on $N$. One can equate the space of leaves $M / \mathcal{F}$ of the foliation $\mathcal{F}$ with $N/ \mathcal H$.
	
	A \textit{transverse structure} to $\mathcal{F}$ is an $\mc H$-invariant geometric  structure on $N$. For instance:
	
	\begin{definition} \label{definitian transversely holom}The foliation $\mathcal{F}$ on $M$ is:
		\begin{enumerate}[$(i)$]
			\item \textit{Riemannian} if there exists an $\mathcal H$-invariant Riemannian metric on $N$, which is equivalent to the existence of a Riemannian metric $g$ with $L_Xg=0$ for any $X$ in the \textit{normal bundle} $\mathcal N\mathcal{F}:=TM/T\mathcal F$, where $T\mathcal{F}$ is the bundle tangent to the leaves;
			\item  \textit{transversely holomorphic} if on $N$ there exists a complex structure of which $\mc H$ is a pseudogroup of local holomorphic transformations, which is equivalent to the existence of an almost complex structure $J$ on $\mathcal N\mathcal{F}$ such that
			\begin{enumerate}[(a)]
				\item $L_X J=0$ for any vector field $X$ tangent to the leaves,
				
				\item 	for any two  sections $Z_i\,(i=1,2)$ of the normal bundle, we have $$N_J[Z_1,Z_2]:=[JZ_1,JZ_2]-J[Z_1,Z_2]-J[JZ_1,Z_2]+J^2[Z_1,Z_2]=0;$$
				
			\end{enumerate}
			\item  \textit{transversely Hermitian} if it is transversely holomorphic and there exists an $\mathcal H$-invariant Hermitian metric on $N$;
			\item \textit{transversely K\"ahler} if on $N$ there exists an $\mathcal H$-invariant K\"ahler structure.
		\end{enumerate}
	\end{definition}
	
	\begin{remark}\label{relationship foliation cpst}
		By Definition \ref{definitian transversely holom} (ii), one knows that a transversely holomorphic foliation of complex codimension $s$ with $m=2s$ is nothing but a \textit{complex structure} on $M$, likewise
		for the the last two cases.
	\end{remark}
	\begin{definition}
		A foliation $\mathcal{F}$ is said to be \textit{transversely orientable} if the normal bundle $\mc N \mc F$ is orientable. This is equivalent to the orientability of $N$ plus the orientation preservations of all  $\gamma_{i j}$.
	\end{definition}
	\begin{definition}
		A foliation is said to be \textit{transversely parallelizable} (abbreviated as TP) if there exist $s$ linearly independent $\mc H$-invariant vector fields.
	\end{definition}

	\begin{definition}
		A smooth form $\alpha$ on $M$ is called \textit{basic} if for any vector field $X$ tangent to the leaves of $\mc F$ the equalities
		$\iota_X\alpha=\iota_X{d\alpha}=0$ hold.
		Basic $0$-forms all also called \emph{basic functions}.
	\end{definition}
	
	Basic forms on the foliated manifold $(M,\mc F)$ are in one-to-one correspondence with $\mc H$-invariant forms on $N$. For any basic form $\alpha$, it is easy to see that $d\alpha$ is also basic. Hence, the set $A^\bullet (M/\mathcal{F})$ of basic forms is a subcomplex of $(A^\bullet(M),d).$ The  cohomology of $(A^\bullet (M/\mathcal{F}),d)$ is called the \textit{basic cohomology} of $(M,\mathcal{F})$ and denoted by $H^\bullet (M/\mathcal{F})$. 
	For a Riemannian foliation of codimension $s$ on a connected compact manifold $M$, $H^{s}(M/\mc F)=\mb  R$ or $H^{s}(M/\mc F)=0$ by a result of \cite{EKASH85}. One then has
	\begin{definition}\label{homolo}
		The foliated manifold $(M,\mc F)$ (or just $\mc F$) is called \textit{homologically orientable}  if $H^{s}(M/\mc F)=\mb  R$.
	\end{definition}
	\begin{remark}\label{taut}
		The above condition is equivalent to the existence of a (real) volume form on the leaves $\chi$ which is relatively closed, that is, $d\chi(X_1,\ldots,X_{m-s},Y)=0$ for $X_1,\ldots,X_{m-s}$ tangent to $\mc F$, cf. \cite{Ru79}. In that case, one can complete the transverse metric by  a Riemannian metric on the whole manifold for which the leaves are minimal and $\chi$ is associated to this metric and we then can also say that $\mc F$ is \textit{taut}.
	\end{remark}
	
	Furthermore,  suppose that  the foliation $\mathcal{F}$ is transversely holomorphic of complex codimension $r$ with $s=2r$. We say that a system of local coordinates  $$(x,z)=\left(x^{1}, \ldots, x^{m-s}, z^{1}, \ldots, z^{r}\right)$$  on $U_i$ is  \textit{adapted} to $\mc F$ (or more precisely to the submersion $\pi_{i}$) if for each point $w\in U_i$, one has $\pi_{i}(w)=\left(z^1(w),\ldots, z^{r}(w)\right)$. On $U_{i}$, and in terms of adapted coordinates, the leaves of $\mathcal{F}$ are defined by $z^{j}=$ constant.
	This notion surely can be defined similarly for general foliations. So in adapted coordinates $(x, z)$ the tangent bundle $T \mathcal{F}$ is generated by the vector fields $\frac{\partial}{\partial x^{1}}, \ldots, \frac{\partial}{\partial x^{m-s}} .$ That is, $T \mathcal{F}$ is the kernel of the differential $d \pi_{i}$. As in the case of a complex structure, the complexification $\mathcal N \mathcal{F}^{\mb C}$ of the normal bundle $\mathcal N \mathcal{F}$ of a transversely holomorphic foliation $\mathcal{F}$ decomposes as
	$$
	\mathcal N \mathcal{F}^{\mb C}=\mathcal N \mathcal{F}^{1,0} \oplus \mathcal N \mathcal{F}^{0,1}
	$$
	and the corresponding dual bundles fulfill the same relation $\mathcal N^{*} \mathcal{F}^{\mb C}=\mathcal N^{*}  \mathcal{F}^{1,0} \oplus$ $\mathcal N^{*}\mathcal{F}^{0,1}$. Notice that, locally,
	$$
	\mathcal N^{*} \mathcal{F}^{1,0}=\left\langle d z^{1}, \ldots, d z^{r}\right\rangle_{\mathbb{C}} \quad \text { and } \quad \mathcal N^{*} \mathcal{F}^{0,1}=\left\langle d \bar{z}^{1}, \ldots, d \bar{z}^{r}\right\rangle_{\mathbb{C}}.
	$$

Let $\tau:E\rightarrow M$ be a complex vector bundle of rank $N$ defined by a cocycle $\left\{V_i,g_{ij}\right\}$, where $\left\{V_i\right\}$ is an open cover of $M$ and ${g_{ij}}$ are the transition  functions $g_{ij}:V_i\cap V_j\rightarrow\text{GL}(N,\mathbb C)$ satisfying the cocycle condition:
	$$g_{ij}(w)=g_{ik}(w)\cdot g_{kj}(w) \,\,\text{ for }\, w\in V_i\cap V_j\cap V_k.$$
	We say that $E$ is an $\mathcal{F}$-\textit{bundle} if the functions $g_{ij}$ are basic on $V_i\cap V_j.$
	Then for a $\mathcal F$-bundle $E$ on $M$, a section $\alpha\in C^\infty(M,E)$ is said to be  \textit{basic} if the local representative functions of $\alpha$ with respect to an adapted system of local coordinates  which can trivialize $E$ are basic. It can be proved easily that $\bigwedge^{k} (\mathcal N^{*} \mathcal{F}^{\mb C})$, the $k$-th exterior product of the complexification of the conormal bundle of $\mathcal{F}$,   is an $\mathcal{F}$-bundle. Therefore, we can view \emph{basic $k$-forms}  as basic sections of $\bigwedge^{k} (\mathcal N^{*} \mathcal{F}^{\mb C})$. And a basic $k$-form $\alpha$ is
	of \emph{pure type $(p, q)$} if for any point of $M$ there exists an adapted system of local coordinates described as above  such that
	\begin{equation*} \label{foliation explicit}
		\alpha(x,z)= \sum_{1\leq i_1<\cdots<i_p\leq r\atop 1\leq j_1<\cdots<j_q\leq r}f_{i_1\cdots
			i_p {j_1}\cdots {j_q}}(z,\b z)d z^{i_{1}} \wedge \cdots \wedge d z^{i_{p}} \wedge d \bar{z}^{j_{1}} \wedge \cdots \wedge d \bar{z}^{j_{q}}.
	\end{equation*}

Let us denote by
	$A^k_\mathbb{C}(M/\mathcal F)\quad \text{(resp. $A ^{p,q}(M/\mathcal F)$)}$
	\textit{the space of complex valued basic $k$-forms} (resp. \textit{the space of complex valued basic forms of pure type $(p,q)$}) on the transversely holomorphic foliated manifold $(M,\mc F).$ Then,
	$$A_{\mathbb{C}}^{k}(M/\mathcal{F})=\bigoplus_{p+q=k} A^{p, q}(M/\mathcal{F}).$$
	Also note that the exterior derivative $d$  can be decomposed into the sum of operators $\partial$ and $\bp$ of order $(1,0)$ and $(0,1)$, respectively. In other words,
	\begin{equation*}\label{partial and barpartial}
		\partial:A^{p, q}(M/\mathcal F) \longrightarrow A^{p+1, q}(M/\mathcal F) \quad \text { and } \quad \bar{\partial}: A^{p, q}(M/\mathcal F)\longrightarrow A^{p, q+1}(M/\mathcal F).
	\end{equation*}
	The differential complex $(A^{p,\bullet}(M/\mc F),\bp)$ is called \textit{the $p$-th basic Dolbeault complex of $(M,\mc F)$}. Its cohomology is called the $p$-th  \emph{basic Dolbeault cohomology} of $(M,\mc F)$ which is denoted by $H^{p,\bullet}(M/\mc F)$. We denote by $h^{p,q}_{\bp}(M/\mc F)$ the dimension of $H^{p,q}(M/\mc F)$ over $\mb C$ and call it the \emph{$(p,q)$-basic Hodge number}.

	Next, we introduce the following new notion on the transverse structures of foliations, which can be seen as the transverse version of
	Definition \ref{pkf}. 
	
	\begin{definition}\label{TRAN P-Kahler}
		A transversely holomorphic  foliation $\mc F$ of complex codimension $r$ is said to be \textit{transversely $p$-K\"ahler} for the integer $p$ with $1\leq p\leq r$ if there exists an $\mathcal{H}$-invariant $p$-K\"ahler structure   on $N$, or equivalently, if $\mc F$ admits a \textit{transversely $p$-K\"ahler form}, that is a $d$-closed transverse positive $(p,p)$-basic form. Here we call a real $(p,p)$-basic form $\Omega$ \textit{transverse positive}  if at any given $w\in M$, for any independent $\tau_j\in(\mathcal N^{*} \mathcal{F}^{1,0})_w$, $1\leq j\leq {r-p}$,
		$$\Omega(w)\wedge\sqrt{-1}\tau_1\wedge\overline{\tau_1}\wedge\cdots\wedge\sqrt{-1}\tau_{r-p}\wedge\overline{\tau_{r-p}}$$
		is strictly positive, or equivalently, if for any point of $M$ there exists an adapted system of local coordinates $(U,x,z)$ containing it such that
		for any independent $\tau_j\in \mathcal N^{*} \mathcal{F}^{1,0}|_U, 1\leq j\leq{r-p},$
		$$\Omega|_U\wedge\sqrt{-1}\tau_1\wedge\overline{\tau_1}\wedge\cdots\wedge\sqrt{-1}\tau_{r-p}\wedge\overline{\tau_{r-p}}$$
		is strictly positive.
	\end{definition}
	
	\begin{remark}
		\begin{enumerate}
			\item Notice that  in Definition \ref{TRAN P-Kahler},  $\tau_j$  is not  required to come from a  global basic form. Indeed, different from the classical case, we can't always extend a nonzero local basic form to be a global one due to the lack of `transverse version' of unity of partition. See for example \cite[Definition of the degree of a transversely Hermitian vector bundle (1.3)]{BH22} for the same reason.
			\item Our `transverse positive' definition for a basic form is equivalent to the one  in \cite[p. 8]{ZZ21} (although they are more concerned with homologically orientable transversely holomorphic foliation).
		\end{enumerate}
	\end{remark}

	\begin{remark}\label{nvwa}
		Analogously to Lemma \ref{pkb}, one can show that a foliation $\mc F$ is transversely 1-K\"ahler (resp. transversely $(r-1)$-K\"ahler) if and only if $\mc F$ is transversely K\"ahler (resp. transversely balanced). Here, we say that $\mathcal{F}$ is \emph{transversely balanced} if it admits a real transverse positive $(1,1)$-basic form $\omega$ such that $\omega^{r-1}$ is $d$-closed, or equivalently, if it admits a $d$-closed strictly positive $(r-1,r-1)$-basic form, cf. \cite[(4.8)]{Mich}. Similarly, we say that $\mathcal{F}$ is \emph{transversely Gauduchon}  if it admits a real transverse positive $(1,1)$-basic form $\omega$ such that $\omega^{r-1}$ is $\p\bp$-closed mimicking the definition as in the complex manifold case, cf. \cite{Ga77}. Clearly, every transversely balanced foliation is
		also transversely Gauduchon.
	\end{remark}
	
	Next, we shall present a class of examples of homologically orientable transversely Hermitian foliations on compact nilmanifolds which are transversely $p$-K\"ahler (with $p>1),$ but not transversely K\"ahler, motivated by \cite{CW90} and \cite{AB91}.
	
	\begin{example}\label{nilmanifold}
		
		Let $N_{2r+1}$ be the following Heisenberg type group (which is a simply connected nilpotent  Lie group):
		\begin{equation*}
			N_{2r+1}:=\left\{A\in \text{GL}(r+2,\mb C)\,\, |\,\,A=\begin{bmatrix}
				1 &  X & z \\
				0 & I_r & Y \\
				0  & 0& 1 \\
			\end{bmatrix} \text{ where } z\in\mb C\text { and } {X,Y\in \mb C^r}\right\},
		\end{equation*}
		and let $G_{2r+1}$ be the discrete subgroup of $N_{2r+1}$ all of whose entries are Gaussian integers $\alpha_1+\sqrt{-1}\alpha_2$ with $\alpha_i\in \mb Z$. Then the homogeneous manifold $N_{2r+1}/G_{2r+1}$ becomes  a holomorphically parallelizable compact connected complex manifold, denoted by $\eta\beta_{2r+1}$. For $r=1$, $\eta\beta_3$ is the \textit{Iwasawa manifold}, the classical example of compact non-K\"ahler threefold, cf. \cite[p. 205]{AB91}.
		
		Now following the ideas presented in the construction of examples for foliations with specific structures (as discussed in \cite[Section 2]{CW90}), we intend to create a class of examples in question by use of  the manifold $\eta\beta_{2r+1}$. Let $\Gamma_{2r+1}$ be the following finitely generated subgroup of $N_{2r+1}$ of matrices of the form (which contains the uniform group $G_{2r+1}$ of $N_{2r+1}$):
		\begin{equation*}
			\begin{bmatrix}
				1 &  X+\boldsymbol{s}X^\prime & z+\boldsymbol{s} z^\prime \\
				0 & I_r & Y \\
				0  & 0& 1 \\
			\end{bmatrix}
		\end{equation*}
		where $X, X^\prime,Y$ are matrices with Gaussian integer entries, $z,z^\prime$ also are Gaussian integers and $\boldsymbol{s}$  is a fixed irrational number.
		
		Then $\Gamma_{2r+1}$ can be identified with the Gaussian integer lattice of the group $U_{3r+2}$ of matrices of the form:
		\[
		\begin{bmatrix}
			1 &x_1 &\cdots& x_r&x_1^\prime &\cdots&x^\prime_r &z&z^\prime \\
			0&1&\cdots &0&0&\cdots&0&y_1&0\\
			0&0&\cdots&\vdots&\vdots&\vdots&\vdots&\vdots&\vdots \\
			0&0&0&1&0&\cdots&0&y_r&0\\
			0&0&0&0&1&\cdots&0&0&y_1\\
			0& 	0&	0&	0&	0&\ddots&\vdots&\vdots&\vdots\\
			0&	0&	0&	0&	0&	0&1&	0&y_r\\
			0&	0&	0&	0&	0&	0&	0&1&	0\\
			0&	0&	0&	0&	0&	0&	0&	0& 1\\
		\end{bmatrix}
		\]
		where $x_i, x_i^\prime, y_i, y_i^\prime, z, z^\prime \in\mb C$ for any $i\in\left\{1,\ldots,r\right\}$. Furthermore, there exists a surjective submersion $u_r:U_{3r+2}\rightarrow N_{2r+1}$ (also with connected fibers) given by the correspondence
		$$\left(x_i,x_i^\prime,z,z^\prime, y_i,y_i^\prime\right)\longmapsto\left(x_i+\boldsymbol{s}x_i^\prime,z+\boldsymbol{s}z^\prime,y_i\right).$$ The foliation defined by the submersion $u_r$ is $\Gamma_{2r+1}$-invariant and therefore it projects to a complex $(2r+1)$-codimensional foliation,
		denoted by  $\mathcal{F}_{2r+1}$ (which
		depends on $U_{3r+2}, \Gamma_{2r+1}$ and $u_r$),  on a complex $(3r+2)$-dimensional compact manifold $M_{3r+2}:=U_{3r+2}/\Gamma_{2r+1}$. 
        
Noteworthy to mention that any foliated geometric structures on the foliated manifold $(M_{3r+2}, \mc F_{2r+1})$ correspond bijectively to $\Gamma_{2r+1}$-invariant ones on $N_{2r+1}$. In particular, the basic forms on $(M_{3r+2}, \mc F_{2r+1})$ are in one-to-one correspond with the $\Gamma_{2r+1}$-invariant forms on $N_{2r+1}$.  So the same reasoning  as in \cite[Section 4]{AB91} will enable one to show that for arbitrary $r$, $\mc F_{2r+1}$ is not transversely $p$-K\"ahler for $1\leq p\leq r$ and it is transversely $p$-K\"ahler for $r+1\leq p\leq {2r+1}$. The homologically orientability property of $\mc F_{2r+1}$ can be proved utilizing the method provided in \cite[p. 182]{CW91} similarly.
	\end{example}

	\subsection{Transversely elliptic operators}
	In this subsection, we recall the notion of a transversely elliptic basic differential operator and the construction of the scalar product both  focusing on the space of  basic forms, cf. \cite{EKA90}. For the rest of this paper, unless otherwise stated, we always assume that the foliation $\mc F$ is transversely Hermitian of complex codimension $r$ with $s=2r$ and that the manifold $M$ is compact.
	\begin{definition}\label{basic differential operator}
		A \textit{basic differential operator of order} $l$ \textit{on the space of complex valued basic forms} is a linear map $D: A^{\bullet}_\mb C(M /\mathcal{F}) \rightarrow A^{\bullet}_\mb C(M /\mathcal{F})$ such that in adapted coordinates $\left(x^{1}, \ldots, x^{m-s}, y^{1}, \ldots, y^{s}\right)$  it has the expression:
		$$
		D=\sum_{|\mu| \leq l} a_{\mu}(y) \frac{\partial^{|\mu|}}{\partial ^{\mu _1} y^{1} \cdots \partial^{\mu_s} y^{s}}
		$$
		where $\mu=\left(\mu_1,\ldots,\mu_s\right)\in\mb N^s$,\,$|\mu|=\mu_ 1+\cdots+\mu_s$ and $a_{\mu}$ are matrices of appropriate size with basic functions as coefficients.
		The \textit{principal symbol} of $D$ at the point $w=(x, y)$ and the basic
		covector $\xi \in ({\mathcal N^{*} \mathcal{F}})_w$, $\xi=\left(\xi_1,\ldots,\xi_{s}\right)$ is the linear map $$\sigma(D)(w, \xi): \bigwedge^\bullet({\mathcal N^{*} \mathcal{F}}^\mb C)_w \longrightarrow \bigwedge^\bullet({\mathcal N^{*} \mathcal{F}}^\mb C)_w$$ defined by
		$$\sigma(D)(w, \xi)(\eta)=\sum_{|\mu|=l} \xi_{1}^{\mu_{1}} \cdots \xi_{s}^{\mu_{s}} a_{\mu}(y)(\eta).$$
		We say that $D$ is \textit{transversely elliptic} if $\sigma(D)(z, \xi)$ is an isomorphism for every $w \in M$ and every basic covector $\xi$ different from $0$.
	\end{definition}
	
	\begin{Construction}\label{CONSTRUCTION}
		Now we come to briefly introduce the construction of the scalar product on the space  $A^{p,q}(M/\mc F)$ of $(p,q)$-basic forms, viewed as basic sections of $$N_{p,q}:=\bigwedge^{p,q}(\mc N^*\mc F^\mb C).$$  In this paper,  we will always follow this construction.   One starts with the principal $SO(s)$-bundle $$p: {M}^{\#} \longrightarrow M$$ of orthonormal frames transversal to $\mc F$, lifting the foliation ${\mc F}$ to a TP transversely Hermitian foliation ${\mc F}^{\#}$ on ${M}^{\#}$ with $\dim {\mathcal F}^{\#}=\dim \mc F.$
		In addition, the foliation is  $SO(s)$-invariant, i.e., for any element $a\in SO(s)$ and any leaf of $\mc F^{\#}$, $a(L)$ is also a leaf of $\mc F^\#$. We can further choose some transverse metric such that it is invariant with respect to the $SO(s)$-action and the fibers of  $p: {M}^{\#} \rightarrow M$  are of measure $1$.
		
		Let $N_{p,q}^{\#}=p^* N_{p,q}$. Then $N_{p,q}^{\#}$ is an $SO(s)$-bundle and in fact a Hermitian $\mc F^{\#}$-bundle (cf. \cite[Definition 2.5.2]{EKA90}). Also, one has the following canonical isomorphism:
		\begin{equation}\label{canonical isomophism 1}
			A^{p,q}(M/\mc F)\longrightarrow C^\infty _{SO(s)}(N_{p,q}^{\#}/\mc F^{\#}),
		\end{equation}
		where $C^\infty _{SO(s)}(N_{p,q}^{\#}/\mc F^{\#})$ is the space of  basic sections of $N_{p,q}^{\#}$ which are invariant under the action of $SO(s)$.
		By invoking the results in  \cite{Mo82}, one can get a compact manifold $W$ and a fiber bundle $$\pi:{M}^{\#}\longrightarrow W$$ such that the fibers amount to the closures of leaves of ${\mc{F}}^{\#}$, and then extend the above transverse metric to a Hermitian metric on ${M}^{\#}$ such that the fibers of $\pi:{M}^{\#}\rightarrow W$ also possess measure 1. 
        
        We then call the manifold $W$ the \textit{basic manifold} of $\mc F$. The $SO(s)$-action on ${M}^{\#}$ descends to an $SO(s)$-action on $W$.
		The bundle $N_{p,q}^{\#}$ on $M^{\#}$ is associated with a Hermitian bundle denoted by $\widetilde{N_{p,q}}$ on $W$, which will be called the \textit{useful bundle} corresponding to $N_{p,q}$. We will denote the Hermitian scalar product on $\widetilde{N_{p,q}}$ by $\widetilde{h_{p,q}}$, which is indeed induced by the transversely Hermitian metric on $(M,\mc F)$. One can show that  $\widetilde{N_{p,q}}$  is also an $SO(s)$-bundle and  there exists a canonical isomorphism of $A(W)$-module (here $A(W)$, as the ring of smooth functions on $W$, is identified with the ring of basic functions on $M$) by utilizing (\ref{canonical isomophism 1}):
		\begin{equation}\label{!canonical isomophism 2}
			\mathscr{B}: A^{p,q}(M/\mc F)\longrightarrow C^{\infty}_{SO(s)}(\widetilde{N_{p,q}}).
		\end{equation}
		where $C^{\infty}_{SO(s)}(\widetilde{N_{p,q}})$ denotes the space of smooth sections of $\widetilde{N_{p,q}}$ which are invariant under the action of $SO(s)$.
		
		Now, for arbitrary $\alpha,\beta\in A^{p,q}(M/\mc F)$, we define their scalar product by:
		$$\label{scalar}
		\langle\alpha,\beta\rangle:=\int_W \widetilde{h_{p,q}}\left(\mathscr{B}(\alpha),\mathscr B(\beta)\right)(\textrm{w})\,d\mu(\textrm{w}),
		$$
		where $\mu$ is the volume form associated with the Hermitian  metric induced by that of $M^{\#}$.
		The transverse $\star$-operator can be defined fiberwise on the orthogonal complements of the spaces tangent to the leaves standardly. Furthermore, one can define the H\"older norms or Sobolev norms in the space $A^{p,q}(M/\mc F)$ thanks to the isomorphism (\ref{!canonical isomophism 2}). For example, for  $\v\in A^{p,q}(M/\mc F)$ and $k\geq 0$ ($k\in \mb Z$), $0<\alpha<1$ ($\alpha\in\mb R$), we set
		\begin{equation*}\label{Holder foliation}
			||\varphi||_{k+\alpha}:=||\mathscr{B}(\varphi)||_{k+\alpha}.
		\end{equation*}
		The readers can refer to \cite[Section 3.5.6]{EKA90} for more details.
	\end{Construction}
	
	\begin{remark}\label{ADJOINT}
		The scalar product constructed as above can be used to define $\delta$ (resp. $\p^*$, resp. ${\bp}^*$) as the operators adjoint to $d$ (resp. $\p$, resp. $\bp$). We then can define the \textit{basic Dolbeault Laplacian} operator (likewise for $\square^\mathrm b_ d$ and $\square^\mathrm b_{\p}$):
		$$\square_{\bp}^\mathrm b:=\bp{\bp}^*+{\bp}^*\bp.$$
		Noteworthy is the following formula showed in
		\cite[Proposition 2.2]{ER96}, cf. also \cite{AL92}:
		\begin{equation*}\label{adjoint-d}
			\delta\beta=(-1)^{s(p+1)+1} {\star} (d +(P \zeta)\wedge)\star \beta,
		\end{equation*}
		where $\beta$ is a complex valued $p$-basic form and $P\zeta$ is a $1$-basic form dependent on $\mc F$ (a marginally modified mean curvature). 
        
        One  can also refer to \cite[line 10 in p. 1262]{ER96} for the explicit definition for the operator $$\star: A^p_\mb C(M/\mc F)\rightarrow A^{s-p}_\mb C(M/\mc F).$$
		And we have
		\begin{equation*}\label{2id}
			\star^2={(-1)}^{p(s-p)}\text{id}.
		\end{equation*}
		We split $P\zeta$ into forms $\zeta_1$ and $\zeta_2$ of types $(1,0)$ and $(0,1)$, respectively, to get
		\begin{equation}\label{adjoint-pb}
			{\p}^*\beta=(-1)^{s(p+1)+1} {\star} (\p+\zeta_1\wedge)\star \beta,\,\,\,\,\text{and}\,\,\,\,\,\,{\bp}^*\beta=(-1)^{s(p+1)+1} {\star} (\bp+\zeta_2\wedge)\star\beta
		\end{equation}
		with respect to the induced operator \begin{equation}\label{star}
			\star: A^{p_1,p_2}(M/\mc F)\rightarrow A^{r-p_1,r-p_2}(M/\mc F)
		\end{equation}
		with $p_1+p_2=p$.
		
		If $\mc F$ is further required to be  transversely K\"ahler with a transverse K\"ahler form $\omega$. We then can define the \textit{transverse Lefschetz} operator:
		$$L:A^{p_1,p_2}(M/\mc F)\rightarrow A^{p_1+1,p_2+1}(M/\mc F),\,\,L\alpha:=\alpha\wedge\omega$$
		and its joint \textit{transverse dual Lefschetz operator}
		$$\Lambda: A^{p_1,p_2}(M/\mc F)\rightarrow A^{p_1-1,p_2-1}(M/\mc F),\,\,
		\Lambda:=\star^{-1}L\star$$
		with respect to the operator $\star$ (\ref{star}).
		One may observe that the equalities (\ref{adjoint-d}) and (\ref{adjoint-pb})  are not so concise as those in the Hermitian manifold case. Generally, the forms $\zeta_i\,  (i=1,2)$ are exactly nonzero, see \cite{Car84} and \cite[Example 2.3 in Appendix B]{Mo88} for example. Sometimes to make things simpler, one may need the  homologically orientability condition, see Remark \ref{homologic1} for more details.
		
	\end{remark}
	
	\begin{remark}\label{homologic1}
		The homologically orientability or taut (cf. Definition \ref{homolo} and Remark \ref{taut}) hypothesis will enable one to  define an inner  product on $A^p(M/\mc F)$ without using the basic manifold $W$ introduced in Construction \ref{CONSTRUCTION}.  As in the classical case, one can define the Hodge star operator
		$$\star_\mathrm b:A^*(M/\mc F)\longrightarrow A^{s-*}(M/\mc F)$$
		in the following way. Let $U$ be an open set on which the foliation is trivial. Let $\theta_1,\ldots,\theta_s$ be real $1$-forms such that $\left(\theta_1,\ldots,\theta_s\right) $ is an orthonormal basis of the free module $\Omega^1(U/\mc F)$. Then define $\star_\mathrm b$ by
		$$\star_\mathrm b\left(\theta_{i_1}\wedge\cdots\wedge\theta_{i_p}\right)=\epsilon\theta_{j_1}\wedge\cdots\wedge\theta_{j_{s-p}},$$
		where $\left\{j_1,\ldots,j_{s-p}\right\}$ is the increasing complementary sequence of $\left\{i_1,\ldots,i_p\right\}$ in the set $\left\{1,\ldots,s\right\}$ and $\epsilon$ is the signature of the permutation $\left\{i_1,\ldots,i_p,j_1,\ldots,j_{s-p}\right\}.$ 
        
        A straightforward calculation shows that $\star_\mathrm b$ satisfies the identity $\star_\mathrm b\star_\mathrm b=(-1)^{p(s-p)}\text{id}.$
		For $\alpha,\beta\in A^p(M/\mc F)$, we define a Hermitian product by
		\begin{equation}\label{homologically inner product }
			\langle\alpha,\beta\rangle=\int_M\alpha\wedge{\star_\mathrm b\beta}\wedge\chi.
		\end{equation}
		Then it is easy to see that the adjoint operator $\delta$ of $d$ with respect to (\ref{homologically inner product })
		satisfies
		$$\delta={(-1)}^{s(p-1)-1}\star_\mathrm b d\star_\mathrm b.$$
		The construction can be repeated for complex valued basic forms on transversely Hermitian foliations, and similarly, one can obtain the 	analogous clean inequalities for $\partial$, $\bp$ and their adjoint operators.
	\end{remark}

	\begin{remark} \label{transformation} Definition \ref{basic differential operator} and Construction \ref{CONSTRUCTION} are the special cases of the results in \cite{EKA90}. Actually,  for a   compact manifold  endowed with a Riemannian foliation $\mc F$,  El Kacimi Alaoui considered a differential operator $D$ of order $l$ on $C^{\infty}(E / \mc F)$ and can define the scalar product on it, where $E$ is a Hermitian $\mc F$-bundle and $C^{\infty}(E / \mc F)$ denotes the space of basic sections of $E$ (see these concepts in Subsection \ref{subsection transversely }).
		
		El Kacimi Alaoui showed that  if a basic differential operator $D$ is (strongly) transversely elliptic, then so is its adjoint operator. Furthermore, a useful fact says that  the study of a (strongly) transversely elliptic basic differential operator is equivalent to that of  a  (strongly) elliptic differential operator in the usual sense of the same order on the basic manifold $W$ acting on the useful bundle of the Hermitian $\mc F$-bundle $E$ under the action of a compact Lie group, cf. \cite[Propositions 2.7.7, 2.7.8]{EKA90}.
	\end{remark}

	\subsection{Basic Bott--Chern cohomology and mild $\p\bp$-property for foliations}\label{subsection mild}
	We start this subsection by stating some results mainly from \cite{EKA90,EKAG97,Ra17} to be used later, and then extend several notions, lemmata  as aforementioned to the foliated case.
	
	One can define the  \textit{basic Bott--Chern cohomology} of $\mc F$:
	$$H^{\bullet,\bullet}_{\mathrm {BC}}(M/\mc F):=\frac{\ker \p\cap \ker\db}{\im\ \p\db}$$
	with the help of the basic Dolbeault double complex, see (\ref{partial and barpartial}). Also, by virtue of Construction \ref{CONSTRUCTION} and Remark \ref{ADJOINT} we can define the \textit{basic Bott--Chern Laplacian}\beq\label{basic-bc-Lap}
	\square^\mathrm b_{{\mathrm {BC}}}:=\p\db\db^*\p^*+\db^*\p^*\p\db+\db^*\p\p^*\db+\p^*\db\db^*\p+\db^*\db+\p^*\p,
	\eeq and $\G^\mathrm b_{\mathrm {BC}}$ is the associated \textit{basic Green's operator}.\footnote{The basic Aeppli case can be defined similar.}  We then have the \textit{basic Hodge decomposition} of $\square^\mathrm b_{\mathrm {BC}}$ on $(M,\mc F)$:
	\begin{equation}\label{bc-hod-foliation} 	A^{p,q}(M/\mc F)=\ker\square^\mathrm b_{\mathrm {BC}}\oplus \textrm{im}~(\p\db)
		\oplus(\textrm{im}~\p^*+\textrm{im}~\db^*),
	\end{equation}
	whose three parts are orthogonal to each other with respect to the
	$L^2$-scalar product (existing on the completion of $A^{p,q}(M/\mc F)$) defined by transversely Hermitian metric, combined with the equality
	$$\label{identity foliation}
	\1=\mathbb{H}^\mathrm b_{\mathrm {BC}}+\square_{\mathrm {BC}}^\mathrm b\G_{\mathrm {BC}}^\mathrm b=\mathbb{H}^\mathrm b_{\mathrm {BC}}+\G^\mathrm b_{\mathrm {BC}}\square^\mathrm b_{\mathrm {BC}},
	$$
	where $\mathbb{H}^\mathrm b_{\mathrm {BC}}$ is the \textit{basic harmonic projection operator}.
	
	\begin{remark}\label{decompppppp}
		In the case when $\mathcal F$ is transversely Hermitian, El Kacimi Alaoui established a Hodge--Kodaira decomposition for the  basic Dolbeault complex, so our (\ref{bc-hod-foliation}) is an application of \cite[Theorem 3.3.3]{EKA90}.
	\end{remark}

	We have the following important properties.
	\begin{proposition}[{\cite[Proposition  6.2]{EKAG97}, \cite[Proposition 3.1]{Ra17}}]
		
		$\square ^\mathrm b_{\bp}$, $\square^\mathrm b_{\mathrm{A}}$ and $\square^\mathrm b_{\mathrm {BC}}$ are all self-adjoint and transversely elliptic.
	\end{proposition}
	\begin{example}
		According to (\ref{!canonical isomophism 2}) and Remark \ref{transformation}, we have the following commutative diagram:
		
		\begin{tikzcd}
			&&&&A^{p,q}(M/\mc F) \ar[r,"\square^\mathrm b_{\mathrm {BC}}"]  \ar[d,"\mathscr {B}"] &A^{p,q}(M/\mc F)\ar[d,"\mathscr {B}"] \\
			&&&&C^{\infty}_{SO(s)}(\widetilde{N_{p,q}})\ar[r,"\widetilde{\square^\mathrm b_{\mathrm {BC}}}"]&C^{\infty}_{SO(s)}(\widetilde{N_{p,q}}),
		\end{tikzcd}
		
		where $\widetilde{\square^\mathrm b_{\mathrm {BC}}}$ is an ordinary elliptic operator induced by the basic Bott--Chern Laplacian.
	\end{example}

	Recall that a foliation satisfies the \emph{$\p\bp$-property} if  for pure-type basic forms,
	\begin{equation}\label{foliation ddbar}
		\ker \partial \cap\textrm{im}\ \bp=\ker \bp \cap\textrm{im}\  \p =\textrm{im}\ \p\bp.
	\end{equation}
	This property is thoroughly studied in the foliated case in \cite{Ra17,Ra21}. Hence, we can  extend the `weak' $\partial\bp$-properties as mentioned before to the  foliation version  without any difficulties.  For example:
	
	\begin{definition}\label{MILD}
		The foliation satisfies the \textit{$(p,q)$-th mild $\p\b{\p}$-property} or \textit{$\mathbb B^{p,q}$} if for any  $(p-1,q)$-basic form $\xi$ with $\p\b{\p}\xi=0$ on $(M,\mc F)$, there exists a $(p-1,q-1)$-basic form $\theta$ such that $\p\b{\p} \theta=\p\xi$.
	\end{definition}

By analogy with Lemma \ref{ddbar-eq}, the $\partial \bar{\partial}$-equation on a transversely Hermitian foliation admits a minimal $L^2$-norm basic solution given explicitly via the basic Green operators. Using this, and inspired by \cite[Observation 2.11]{RWZ21}, we have:

\begin{lemma}\label{ss}
Suppose $\zeta$ and $\xi$ are $(p+1,q-1)$- and $(q+1,p-1)$-basic forms on a transversely Hermitian foliation, respectively, both $\partial \bar{\partial}$-closed.
	Then the system
	\begin{equation}\label{conjugatequ-1}
		\begin{cases}
			\partial x = \bar{\partial} \zeta, \\
			\bar{\partial} x = \partial \overline{\xi},
		\end{cases}
	\end{equation}
	admits a canonical $(p,q)$-basic form solution given by
	\[
	x = \bar{\partial}(\partial \bar{\partial})^* \mathbb{G}^\mathrm{b}_{\mathrm{BC}} \bar{\partial} \zeta - \partial(\partial \bar{\partial})^* \mathbb{G}^\mathrm{b}_{\mathrm{BC}} \partial \overline{\xi}.
	\]
\end{lemma}

	\subsection{$\partial\bp$-property for foliations without homologically orientability}\label{nonHO}
	L. A. Cordero--R. A. Wolak \cite[Lemma 1 in Section 4]{CW91} showed that  a homologically orientable transversely K\"ahler foliation on a compact manifold  satisfies the $\p\bp$-property \eqref{foliation ddbar},  cf. also \cite[Theorem 7.1]{Ra17}. In this subsection, we aim to show that the homologically orientability assumption is indeed unnecessary. This result may be well-known to experts; nevertheless, we provide a proof here for the convenience of the readers.   This observation will help us  remove the $(1,2)$-th mild $\p\bp$-property assumption  in Theorem \ref{FOLiation Pkahler} when $p=1$ (which is Corollary \ref{1FOLiation 1kahler}), see also Remark \ref{Non} (\ref{BBB}).

	We first establish the following transverse K\"ahler identities, which eliminate the need for the homologically orientability assumption in \cite[Lemma 3.4.4, Proposition 3.4.5]{EKA90}.
	\begin{lemma}\label{TranKahiden}
		Let  $\mc F$ be a transversely K\"ahler foliation (which is not necessarily homologically orientable) with a transverse K\"ahler form $\omega$ on a compact manifold. Then following the notations in Remark \ref{ADJOINT}, the  identities below hold true:
		\begin{enumerate}
			\item \label{1}$[\bp,L]=[\partial, L]=0\text{\,\, and \,\,}[\bp^*,\Lambda]=[\p^*,\Lambda]=0.$
			\item \label{2}$[\bp^*,L]=\sqrt{-1}\p,\,\, [\p^*,L]=-\sqrt{-1}\,\bp\,\,\text{ and }\,\,
			[\Lambda,\bp]=-\sqrt{-1}\p^*,\,\,[\Lambda,\partial]=\sqrt{-1}\,\bp^*.$
			\item \label{3}$\partial\bp^*+\bp^*\p=0.$
			\item \label{4} $\square^\mathrm{b}_\p=\square^\mathrm{b}_{\bp}=\frac{1}{2}\square^{\mathrm b}_d \,\,$ and $\,\,\square_d^\mathrm{b}$ commutes with $\star,\p,\bp,\p^*,{\bp}^*, L,$ and $\Lambda$.
		\end{enumerate}
	\end{lemma}
	\begin{proof}
		Keeping with the same spirit as in the classical setting (cf. e.g. \cite[p. 111]{GH78} or \cite[Proposition 3.1.12]{Hb05}), we will only give  a sketch here.
		
		Let us first prove (\ref{1}). The first assertion in (\ref{1}) holds thanks to the $d$-closedness of the transverse K\"ahler form $\omega$.  The second assertion in (\ref{1}) follows from the first one and (\ref{2id}): for a complex valued $p$-basic form $\beta$ we have
		$$	\begin{aligned}
			[{\bp}^*,\Lambda](\beta) &= (-1)^{s(p+1)+1} {\star} (\bp+\zeta_2\wedge)\star\star^{-1}L\star\beta - \star^{-1}L\star(-1)^{s(p+1)+1} {\star} (\bp+\zeta_2\wedge)\star\beta \\
			&= (-1)^{s(p+1)+1}\star([\bp,L]+[\zeta_2\wedge,L])\star\beta= 0\\
		\end{aligned}$$
	and	similarly for $[\p^*,\Lambda]=0$.
		
		(\ref{2}) comes from (\ref{1}) plus the transverse Lefschetz decomposition theorem for basic forms (the proof is almost the same as the one on an Hermitian manifold). That is, let $\mc F$ be a transversely Hermitian foliation, then the following decomposition holds:
		$$A_\mb C^k(M/\mc F)=\bigoplus_{i\geq 0} L^i(P^{k-2i}(M/\mc F)) $$
		where $P^{k-2i}(M/\mc F)=\ker\, (\Lambda: A^{k-2i}_\mb C(M/\mc F)\rightarrow A^{k-2i-2}_\mb C(M/\mc F))$ and an element $\alpha\in P^{k-2i}(M/\mc F)$ is called \textit{primitive}.
		
		One can derive (\ref{3}) and (\ref{4}) via (\ref{1}) and (\ref{2}) and the proofs are trivial.
		For instance,
		$$\begin{aligned}
			\sqrt{-1}(\partial\bp^*+\bp^*\p)&=\p[\Lambda,\p]+[\Lambda,\p]\p\\
			&=\p\Lambda\p-\p\Lambda\p=0.
		\end{aligned}
		$$
		
		The proof is thus completed.
	\end{proof}
	
	With Lemma \ref{TranKahiden} and the decomposition of the basic Dolbeault cohomology, which also doesn't require the homologically orientability assumption (see \cite[Theorem 3.3.3]{EKA90} and Remark \ref{decompppppp}), in hand, we now can achieve our goal in this subsection. The idea of the proof is the same as that of  \cite[Corollary 3.2.10]{Hb05} for example so we will omit the proof.
	\begin{theorem}\label{ddbarnonho}
		With the same settings as in Lemma \ref{TranKahiden}, $\mc F$ satisfies the $\p\bp$-property.
	\end{theorem}

	\begin{remark}
		Similarly, one can show that transverse K\"ahler foliations on compact manifolds, even without the homologically orientability assumption, can possess some  properties as K\"ahler structures on compact manifolds, including the Hodge decomposition, Hard Lefschetz decomposition, etc. However, it's important to note that the assumption in question is still required for certain duality-type theorems due to the existence of the forms $\zeta_i,(i=1,2)$ (\ref{adjoint-pb}). Examples of such theorems include the transverse Serre duality and the duality theorem for basic Bott--Chern and Aeppli cohomology (cf. \cite[Corollary 3.1]{Ra17}).
	\end{remark}
	
	\begin{remark}
		Noteworthy to mention that the converse of Theorem \ref{ddbarnonho} isn't always true. Cordero--Wolak \cite[Section 2.3]{CW90} constructed a transversely symplectic but not transversely K\"ahler foliation on a compact non-complex nilmanifold. Recently, Ra\'zny \cite[6.3]{Ra17} has shown that this foliation satisfies the $\p\bp$-property.
	\end{remark}
	
	\subsection{Deformation theory of foliations}\label{subsection foliation defor}
	In this subsection, we will give some preliminaries on the deformation theory of foliations following \cite[Section 4]{EKAG97}.

	We are going to recall the definition of a deformation of a foliation in a general way. Let $M$ be a smooth manifold of dimension $m$. For each $x \in M$, let $G_x(M, s)$ be the Grassmannian manifold of $s$-planes in $T_{x} M$. Then
	$$
	\mathcal{G}(M,s):=\bigcup_{x \in M} G_x(M,s)
	$$
	can be given a structure of a differentiable manifold such that the canonical projection $(x, \tau) \in \mathcal{G}(M,s) \rightarrow x \in M$ is a locally trivial fibration, whose fiber is the Grassmannian
	$G(m,s)$ of $s$-planes in the space $\mathbb{R}^{m}$. Then a subbundle of rank $s$ of $T M$ is just a section of the bundle $\mathcal{G}(M,s) \rightarrow M$. Denote by $C^{\infty}(\mathcal{G}(M,s))$ the space of sections of this bundle.
	
Let $\tau \in C^{\infty}(\mathcal{G}(M,s))$. By Frobenius theorem, $\tau$ is tangent to the foliation if and only if, for any pair $(U, V)$ of (global) sections of $\tau$, the Lie bracket $[U, V]$ is also a section of $\tau$. Let $\left(X_{1}, \ldots, X_{s}\right)$ be a local basis of $\tau$. Then
	$$
	U=\sum_{i=1}^{s} a^{i} X_{i} \quad \text { and } \quad V=\sum_{j=1}^{s} b^{j} X_{j}.
	$$
	So the bracket $[U, V]$ can be expressed as
	$$
	[U, V]=\sum_{i, j=1}^{s}\left\{a^{i} b^{j}\left[X_{i}, X_{j}\right]+\left(a^{i} X_{i}\left(b^{j}\right) X_{j}-b^{j} X_{j}\left(a^{i}\right) X_{i}\right)\right\}.
	$$
	So the value of $[U, V]$ in $v \tau=T M / \tau$ at a point $x \in M$ depends only on the value of $U$ and $V$ at $x$. Hence, $Q_{\tau}(U, V)=\pi \left([U,V]\right)$ is a skew-symmetric bilinear map $Q_{\tau}: \tau \times \tau \rightarrow \nu \tau$ where $\pi: T M \rightarrow \nu \tau$ is the canonical projection. In other words, $Q_{\tau}$ is a global section of the vector bundle $\Lambda^{2}(\tau, v \tau)$ of skew-symmetric bilinear forms on the bundle $\tau$. 
    
The integrability condition of $\tau$ is equivalent to $Q_{\tau}$ identically equals to $0$. So we get a map $$Q: C^{\infty}(\mathcal{G}(M,s)) \longrightarrow \mathscr Z,$$ where $\mathscr Z$ is a fiber bundle over $\mathcal{G}(M,s)$ whose fiber over a point $\sigma \in \mathcal{G}(M,s)$ is the infinite-dimensional space $\Omega^{2}(\sigma, v \sigma)$ of global sections of the bundle $\Lambda^{2}(\tau, v \tau)$. The space $$\mathrm{Fol}(M, s)$$ of codimension $s$ foliations on $M$ is exactly the set $\{Q=0\} .$ It will be equipped with the $C^{\infty}$-topology induced by the topology of the Fr\'echet manifold $C^{\infty}(\mathcal{G}(M,s)).$  Let $\mathcal{D}(M)$ be the diffeomorphism group of $M$. Then $\mathcal{D}(M)$ acts on $C^{\infty}(\mathcal{G}(M,s))$ and the action preserves $\mathrm{Fol}(M, s)$. We denote by $\mathcal O_\mathcal F$ the orbit of $\mathcal F\in\mathrm{Fol}(M,s)$.
	
	\begin{definition}
		A \textit{deformation} of $\mathcal F$ parametrized  by an open neighborhood $U$ of 0 in $\mathbb{R}^{d}$ is a continuous map $\rho: U \rightarrow \mathrm{Fol}(M,s), t\mapsto \mathcal F_t$
		such that $\rho(0)=\mc F_0=\mc F$.
	\end{definition}
	\begin{definition}
		A deformation $\rho: U \rightarrow \mathrm{Fol}(M,s), t\mapsto \mathcal F_t$ of $\mc F$ parametrized  by an open neighborhood $U$ of 0 in $\mathbb{R}^{d}$ is called \textit{with fixed differentiable type} if for all $t\in U,\,\rho(t)\in\mathcal O_\mc F$, i.e., there exists a differentiable  family $\left\{h_t\right\}_{t\in U}$ of diffeomorphisms on $M$ such that $h_t^*\mathcal F_t=\mc F$.
	\end{definition}
	
	One has the following nice properties when the initial foliation is transversely holomorphic.
	\begin{proposition}[{\cite[Proposition 4.3]{EKAG97}}]
		Let $\mathcal F_t$ be the deformation of an  $r$-complex codimensional $(s=2r)$ transversely holomorphic foliation $\mathcal F$ with fixed differentiable type.
		\begin{enumerate}[$(i)$]
			\item Suppose that $\mc F$ is Riemannian. Then for any $t\in U$, $\mc F_t$ is also Riemannian.
			\item The family $\left\{\mc F_t\right\}_{t\in U}$ is  lifted to $M^\#$ in a  new family of TP foliations $\left\{\mc F_t^\#\right\}_{t\in U}$ invariant under $SO(s)$ with fixed differentiable type.
			\item Assume that the foliations $\mc F_t$ are transversely holomorphic and $\mc F_0=\mc F$ possesses a transversely Hermitian metric $\Omega_{0}.$ Then $\left\{\mc F_t\right\}_{t\in U}$  are provided with transversely Hermitian metrics $\Omega_t$ varying	differentiably depending on t.
		\end{enumerate}
	\end{proposition}
	\subsection{Proof of Theorem \ref{FOLiation Pkahler}}\label{Foliation proofproof}
	We split the proof of  Theorem \ref{FOLiation Pkahler} in four steps.
	\setcounter{step}{0}
	\begin{step}\label{foliation step 1}
		{\rm\textbf{Integrable Beltrami differentials, Kuranishi family: foliated version.}}\end{step}
	
	We first briefly introduce the main results obtained in \cite{EKAN89}, which is of great significance in our proof.
	
	With the same setting as in  Theorem \ref{FOLiation Pkahler}, El Kacimi Alaoui--M. Nicolau proved the existence of \emph{versal space} (also called \emph{Kuranishi space}) for deformations mimicking the construction of the Kuranishi space presented in \cite{Ku64}, and the universal property of the versal space was strengthened in \cite{Gb92}. As a result, one can get a family, denoted by $\phi(t)\in A^{0,1}(M/\mc F,\mc N\mc F^{1,0})$ for the Kuranishi family, satisfying the equation
	$$
	\bp\phi(t)=\frac{1}{2}\left[\phi(t),\phi(t)\right]
	$$
	on some analytic set and determining the transversely holomorphic structure on $(M,\mc F_t)$.
	Hence, this construction and $\phi(t)$ will play the same roles as the Kuranishi's completeness theorem (introduced in Section \ref{subsection-Kuranishi family})  and the  integrable Beltrami differentials $\v(t)$, respectively.  
    
Let $z^1,\cdots,z^r$ and $\zeta^1,\cdots,\zeta^r$ be the transverse holomorphic coordinates on $\mc F_0$ and $\mc F_t$, respectively. Then  $\phi(t)$ can be defined by
	
	\begin{equation}\label{kappa local}
		\phi(t):=\lk \lk \frac{\p \zeta}{\p z} \rk^{\!\!\!-1} \lk \frac{\p \zeta}{\p \bar{z}} \rk
		\rk^i_{\overline{j}}d\bar z^j\otimes\frac{\p}{\p z^i}.
	\end{equation}
	It will be often written as $\phi$ briefly.
	Also, we can define the extension map $$\label{RZ foliation}
	e^{\iota_{\phi}|\iota_{\overline{\phi}}}:
	A^{p,q}(M/\mc F_0)\longrightarrow A^{p,q}(M/\mc F_t),
	$$ analogous to (\ref{explicit}) which can preserve all $(p,q)$-basic forms.

	\begin{step}\label{foliation step 2}
		{\rm\textbf{Transverse  positive $(p,p)$-basic forms: deformation openness.}}\end{step}
	
	To prove Theorem \ref{FOLiation Pkahler}, we  need the following key observation, which will be the  deformation openness property of the transverse  positivity for $(p,p)$-basic forms. Here we follow the idea as that of \cite[the second proof of Theorem 4.9]{RWZ21}. 
	Note that to make the proposition holds, the new extension map $e^{\iota_{\phi}|\iota_{\overline{\phi}}}$ is a key player. We will use its two properties in the following  proof:  one is that it depends smoothly on $t$; the other one is that  it can send each kind of  (strictly) positive  $(p,p)$-basic forms on $(M,\mathcal F_0)$ bijectively onto the corresponding one on $(M,\mc F_t)$, which can be showed by an argument similar to the one in \cite[Proposition 4.11]{RWZ19}.
	\begin{proposition}
		Let $\left\{\mc F_t\right\}_{t\in U}$ be a smooth family of transversely Hermitian structures on a compact  foliated manifold $(M,\mc F)$ with fixed differentiable type, and $\Omega_t$ a family of real $(p,p)$-basic forms, depending smoothly on $t$.  Assume that $\Omega_0$ is a  transverse positive $(p,p)$-basic form on $(M,\mc F_0)$. Then $\Omega_t$ is also  transverse positive on $(M,\mc F_t)$ for sufficiently small $t$.
	\end{proposition}
	\begin{proof}
		Let $\Omega_0$ be a transverse positive $(p,p)$-basic form on $(M,\mc F_0)$ and $\Omega_t$ its real extension on $(M,\mc F_t)$.
		It is sufficient to show that: at any given point $x\in M$, there exists a uniform small constant $\epsilon$ such that for any $t\in U_\epsilon$, where $U_{\epsilon}=\{t \in \mathbb{R}^d\ \big |\ |t|<\epsilon \}$,  and any nonzero decomposable  $\tau\in\bigwedge^{q,0}(\mc N^*\mc F_0)_x$ with $p+q=r$, we have
		\begin{equation}\label{pkahler foliation transverse}
			\Omega_{t}(x) \wedge \sigma_{q}e^{{\iota_{\phi}}}(\tau) \wedge e^{\iota_{\bar\phi}} (\bar\tau)>0.
		\end{equation}

		Indeed, let $\omega_0$ be a transversely Hermitian metric on $(M,\mc F_0)$. For any fixed point $x \in M$, we define a continuous function $f_{x}(t,[\tau])$ on $U_{\epsilon} \times Y_{x}$ by
		\begin{equation}\label{definition f}
			f_{x}(t,[\tau]):=\frac{	\Omega_{t}(x) \wedge \sigma_{q} e^{\iota_\phi}(\tau) \wedge e^{\iota_{\bar\phi}} (\bar\tau)}{|\tau|_{\omega_0(x)}^{2} \cdot \omega_0(x)^{r}},
		\end{equation}
		where $Y_{x}=\left.\rho(G(q, r))\right|_{x} \subset \mathbb P\left(\bigwedge^{q,0}(\mc N^*\mc F_0)_x\right)$ is compact.  By assumption, $$f_{x}(0,[\tau])=\frac{	\Omega_{0}(x) \wedge \sigma_{q} \tau\wedge \overline\tau }{|\tau|_{\omega_0(x)}^{2} \cdot \omega_0(x)^{r}}>0.
		$$
		Hence, by the continuity of $f_{x}(0,[\tau])$ on $t$ and $[\tau]$ plus the compactness of $Y_x$, there exists a constant $\epsilon_x>0$ depending only on $x$, such that
		$$\label{12}
		f(x,\overline{U}_{\epsilon_x/2},Y_x):=f_{x}(\overline{U}_{\epsilon_x/2},Y_x)>0.
		$$
		
		Let   $\{T_j\}_{j\in J}$ be trivializing adapted covering of $M$,
		and choose any $x\in T_j$ for some $j$. So we can identify $Y_x$ and $Y_y$ for any point $y\subset T_j$, and $f_y$ is defined on $Y_x$.
		Similarly,   by the continuity of $f$ on $\overline{U}_{\epsilon_x/2}\times Y_x$, there exists an open neighborhood $R_x\subset T_j$ of $x$ such that
		$$
		f(R_x, \overline{U}_{\epsilon_x/2},Y_x)>0.
		$$

		Since $M$ is compact, one gets a finite open covering $R_{x_i}$, $i=1,\cdots, m$, $M=\mathop\bigcup_{i=1}^m R_{x_i}$. Set
		$$\varepsilon:=\min_{1\leq i\leq m}\epsilon_{x_i}/2>0.$$
		Then
		$$
		f(x,\overline{U}_{\varepsilon},Y_x)=f_x(\overline{U}_{\varepsilon},Y_x)>0
		$$
		for any $x\in M$.

		Therefore,  by virtue of  the definition \eqref{definition f},
		for any $|t|\leq \varepsilon$, we have
		$$
		\Omega_{t}(x) \wedge \sigma_{q} e^{\iota_\phi}(\tau) \wedge e^{\iota_{\bar\phi}} (\bar\tau)=f_{x}(t,[\tau])|\tau|^2_{\omega_0(x)}\cdot \omega_0(x)^r>0
		$$
		for any nonzero decomposable $\tau\in \bigwedge^{q,0}(\mc N^*\mc F_0)_x$. This is the desired inequality (\ref{pkahler foliation transverse}).
		
	\end{proof}
	
	With these preparations, to prove Theorem \ref{FOLiation Pkahler}, it suffices to prove the special case $p=q$ of the following theorem.
	\begin{theorem}\label{FOliation Mild}
		Let $\left\{\mc F_t\right\}_{t\in \Delta_{\epsilon}}$ be a smooth family of transversely Hermitian structures on a compact  foliated manifold $(M,\mc F)$ with fixed differentiable type, parametrized by a small complex disc. If   $\mathcal{F}_{0}$  satisfies the $(p,q+1)$-th and $(q,p+1)$-th mild $\p\b{\p}$-properties, then there is a $d$-closed $(p,q)$-basic form $\Omega(t)$ on $(M,\mc F_t)$ depending smoothly on $t$ with $\Omega(0)=\Omega_0$ for any  $d$-closed $\Omega_0\in A^{p,q}(M/\mc F)$.
	\end{theorem}

	Next we will sketch the proof of Theorem \ref{FOliation Mild} by adopting the power series idea used in \cite[Section 4]{RWZ19}. Here we just indicate the main procedures emphasizing how we adapt the method to our new setting.
	
	\begin{step}\label{foliation step 3}
		{\rm\textbf{Obstruction equation and construction of power series.}}\end{step}
	
	As both $e^{ {\iota}_{(\1-\b{ \phi} \phi)^{-1}\b{ \phi}}}$ and
	$e^{ {\iota}_{ \phi}}$ are invertible operators when $t$ is
	sufficiently small by (\ref{kappa local}), it follows that for any $\Omega \in
	A^{p,q}(M/\mc F)$,
	\begin{align}\label{2.1.1}
		e^{ {\iota}_{ \phi}| {\iota}_{\b{ \phi}}}(\Omega)=e^{ {\iota}_{ \phi}}\circ e^{ {\iota}_{(\1-\b{ \phi} \phi)^{-1}\b{ \phi}}}\circ e^{- {\iota}_{(\1-\b{ \phi} \phi)^{-1}\b{ \phi}}}\circ e^{- {\iota}_{ \phi}}\circ e^{ {\iota}_{ \phi}| {\iota}_{\b{ \phi}}}(\Omega).
	\end{align}
	Here we follow the notations:
	$\overline{ \phi}  \phi =  \phi \lc \overline{ \phi}$, $\1$ is the identity operator defined as:
	\[
	\mathbf{1} = \frac{1}{p+q} \Bigl( \sum_1^r dz^i \otimes \frac{\partial}{\partial z^i} + \sum_1^r d\bar{z}^i \otimes \frac{\partial}{\partial \bar{z}^i} \Bigr)
	\]
	when it acts on $(p,q)$-basic forms, and likewise for others.
	Set
	\begin{align}\label{2.1.2}
		\widetilde{\Omega}=e^{- {\iota}_{(\1-\b{ \phi} \phi)^{-1}\b{ \phi}}}\circ e^{- {\iota}_{ \phi}}\circ e^{ {\iota}_{ \phi}| {\iota}_{\b{ \phi}}}(\Omega)=(\1-\b{ \phi} \phi)\Finv\Omega,
	\end{align}
	where $\Omega$ and $\widetilde{\Omega}$ are apparently one-to-one
	correspondence. Here the notation $\Finv$  denotes the
	simultaneous contraction. By observing  the local expression  for any element in $ A^{0,1}(M/\mc F,\mc N\mc F^{1,0})$ (see (\ref{foliation explicit})), we have:
	
	\begin{proposition}\label{foliation main1}
		Let $\phi\in A^{0,1}(M/\mc F,\mc N\mc F^{1,0})$ on $(M,\mc F)$. Then on the space $A^{\bullet,\bullet}(M/\mc F)$,
		$$ e^{- \iota_{\phi}}\circ d\circ e^{ \iota_\phi}=d-\sL_\phi^{1,0}+\iota_{\bp\phi-\frac{1}{2}[\phi,\phi]},$$
		where $\sL_\phi^{1,0}:=\iota_\phi\p-\p \iota_\phi$ is the Lie derivative. \end{proposition}
	Together with
	\eqref{2.1.1} and \eqref{2.1.2}, Proposition \ref{foliation main1} implies that
\begin{align*}
	\begin{split}
		d\bigl(e^{\iota_{\phi} | \iota_{\bar{\phi}}}(\Omega)\bigr)
		&= d \circ e^{\iota_{\phi}} \circ e^{\iota_{(\mathbf{1} - \bar{\phi} \phi)^{-1} \bar{\phi}}}(\widetilde{\Omega}) \\
		&= e^{\iota_{\phi}} \Bigl( \bar{\partial}_{\phi} A_0 + \sum_{k=0}^{+\infty} \bigl( \partial A_k + \bar{\partial}_{\phi} A_{k+1} \bigr) \Bigr).
	\end{split}
\end{align*}
	where $$A_k:=\frac{{\iota}^k_{(\1-\b{ \phi} \phi)^{-1}\b{ \phi}}}{k!}(\widetilde{\Omega})$$  is a $(p+k,q-k)$-basic form and $$\b{\p}_{ \phi}:=\b{\p}+[\p, {\iota}_ \phi].$$
	Thus, $d(e^{ {\iota}_{ \phi}| {\iota}_{\b{ \phi}}}(\Omega))=0$ amounts to
	\beq\label{FOL obstruction}
	\b{\p}_{ \phi}A_0=0,\quad \p A_k+\b{\p}_{ \phi}A_{k+1}=0,\quad k=0,1,2,\ldots.
	\eeq
	Following  the proof given in \cite[Propositions 4.2, 4.5]{RWZ21}, one sees that (\ref{FOL obstruction})  can be reduced to the following one with only two equations:
	\begin{equation}\label{1.6}
		\begin{cases}
			\sum_{k=0}^{\infty}\Bigl(\b{\p}\circ\frac{ {\iota}^{k-1}_{ \phi}}{(k-1)!}+\p\circ\frac{ {\iota}^{k}_{ \phi}}{k!}\Bigr)\circ\frac{ {\iota}^{k}_{(\1-\b{ \phi} \phi)^{-1}\b{ \phi}}}{k!}(\widetilde{\Omega})=0,\\
			\sum_{k=0}^{\infty}\Bigl(\b{\p}\circ\frac{ {\iota}^{k-1}_{ \phi}}{(k-1)!}+\p\circ\frac{ {\iota}^{k}_{ \phi}}{k!}\Bigr)\circ\frac{ {\iota}^{k-1}_{(\1-\b{ \phi} \phi)^{-1}\b{ \phi}}}{(k-1)!}(\widetilde{\Omega})=0.
		\end{cases}
	\end{equation}
	
	Rewrite (\ref{1.6}) as
	$$
	\begin{cases}
		\p\widetilde{\Omega}'=-\b{\p}\sum_{k=1}^{\infty}\frac{ {\iota}^{k-1}_{ \phi}}{(k-1)!}\circ\frac{ {\iota}^{k}_{(\1-\b{ \phi} \phi)^{-1}\b{ \phi}}}{k!}(\widetilde{\Omega}),\\
		\b{\p}\widetilde{\Omega}'=-\p \sum_{k=0}^{\infty} \frac{ {\iota}^{k+1}_{ \phi}}{(k+1)!}\circ\frac{ {\iota}^{k}_{(\1-\b{ \phi} \phi)^{-1}\b{ \phi}}}{k!}(\widetilde{\Omega}),
	\end{cases}
	$$
	where
	$$\widetilde{\Omega}':=\widetilde{\Omega}+\sum_{k=1}^{\infty}\frac{ {\iota}^{k}_{ \phi}}{k!}\circ\frac{ {\iota}^{k}_{(\1-\b{ \phi} \phi)^{-1}\b{ \phi}}}{k!}(\widetilde{\Omega}).$$
	By Lemma \ref{ss}, we need to prove
	\begin{align}\label{2.1}
		\begin{cases}
			\p\b{\p}\sum_{k=1}^{\infty}\frac{ {\iota}^{k-1}_{ \phi}}{(k-1)!}\circ\frac{ {\iota}^{k}_{(\1-\b{ \phi} \phi)^{-1}\b{ \phi}}}{k!}(\widetilde{\Omega})
			=0,\\
			\p\b{\p}\sum_{k=0}^{\infty} \frac{ {\iota}^{k+1}_{ \phi}}{(k+1)!}\circ\frac{ {\iota}^{k}_{(\1-\b{ \phi} \phi)^{-1}\b{ \phi}}}{k!}(\widetilde{\Omega})=0,
		\end{cases}
	\end{align}
	at the $(N+1)$-th order  if it has not been exceeding for the orders $N$.  Analogous to \cite[p. 471]{RWZ19}, we write the power series $\alpha(t)$
	$$\alpha(t)=\sum_{k=0}^{\infty}\sum_{i+j=k}\alpha_{i,j}t^i\b{t}^j,$$
	of (bundle-valued)  $(p,q)$-basic forms as
	\[ \begin{cases}
		\alpha(t) = \sum^{\infty}_{k=0} \alpha_k, \\[4pt]
		\alpha_k = \sum_{i+j=k} \alpha_{i,j}t^i \overline{t}^j,
	\end{cases} \]
	where $\alpha_k= \sum_{i+j=k} \alpha_{i,j}t^i \overline{t}^j$ is the $k$-order homogeneous part in the expansion
	of $\alpha(t)$ and all $\alpha_{i,j}$ are smooth (bundle-valued)
	$(p,q)$-basic forms on $\mc F_0$ with $\alpha(0)=\alpha_{0,0}$.
	
Indeed, one can prove \eqref{2.1} by repeating the arguments in \cite[p. 476]{RWZ19}. Using Lemma \ref{ss}, we obtain a formal solution of \eqref{1.6} by induction,
	\begin{equation}\label{express-tO}
		\begin{aligned}
			\widetilde{\Omega}_l
			&=-\Big(\sum_{k=1}^{\min\{q,r-p\}}\frac{ {\iota}^{k}_{ \phi}}{k!}\circ\frac{ {\iota}^{k}_{(\1-\b{ \phi} \phi)^{-1}\b{ \phi}}}{k!}(\widetilde{\Omega})\Big)_l
			-\db(\p\db)^*\G^\mathrm b_{\mathrm {BC}}\db\Big(\sum_{k=1}^{\min\{q,r-p\}}\frac{ {\iota}^{k-1}_{ \phi}}{(k-1)!}\circ\frac{ {\iota}^{k}_{(\1-\b{ \phi} \phi)^{-1}\b{ \phi}}}{k!}(\widetilde{\Omega})\Big)_l\\
			&\quad +\p(\p\db)^*\G^\mathrm b_{\mathrm {BC}}\p\Big(\sum_{k=0}^{\min\{q,r-p\}} \frac{ {\iota}^{k+1}_{ \phi}}{(k+1)!}\circ\frac{ {\iota}^{k}_{(\1-\b{ \phi} \phi)^{-1}\b{ \phi}}}{k!}(\widetilde{\Omega})\Big)_l.
		\end{aligned}
	\end{equation}

	\begin{step}\label{foliation step 4}
		{\rm\textbf{Regularity argument.}}\end{step}
	From the induction expression \eqref{express-tO}, one obtains the formal expression of $\widetilde{\Omega}:$
	\begin{equation}\label{omexp}
		\begin{aligned}
			\widetilde{\Omega}&=-\sum_{k=1}^{\min\{q,r-p\}}\frac{ {\iota}^{k}_{ \phi}}{k!}\circ\frac{ {\iota}^{k}_{(\1-\b{ \phi} \phi)^{-1}\b{ \phi}}}{k!}(\widetilde{\Omega})
			-\db(\p\db)^*\G_{\mathrm {BC}}^\mathrm b\db\sum_{k=1}^{\min\{q,r-p\}}\frac{ {\iota}^{k-1}_{ \phi}}{(k-1)!}\circ\frac{ {\iota}^{k}_{(\1-\b{ \phi} \phi)^{-1}\b{ \phi}}}{k!}(\widetilde{\Omega})\\
			&\quad +\p(\p\db)^*\G_{\mathrm {BC}}^\mathrm b\p\sum_{k=0}^{\min\{q,r-p\}} \frac{ {\iota}^{k+1}_{ \phi}}{(k+1)!}\circ\frac{ {\iota}^{k}_{(\1-\b{ \phi} \phi)^{-1}\b{ \phi}}}{k!}(\widetilde{\Omega})
			+\Om_0.
		\end{aligned}
	\end{equation}
	Set
\[
\begin{aligned}
	F =\ & \partial(\partial\bar{\partial})^* \mathbb{G}_{\mathrm{BC}}^{\mathrm{b}} \partial \sum_{k=0}^{\min\{q,r-p\}} \frac{\iota_{\phi}^{k+1}}{(k+1)!} \circ \frac{\iota_{(\mathbf{1} - \bar{\phi} \phi)^{-1} \bar{\phi}}^k}{k!}
	- \sum_{k=1}^{\min\{q,r-p\}} \frac{\iota_{\phi}^k}{k!} \circ \frac{\iota_{(\mathbf{1} - \bar{\phi} \phi)^{-1} \bar{\phi}}^k}{k!} \\
	& - \bar{\partial}(\partial\bar{\partial})^* \mathbb{G}_{\mathrm{BC}}^{\mathrm{b}} \bar{\partial} \sum_{k=1}^{\min\{q,r-p\}} \frac{\iota_{\phi}^{k-1}}{(k-1)!} \circ \frac{\iota_{(\mathbf{1} - \bar{\phi} \phi)^{-1} \bar{\phi}}^k}{k!}.
\end{aligned}
\]
	and write
	\begin{equation}\label{global-op}
		\Om_0=(\1-F)\widetilde{\Omega}.
	\end{equation}
	We claim that $\widetilde{\Omega}(t)$ converges in H\"older norm \eqref{Holder foliation} as $t\rightarrow 0$. Here, we resort to  Construction \ref{CONSTRUCTION} and Remark \ref{transformation} to get two transversely elliptic estimates for a basic form $\psi$,
	$$
	\|\overline{\partial}^*\psi\|_{k-1, \alpha}\leq C_1\|\psi\|_{k,
		\alpha}\quad\text{and}\quad
	\|\G_{\mathrm {BC}}^\mathrm b\psi\|_{k, \alpha}\leq C_{k,\alpha}\|\psi\|_{k-4, \alpha},
	$$
	where $k>3$ and $C_{k,\alpha}$ depends on only on $k$ and $\alpha$,
	not on $\psi$. As $ \phi(t)$ converges smoothly to zero as $t\rightarrow 0$,
	one estimates  by \eqref{global-op},
	\begin{equation*}\label{estimate-omeg}
		\| \Om_0\|_{k, \alpha}\geq
		(1-\epsilon_{k,\alpha})
		\|\widetilde{\Omega} \|_{k, \alpha},
	\end{equation*}
	where $0<\epsilon_{k,\alpha}\ll 1$ is some constant depending on $k,\alpha$.

	Finally, we proceed to the regularity of $\widetilde{\Omega}(t)$ since there is possibly no uniform
	lower bound for the convergence radius obtained as above in the $C^{k,\alpha}$-norm when $k$ converges to $+\infty$. Similarly, we can transfer the world of transversely elliptic theory to the one of elliptic theory in the ordinary sense as aforementioned. So our argument also lies heavily in the elliptic estimates \cite[Appendix  8]{Ko86}, \cite{DN55} and also \cite[Subsection 3.2]{RWZ21}.
	Without loss of generality, we just consider the equation:
	$$\label{reg-eqn-1}
	\begin{aligned}
		\square_{\bp}^\mathrm b\widetilde{\Omega}&=-\square^\mathrm b_{\bp}\sum_{k=1}^{\min\{q,r-p\}}\frac{ {\iota}^{k}_{ \phi}}{k!}\circ\frac{ {\iota}^{k}_{(\1-\b{ \phi} \phi)^{-1}\b{ \phi}}}{k!}(\widetilde{\Omega})
		-\db \db^*\db(\p\db)^*\G_{\mathrm {BC}}^\mathrm b\db\sum_{k=1}^{\min\{q,r-p\}}\frac{ {\iota}^{k-1}_{ \phi}}{(k-1)!}\circ\frac{ {\iota}^{k}_{(\1-\b{ \phi} \phi)^{-1}\b{ \phi}}}{k!}(\widetilde{\Omega})\\
		&\quad +\square_{\bp}^\mathrm b\p(\p\db)^*\G_{\mathrm {BC}}^\mathrm b\p\sum_{k=0}^{\min\{q,r-p\}} \frac{ {\iota}^{k+1}_{ \phi}}{(k+1)!}\circ\frac{ {\iota}^{k}_{(\1-\b{ \phi} \phi)^{-1}\b{ \phi}}}{k!}(\widetilde{\Omega})
	\end{aligned}
	$$
	by applying the basic Dolbeault Laplacian
	$\square^\mathrm b_{\bp}=\db^*\db+\db\db^*$ to the expression formula \eqref{omexp} and omitting the lower-order term $\square^\mathrm b_{\bp} \Om_0$ in this expression.
	
	For each $l=1,2,\cdots$, choose a
	smooth function $\eta^l(t)$ with values in $[0,1]$:
	$$\label{eta-l}
	\eta^l(t)\equiv
	\begin{cases}
		1,\ \text{for $|t|\leq (\frac{1}{2}+\frac{1}{2^{l+1}})\gamma$};\\
		0,\ \text{for $|t|\geq (\frac{1}{2}+\frac{1}{2^{l}})\gamma$},
	\end{cases}
	$$
	where $\gamma$ is a positive constant to be determined.
	Inductively, by Douglis--Nirenberg's
	interior estimates \cite[Theorem 2.3 in the Appendix]{Ko86}, \cite{DN55}, for any $l=1,2,\cdots$,
	$\eta^{2l+1}\widetilde{\Omega}(t)$ is $C^{k+l, \alpha}$, where
	$\gamma$ can be chosen independent of $l$. Since $\eta^{2l+1}(t)$ is
	identically equal to $1$ on $|t|<\frac{\gamma}{2}$ which is independent
	of $l$, $\widetilde{\Omega}(t)$ is $C^{\infty}$ on $(M,\mc F_0)$ with
	$|t|<\frac{\gamma}{2}$.  Then $\widetilde{\Omega}(t)$ can be considered as a real analytic family of $(p,q)$-basic   forms in $t$
	and thus is smooth on $t$.
	
	The proof of  Theorem \ref{FOliation Mild} (and thus  Theorem \ref{FOLiation Pkahler}) is completed.

	\begin{remark}\label{Non} On the assumptions of Theorem \ref{FOLiation Pkahler}, we should notice:
		\begin{enumerate}[(a)]
			\item 	\label{AAA}	 One can conclude from our proof that the homologically orientability assumption used in \cite{EKAG97,Ra21} is not necessary in our approach. In the previous works,  they  follow the original cohomological
			argument used by Kodaira--Spencer, and  it seems that the assumption  in question on $\mc F$ is necessary for the upper semi-continuity theorem for  transversely  elliptic operators  to work (see \cite[Remark 4.6]{Ra21} for example). In our proof,   we directly construct the  desired transversely $p$-K\"ahler form.  We more faithfully follow the scalar product or H\"older norm on the space of basic forms introduced in Construction \ref{CONSTRUCTION} and allow the existence of a correction term involving the mean curvature of the foliation when defining the adjoint operators, see Remarks \ref{ADJOINT}, \ref{homologic1}.
			
			\item  \label{BBB}As shown in Subsection \ref{nonHO}, any transversely K\"ahler foliation, even if not necessarily homologically orientable, satisfies the $\partial\bar{\partial}$-property \eqref{foliation ddbar} and, therefore, obviously fulfills the mild $(1,2)$-th $\partial\bar{\partial}$-property. Thus, compared to the works of \cite{EKAG97,Ra21}, Theorem \ref{FOLiation Pkahler} requires weaker assumptions when $p=1$, specifically, it drops the homologically orientability assumption.
		\end{enumerate}
	\end{remark}
	
	\subsection{Deformation invariance of basic Hodge/Bott--Chern numbers}\label{defor}Recently, Ra\'zny  showed the rigid property of basic Hodge numbers under deformations of homologically orientable transversely Hermitian foliations on a compact manifold when $\mc F_0$ satisfies $\p\bp$-property and the foliations are fixed \cite[Corollary 5.3]{Ra21}. Motivated by this, we study the deformation invariance of basic Hodge/Bott--Chern numbers with `weak' $\p\bp$-properties in this subsection.
	
	The following theorem can be seen as the foliated version of Kodaira--Spencer's results. It seems that  we cannot drop the homologically orientability assumption as already mentioned in Remark
	\ref{Non} (\ref{AAA}).
	\begin{theorem}
		[{\cite[Theorem 5.4]{EKAG97}, \cite[Theorem 4.5]{Ra21}}]\label{ks foliation}
		Let $\left\{\mc F_t\right\}_{t\in U}$ be a smooth family of homologically orientable transversely Hermitian structures on a compact  manifold $M$ with fixed differentiable type, parametrized by an open neighborhood $U$ of $0$ in $\mathbb{R}^{d}$.  Fix any two non-negative integers $p,q$.   Let $D_t: A^{p,q}(M/\mc F_t)\rightarrow A^{p,q}(M/\mc F_t)$ be a family of transversely elliptic operators of even order. Denote $h(t):=\dim \ker(D_t)$. Then $h(t)$ is upper semi-continuous. Furthermore, if $h(t)$ is independent of $t\in U$, then the operators $\mb G_t$ and $\mb H_t$ depend differentiably on $t\in U$. Here, $\mb G_t$ and $\mb H_t$ are the associated Green's operator and harmonic projection operator, respectively.
	\end{theorem}
	
	Consequently, one can repeat the same procedures in Section \ref{section 4}, \cite[Section 3]{RZ18} and \cite[Section 5]{RWZ19} to get:
	\begin{theorem}\label{basic hodge}
		Let $\left\{\mc F_t\right\}_{t\in U}$, $M$, $U$ be as in Theorem \ref{ks foliation} and $\mb S^{p,q},\mathcal B^{p,q}$ be defined for $\mc F_0$ similarly to Definition \ref{MILD}.
		\begin{enumerate}[$(a)$]
			\item If $\mc F_0$ satisfies both
			$\mathbb{B}^{p,q+1}$ and $\mathcal{B}^{p+1,q}$ with
			the deformation invariance of $h^{p,q-1}_{\db_t}(M/\mc F_t)$ established,
			then $h^{p,q}_{\db_t}(M/\mc F_t)$ are independent of $t$.
			\item If $\mc F_0$ satisfies both $\mb S^{p,q+1}$ and  $\mb B^{p+1,q}$ with
			the deformation invariance of $h^{p,q-1}_{\db_t}(M/\mc F_t)$ established,
			then $h^{p,q}_{\db_t}(M/\mc F_t)$ are independent of $t$.
			\item   If $\mc F_0$ satisfies both
			$\mathbb{B}^{p,q+1}$ and $\mathbb{B}^{q,p+1}$ with the deformation invariance of the $(p-1,q-1)$-basic Aeppli number $h^{p-1,q-1}_{\mathrm{A}}(M/\mc F_t)$ established, then the $(p,q)$-basic Bott--Chern number are  independent of $t$.
		\end{enumerate}
	\end{theorem}
	
	Notice that  in this subsection we always assume that the foliations are fixed. We should then mention an interesting question proposed by Ra\'zny if the foliations are allowed to vary.  For the case of a  family of Sasakian manifolds which is in particular a family of homologically orientable transversely K\"ahler foliations, Ra\'zny \cite[Theorem 1]{Ra21} affirmed this question (see also an alternative proof based on Vaisman geometry, \cite[Theorem 31.25]{OV22})  which completes earlier results of C. P. Boyer--K. Galicki \cite{BG08} and of O. Goertsches--H. Nozawa--D. T\"oben \cite{GNT16} who proved the invariance for special types of deformations. Note that Ra\'zny's result boosts the study of basic Hodge numbers by showing they can help distinguish between different deformation classes of Sasakian structures, without relying on contact topology, see \cite{KP22} and the references therein for some recent progress. Very recently, Ra\'zny also gave a positive answer to this question for a new class of transversely K\"ahler foliations, cf. \cite{Ra22}.
	
	\begin{question}[{\cite[Question 1.2]{Ra21}}] \label{Razny question}
		Are the basic Hodge numbers rigid under deformations of (homologically orientable)
		transversely K\"ahler foliations on compact manifolds?
	\end{question}
	
	\appendix
	\section{The \texorpdfstring{$p$}{p}-K\"{a}hler structures} \label{p}
	
	In this appendix, we state some basic knowledge about the $p$-K\"{a}hler structures. Let $V$ be a complex $n$-dimensional vector space and  $V^{*}$ its
	dual space, i.e., the space of complex linear functionals
	over $V$. Denote the complexification of the space of the exterior $m$-vectors
	of $V^{*}$ by $\bigwedge^{m}_{\mathbb{C}} V^{*}$, which admits a
	natural direct sum decomposition
	\[ \bigwedge^{m}_{\mathbb{C}} V^{*} = \sum_{r+s=m} \bigwedge^{r,s} V^*, \]
	where $\bigwedge^{r,s} V^*$ is the complex vector space
	of $(r,s)$-forms on $V^*$.
	The case $m=1$ exactly reads \[ \bigwedge^{1}_{\mathbb{C}} V^{*} =
	V^* \bigoplus \overline{V^{*}},\] where the natural isomorphism $V^*
	\cong \bigwedge^{1,0}V^*$ is used. 
    
Let $q\in \{1, \cdots, n\}$ and $p=n-q$. Clearly, the complex dimension
	$N$ of $\bigwedge^{q,0}V^*$ equals to the combination number $C^{q}_n$. After fixing a basis $\{
	\beta_i \}_{i=1}^N$ of the complex vector space $\bigwedge^{q,0}V^*$, we can give
	the canonical Pl\"ucker embedding as in \cite[p.   209]{GH78}
	by
\[
\rho: G(q,n) \hookrightarrow \mathbb{P}(\wedge^{q,0} V^*), \quad s \mapsto [\cdots, s_i, \cdots].
\]
	Here $G(q,n)$ denotes the Grassmannian of $q$-planes in the vector
	space $V^*$ and $\mathbb{P}(\bigwedge^{q,0}V^*)$ is the
	projectivization of $\bigwedge^{q,0}V^*$. 
	A $q$-plane in $V^*$ can
	be represented by a decomposable $(q,0)$-form $s \in
	\bigwedge^{q,0}V^*$ up to a nonzero complex number, and
	$\{s_i\}_{i=1}^N$ are exactly the coordinates of $s$
	under the fixed basis $\{ \beta_i \}_{i=1}^N$.  \textit{Decomposable
		$(q,0)$-forms} are those forms in $ \bigwedge^{q,0}V^*$ that can be
	expressed as $\gamma_1 \bigwedge \cdots \bigwedge \gamma_q$ with
	$\gamma_i \in V^* \cong \bigwedge^{1,0}V^*$ for $1 \leq i \leq q$.
	The $pq$-dimensional locus $\rho(G(q,n))$ in
	$\mathbb{P}(\bigwedge^{q,0}V^*)$ characterizes the
	decomposable $(q,0)$-forms in $\mathbb{P}(\bigwedge^{q,0}V^*)$.
	
	Now we list several positivity notions (cf. \cite{HK74,Hv77,Dm12} for more details). A $(q,q)$-form $\Theta$ in
	$\bigwedge^{q,q}V^*$ is defined to be  \textit{strictly positive (resp., positive)} if
	\[ \Theta =\sigma_{q}\sum_{i,j=1}^N \Theta_{i\overline{ j}} \beta_i \wedge \b\beta_j,\]
	where $\Theta_{i\overline j}$ is a positive (resp. semi-positive) Hermitian matrix of the size $N \times N$ with $N=C_{n}^q$ under the
	basis $\{\beta_i \}_{i=1}^N$ of the complex vector space $\bigwedge^{q,0}V^*$ and $\sigma_{q}$ is defined to be the constant
	$2^{-q}(\sqrt{-1})^{q^2}$.
	According to this definition, the fundamental form of a Hermitian metric on a complex manifold is actually a strictly positive $(1,1)$-form everywhere. A $(p,p)$-form $\Gamma\in \bigwedge^{p,p}V^*$ is said to be  \textit{strictly weakly positive} (resp.\textit{ weakly positive}) if
	the volume form $$\Gamma\wedge\sigma_{q}\tau\wedge\bar{\tau}$$ is
	strictly positive (resp.  positive)  for every nonzero decomposable $(q,0)$-form $\tau$ of $V^*$, while
	a $(q,q)$-form $\Upsilon\in \bigwedge^{q,q}V^*$ is called  \textit {strongly positive} if
	$\Upsilon$ is a convex combination
	$$\Upsilon=\sum_s\gamma_s \sqrt{-1}\alpha_{s,1}\wedge\bar\alpha_{s,1}\wedge\cdots\wedge\sqrt{-1}\alpha_{s,q}\wedge\bar\alpha_{s,q},$$
	where $\alpha_{s,i}\in V^*$ and $\gamma_s\geq 0$.

As shown in \cite[Chapter III, \S\,1.A]{Dm12}, the sets of weakly positive and strongly positive forms are closed convex cones,
	and by definition, the weakly positive cone is dual to the strongly positive cone via the pairing
	$$\bigwedge^{p,p}V^*\times \bigwedge^{q,q}V^*\longrightarrow \mathbb{C}.$$
	Then all weakly positive forms are real.
	
An element $\Xi$ in $\bigwedge^{p,p}V^*$ is called  \textit{transverse}, if
	it is strictly weakly positive. There exist many various names for this
	terminology and we refer to \cite[Appendix]{AB91} for a list.
	
	These positivity notions on complex vector spaces can be extended pointwise to
	complex differential forms on a complex manifold.
	\begin{definition}[{\cite[Definition $1.11$]{AA87}}, for example]\label{pkf}\rm
		For an $n$-dimensional complex manifold $M$ and a positive integer $p\leq n$, $M$ is called a  \textit{$p$-K\"ahler manifold}
		if there exists a  \textit{$p$-K\"ahler form}, that is a $d$-closed transverse $(p,p)$-form on $M$.
	\end{definition}
	
	The readers can refer to \cite{Su76} for more related concepts (such as differential form transversal to the cone structure on a real differentiable manifold) to $p$-K\"ahler structures.

\end{document}